\documentclass[11pt,a4paper]{article}
\usepackage{amsthm,amsmath,amssymb,amsfonts,dsfont}
\usepackage{a4,a4wide}
\usepackage{graphics}
\usepackage{color}
\usepackage[T1]{fontenc}
\usepackage[colorlinks=true]{hyperref}

\newtheorem{theo}{\textbf{Theorem}}[section]

\newtheorem{lem}[theo]{\textbf{Lemma}}
\newtheorem{prop}[theo]{\textbf{Proposition}}

\newtheorem{cor}[theo]{\textbf{Corollary}}
\newtheorem{defi}[theo]{\textbf{Definition}}

\newtheorem{rem}[theo]{\textbf{Remark}}
\newtheorem{ass}{\textbf{Assumption}}

\newcommand{\rb}[1]{{
#1}}
\newcommand{\ra}[1]{{
#1}}

\newcommand{\gr}[1]{{
#1}}

\title{Macroscopic limit from a structured population model to the Kirkpatrick-Barton model}

\def\signgr{\bigskip \begin{center} {\sc Ga\"el
      Raoul\par\vspace{3mm}
      CMAP, CNRS, Ecole polytechnique, Institut Polytechnique de Paris, 91120 Palaiseau, France\par
      \par\vspace{3mm} e-mail:}
    \tt{gael.raoul@polytechnique.edu} 
    \end{center}}
\begin{document}

\author{G. Raoul}


\maketitle


\begin{abstract} 
We consider an ecology model in which the population is structured by a spatial variable and a phenotypic trait. The model combines  a parabolic operator on the spatial variable  with a kinetic operator on the trait variable. \gr{We prove the existence of solutions to that model, and show that these solutions are unique.} The kinetic operator present in the model, that represents the effect of sexual reproductions, satisfies a Tanaka-type inequality: it implies a contraction of the Wasserstein distance in the space of phenotypic traits. We combine this contraction argument with parabolic estimates controlling the spatial regularity of solutions to prove the convergence of the population size and the mean phenotypic trait to solutions of the Kirkpatrick-Barton model, which is a well-established model in evolutionary ecology. Specifically, at high reproductive rates, we provide explicit convergence estimates for the moments of solutions of the kinetic model. 
\end{abstract}


\medskip


%
%
%
%
%
%
%
%
%
%

\section{Introduction}

We are interested in a structured population model that describes the dynamics of a biological population (typically a species of trees subject to climate change). At each time $t\geq 0$ the population is structured by a phenotypic trait $y\in\mathbb R$ and a spatial variable $x\in \mathbb T^d$ (the $d\in\mathbb N^*$ dimensional torus).  
The population is then represented by a density $n=n(t,x,y)$, and the dynamics of this population is given by the \emph{Spatially structured Infinitesimal Model} (see \cite{Mirrahimi}):
\begin{align*}
&\phantom{dsqgqdfer}\partial_t n(t,x,y)=\Delta_x n(t,x,y)+\left(1+\frac A2-\frac 12(y-y_{opt}(t,x))^2-\int_{\mathbb R} n(t,x,z)\,dz\right)n(t,x,y)\nonumber\\
&\phantom{\partial_t n(t,x,y)=fsgerr}+\gamma\left(\iint_{\mathbb R^2} \Gamma_{A/2}\left(y-\frac{y_*+y_*'}2\right)\frac{n(t,x,y_*)n(t,x,y_*')}{\int_{\mathbb R} n(t,x,z)\,dz}\,dy_*\,dy_*'-n(t,x,y)\right),\\[-35pt]
&{\bf (SIM)}\nonumber\\ \nonumber
\end{align*}
where $A>0$ is the \emph{phenotypic variance at linkage equilibrium} of the population (see \cite{Fisher2,Bulmer}), $y_{opt}: \mathbb T^d\to \mathbb R$ is a description of the environment (typically $y_{opt}(t,x)$ is the temperature at time $t$ and location $x$), and $\Gamma_{A/2}:\mathbb R\to\mathbb R_+$ denotes the Gaussian distribution with variance $A/2$:
\[\Gamma_{A/2}(y):= \frac 1 {\sqrt{2\pi A}}e^{\frac{-|y|^2}{A}}.\]
\rb{Note that generalisations of the (SIM), for instance more general selection operators, could be introduced, but we focus on that specific model in this manuscript, because of its connection to the Kirkpatrick-Barton model that we describe below.} (SIM) is composed of parabolic terms, common in ecology models (see the Fisher-KPP equation \cite{KPP}, and \cite{Alfaro1}), and a kinetic term, with a factor $\gamma>0$, representing the effect of sexual reproduction. Beyond the importance of this model for applications, (SIM) is an opportunity to develop the analysis methods introduced for other kinetic models (in particular the Boltzmann equation), using an unusual diffusion term in the space variable. This diffusive term enables us to propose a new method to derive a macroscopic limit: using from Wasserstein estimates on the collision operator, we are able to show that when $\gamma>0$ is large, the dynamics of $n$ can be described by a closed equation on its two first moments.

Indeed, we show that if $\gamma>0$ is large, the solutions of (SIM) satisfy 
\[n(t,x,y)\sim N(t,x)\Gamma_A\left(y-Z(t,x)\right),\]
 where the macroscopic quantities $N$ and $Z$ asymptotically  satisfy the \emph{Kirkpatrick-Barton Model}: 
\begin{equation*}
{\bf (KBM)}\phantom{zer} \left\{\begin{array}{l}
\partial_t N(t,x)-\Delta_x N(t,x)= \left[1-\frac 12(Z(t,x)-y_{opt}(t,x))^2-N(t,x)\right]N(t,x),\phantom{fazer}\\ \\
\partial_t Z(t,x)-\Delta_x Z(t,x)=2\frac{\nabla_xN\cdot\nabla_x Z}N(t,x)-A(Z(t,x)-y_{opt}(t,x)).
\end{array}\right.
\end{equation*}
(KBM), introduced in \cite{Kirkpatrick}, is used in the evolutionary ecology literature. So far, (SIM) and (KBM) have received limited attention from the mathematical community. In \cite{Prevost} the existence of solutions for models related to (SIM) is discussed. In \cite{Mirrahimi} the propagation fronts for a simplified model are constructed (this article also contains non-rigorous asymptotics related to the present study), while in \cite{Miller}, the long time dynamics of a different simplified model is discussed. In \cite{Miller2}, travelling waves and steady distributions have been constructed for (\ra{KBM}) when a parameter is small. This remarkable result is obtained by a perturbative argument around the case $\varepsilon=0$ which corresponds to the \ra{Fisher-KPP equation~\cite{FKPP1,KPP}: the small parameter $\varepsilon>0$ could correspond to a weak selection and a weakly heterogeneous environment}. Finally, we mention the study of acceleration fronts for sexual reproduction models in \cite{Dekens1,Calvez5}.  We refer to Section~\ref{subsec:biology} for a discussion of the biological aspects of (SIM), (KBM), and the biological implications of this manuscript's results.

In the case of asexual populations, the last term of (SIM) simplifies considerably: it is then replaced by a local term plus a diffusive part (that represents mutations). These asexual population models have recently received considerable attention, and the propagation phenomena that they present are now well understood. The main idea in the asexual case is to consider the model as a semi-linear parabolic equation, to control the non-local competition term through a Harnack inequality, and to use topological fixed-point arguments to construct propagation fronts \cite{Alfaro1,Berestycki,Bouin2}. Additional difficulties arise when the phenotypic trait $y$ has an impact on the spatial diffusion of individuals in space (see \cite{Bouin2,Turanova,BerestyckiN}), and these models can lead to acceleration fronts \cite{BerestyckiN,Bouin3}. Finally, when the mutation rate is low, these asexual models can be related to constrained Hamilton-Jacobi equations \cite{Bouin4,Turanova,Bouin1}. Note that in the asexual case, the propagation speed of the population (which plays an important role in biology) is given by a linearisation of the model, and is then explicit in terms of a certain principal eigenvalue problem. This simple characterisation of the propagation velocity is  no longer valid in the case of sexual populations, and the macroscopic limit described here can be used to describe the propagation phenomena for (SIM) (we refer to \cite{Caflish,Liu} for a related idea in mathematical physics).

The macroscopic limit we present here is based on the Wasserstein contraction induced by the reproduction operator (see Theorem~\ref{thm:Tanaka}). This contraction property exists for a series of operators appearing in physics or econometry \cite{Bassetti,Bolley,Villani2}, and was originally obtained by Tanaka \cite{Tanaka}. To our knowledge, few rigorous macroscopic/hydrodynamic results have been derived from it (see \cite{Magal} for a result without spatial structure). Note that the strategy adopted here is to combine Wasserstein estimates (for the reproduction term) with estimates of a different nature (parabolic estimates for the spatial dimension). This strategy is related to the work of Carlen and Gangbo \cite{Carlen} (see also \cite{Agueh}), who are interested in a kinetic Fokker-Planck equation that combines a hyperbolic transport term in space with a kinetic operator in velocity space. This kinetic operator involves a contraction of the Wasserstein distance. The authors show the long time convergence of solutions to the set of local Maxwellians, but this large-time convergence is not quantitative, due to the lack of regularity estimates in the spatial variable. In the present study, the presence of a diffusive term in the space variable allows us to push the analysis further. Finally, we are also able to handle the selection/competition term to justify the macroscopic limit of (SIM) described above.

Recently, structured population models with sexual reproduction but without a spatial variable have attracted attention. \gr{The existence and uniqueness of solutions for such problems has been considered in \cite{Patout,Guerand}.} A difficult problem, studied in \cite{Calvez1,Calvez2,Calvez3,Raoul} is to show that solutions converge to a unique steady-state when the selection term has a single local maximum. A second approach (in fact closely related to the asymptotic $\gamma\gg 1$ that we consider in this manuscript) is to assume that the reproduction kernel $\Gamma_{A/2}$ has a small variance, see \cite{Patout,Dekens2,Taing,Calvez4,Dekens1}. \gr{This small variance approach was also also considered in \cite{Guerand}, where an analysis based on moments of the distribution was introduced to describe the dynamics of solutions. }

\section{Main results and organisation of the paper}
\gr{
In this manuscript, we assume that the optimal phenotypic trait $(t,x)\mapsto y_{opt}(t,x)$ and the initial population $(x,y)\mapsto n^0(x,y)$ satisfy
\begin{ass}\label{Assumption}\label{Assumption3}  For some $C_1>0$,
\begin{enumerate}
\item[(i)] $y_{opt}\in C^1(\mathbb R_+\times \mathbb T^d,\mathbb R)$ such that $\|y_{opt}\|_{W^{1,\infty}(\mathbb R_+\times \mathbb T^d,\mathbb R)}<C_1$.
\item[(ii)] $n^0\in C^0(\mathbb T^d\times \mathbb R,\mathbb R_+)$, such that $n^0>0$ and
\[\forall (x,y)\in\mathbb T^d\times\mathbb R,\quad 0<n^0(x,y)\leq \frac{C_0}{1+y^{10}},\]
as well as, for $k\in\{0,1,4\}$,
\[\left\|x\mapsto \int_{\mathbb R} y^kn^0(x,y)\,dy\right\|_{W^{4,\infty}(\mathbb T^d)}\leq C_1.\]
\end{enumerate}
\end{ass}
The assumption $n^0(x,y)\leq \frac{C_0}{1+y^{10}}$ seems a strong assumption, but it is actually coherent with biological applications where populations typically have Gaussian tails in the phenotypic variable. The last regularity assumptions is technical: it is a $W^{4,\infty}(\mathbb T^d)$ bound on moments of the initial population. For $k=0$, it corresponds to a regularity assumption on the initial population size $N^0$, and on the initial mean phenotypic trait $Z^0$ for $k=1$. the case $k=4$ corresponds to a higher moment that will be useful for asymptotic analysis we will perform in this manuscript. Thanks to these assumptions, we will be able to show that $N$ and $Z$, that is 
\begin{equation}\label{NZ}
N(t,x)=\int_{\mathbb R} n(t,x,y)\,dy,\quad Z(t,x)=\int_{\mathbb R} y\frac{n(t,x,y)}{\int_{\mathbb R} n(t,x,z)\,dz}\,dy,
\end{equation}
are regular in both variables (see Remark~\ref{rem:thm}). We will then be able to apply the comparison principle to several equations, which will be crutial in our analysis. Note that $N(t,x)$ and $Z(t,x)$ represent respectively the population size and the mean phenotypic trait \ra{at time $t\geq 0$ and position $x\in\mathbb T^d$}. One could probably consider weaker assumptions on $n^0$, but we leave this problem for future research.  }

\subsection{Existence and uniqueness of solutions of (SIM)}
\label{subsec:existencemain}
\gr{ In this section, we consider (SIM) with given coefficients, and in particular, $\gamma>0$ is fixed. We indicate the dependency of the constants in $\gamma>0$ for readability (see e.g. \eqref{est:uppery6}), but not in the proofs. The next section will be devoted to the asymptotic limit $\gamma\to \infty$ and the impact of $\gamma>0$ will then be carefully monitored. We define solutions of (SIM) as follows:
\begin{defi} \label{def:sol}
Let $T\in[0,\infty]$, $A,\gamma>0$, $y_{opt}\in C^1([0,T]\times \mathbb T^d)$ and $n^0\in C^0(\mathbb T^d\times\mathbb R,\mathbb R_+)$.
$n\geq 0$ is a solution of (SIM) with initial data $n^0$ if $n\in L^\infty([0,T]\times\mathbb T^d,L^1(\mathbb R))$, $n\in L^2([0,T]\times\mathbb R, H^1(\mathbb T^d))$, $\partial_t n\in L^2([0,T]\times\mathbb R, H^{-1}(\mathbb T^d))$, and satisfies, for $v\in L^2(\mathbb R\times H^1(\mathbb T^d))$ and almost every $t\in[0,T]$, 
\begin{align}
&\int_{\mathbb R}\int_{\mathbb T^d} \partial_t n(t,x,y)v(x,y)\,dx \,dy+
\int_{\mathbb R}\int_{\mathbb T^d} \nabla_x n(t,x,y)\cdot\nabla_x v(x,y)\,dx \,dy\nonumber\\
&= \int_{\mathbb T^d}\int_{\mathbb R} \left(1+\frac A2-\frac 12(y-y_{opt}(t,x))^2-\int_{\mathbb R} n(t,x,z)\,dz\right)n(t,x,y)v(x,y)\,dy\,dx\nonumber\\
&\quad + \gamma\int_{\mathbb T^d}\int_{\mathbb R} v(x,y))\left[\iint_{\mathbb R^2} \Gamma_{A/2}\left(y-\frac{y_*+y_*'}2\right)\frac{n(t,x,y_*)n(t,x,y_*')}{\int_{\mathbb R} n(t,x,z)\,dz}\,dy_*\,dy_*'-n(t,x,y)\right]\,dy\,dx,\label{eq:def:sol}
\end{align}
it should also satisfy $n(0,\cdot,\cdot)=n^0$
and for any $\tilde T\in [0,T]$, $\tilde T<\infty$, there should be a constant $\bar C>0$ such that
\begin{align}\label{est:regnSIM}
&\sup_{y\in\mathbb R}\left(\sup_{t\in[0,\tilde T]}\|n(t,\cdot,y)\|_{H^1(\mathbb T^d)}+\|n(\cdot,\cdot,y)\|_{L^2_{loc}([0,\tilde T),H^2(\mathbb T^d))}+\|\partial_t n(\cdot,\cdot,y)\|_{L^2_{loc}([0,\tilde T)\times \mathbb T^d)}\right) <\bar C<\infty.
\end{align}
\end{defi}
This notion of solution is closely connected to the notion of strong solutions in e.g. \cite{Liebermann}. We prove the existence of such solutions of (SIM) in the following proposition:
\begin{prop}\label{prop:existSIM}

Let $y_{opt}\in W^{1,\infty}(\mathbb R_+\times\mathbb T^d)$,  $n^0$ satisfying Assumption~\ref{Assumption3}, $A>0$, and $\gamma>2+A+\|y_{opt}\|_{L^\infty(\mathbb R_+\times\mathbb T^d)}$. There is a unique global solution $n$ of (SIM) with the initial data $n^0$, in the sense of definition~\ref{def:sol}. More precisely, for a constant $C_\gamma$ that may depend on $\gamma$ and $(t,x,y)\in\mathbb R_+\times\mathbb T^d\times\mathbb R$,
\begin{equation}\label{est:uppery6}
 n(t,x,y)\leq \frac{C_\gamma}{1+y^{10}},
\end{equation}
while for any $T>0$, there is $C_{\gamma,T}>0$ such that
\begin{align}\label{est:regn3}
&\textrm{sup}_{t\in[0,T]}\|n(t,\cdot,y)\|_{H^1(\mathbb T^d)}+\|n(\cdot,\cdot,y)\|_{L^2([0,T],H^2(\mathbb T^d))}+\|\partial_t n(\cdot,\cdot,y)\|_{L^2([0,T]\times \mathbb T^d)} \leq \frac {C\gamma(1+\gamma)}{1+y^8}.
\end{align}
Moreover, for $(t,x,y)\in \mathbb R_+\times\mathbb T^d\times[-1,1]$,
\begin{equation}\label{eq:lowernprop}
n(t,x,y)\geq \left(\min_{\mathbb R_+\times\mathbb T^d\times[-1,1]}n^0\right)e^{-C_\gamma t}.
\end{equation}
\end{prop}

\begin{rem}\label{rem:comparaison}
In Section~\ref{subsec:regNV}, we show that the macroscopic quantities $N$ (resp. $Z$, $V$) defined by \eqref{NZ} (resp. \eqref{NZ}, \eqref{def:V}) from the solution $n$ of (SIM) given by Proposition~\ref{prop:existSIM} are $C^1$ in $t$ and $C^2$ in $x$, and they satisfy \eqref{NZ}, \eqref{def:V} seen as equalities between continuous functions. In particular the comparison principle (Corollary 7.4 p. 159 in \cite{Liebermann}) applies to \eqref{eq:Zmod} and \eqref{eq:V} when the terms $\frac{\nabla_x N(t,x)}{N(t,x)}$ and $\tilde n(t,x,y)$ are  considered as given coefficients (we refer to \eqref{NZ} and \eqref{def:tilden} for the definition of $N$ and $\tilde n$). The comparison principle will play a crucial role in Section~\ref{sec:macro}.

The existence of solutions to (KBM) that are $C^1$ in $t$ and $C^2$ in $x$ will be shown in Section~\ref{Appendix:uniquenessSIM} (using Proposition~\ref{prop:Linftyboundwithoutalpha}), and the uniqueness of such solutions is proven in Proposition~\ref{prop:uniqNZ}.
\end{rem}

\medskip

\rb{We prove Proposition~\ref{prop:existSIM} in Section~\ref{sec:existence}. A preliminary step is to prove the existence of global solutions $n_R$ to the following truncated problem for $R>0$: 
\begin{align}
&\partial_t n_R(t,x,y)-\Delta_x n_R(t,x,y)=\left(1+\frac A2-\frac 12(y-y_{opt}(t,x))^2-\int_{\mathbb R} n_R(t,x,z)\,dz\right)n_R(t,x,y)\nonumber\\
&\quad +\gamma\left(1_{|y|\leq R}\iint_{\mathbb R^2} \Gamma_{A/2}\left(y-\frac{y_*+y_*'}2\right)\frac{n_R(t,x,y_*)n_R(t,x,y_*')}{\int_{\mathbb R} n_R(t,x,z)\,dz}\,dy_*\,dy_*'-n_R(t,x,y)\right),\label{eq:nRsol} 
\end{align}
together with $n_R(0,x,y)=n^0(x,y)1_{|y|\leq R}$. We do so in Section~\ref{subsec:existencetruncated}: in Lemma~\ref{lem:exist}, we use a Cauchy-Lipschitz-type construction to construct solutions of \eqref{eq:nRsol} on a short time interval. We extend these to prove the existence of global solutions of \eqref{eq:nRsol} in Proposition~\ref{prop:existencenR}, and we obtain uniform estimates on the tails of solutions in the phenotypic variable $y$ in Proposition~\ref{prop:tailsnR}. In Section~\ref{sec:existence}, we show that $n_R$ converges weakly to a solution $n$ of (SIM) when $R\to\infty$, which proves Proposition~\ref{prop:existSIM}. }Note that to show the convergence of the non-linear birth term, we use the Dunford-Pettis Theorem: since the functions $n_R$ are uniformly bounded, their weak convergence implies their strong convergence. This idea is reminiscent of weak compactness arguments in kinetic theory \cite{VillaniReview}, where the $L^\infty$ bound on solutions of the truncated problem is replaced by estimates on the entropy of solutions and the De La Vallée-Poussin equi-integrability Criterion.
\medskip

In the next proposition, we show that solutions of (SIM) are unique:

\begin{prop}\label{prop:uniqueness}
Let $y_{opt}$,  $n^0$ satisfying Assumption~\ref{Assumption3}, $A>0$, and $\gamma>0$. Let $n,\tilde n$ solutions of (SIM), in the sense of definition~\ref{def:sol}, defined on $[0,T]\times\mathbb T^d\times\mathbb R$ with the same initial data $n^0\in L^1(\mathbb T^d\times \mathbb R,\mathbb R_+)$, such that
\[\max_{x\in\mathbb T^d} \int(1+y^4)n^0(x,y)\,dy<\infty.\]
If there is a constant $C>0$ such that
\begin{equation}\label{eq:assmoments}
\forall (t,x)\in [0,T]\times\mathbb T^d,\quad \int_{\mathbb R} n(t,x,y)\left(1+y^4\right)\,dy\leq C,\quad \int_{\mathbb R} \tilde n(t,x,y)\left(1+y^4\right)\,dy\leq C,
\end{equation}
then $n=\tilde n$. 
\end{prop}

\begin{rem}\label{rem:tailsn0}
The assumption~\eqref{eq:assmoments} on moments of $n$ and $\tilde n$ could be replaced by the assumption $n^0(x,y)\leq \frac C{1+y^6}$. Indeed, we can use an argument introduced in the proof of Proposition~\ref{prop:existSIM} (see Step~1) to show that this bound is propagated for $t\in[0,T]$, and \eqref{eq:assmoments} is then always satisfied by $n$ and $\tilde n$.
\end{rem}

The proof of Proposition~\ref{prop:uniqueness} is detailed in Section~\ref{Appendix:uniquenessSIM}. It follows an argument developed by J. Guerand, M. Hillairet and S. Mirrahimi in \cite{Guerand} (see section~A.3) for a model without spatial structure. We adapt their proof to include the spatial variable $x$ in (SIM)  with minimal adjustments.

}

\subsection{Macroscopic limits of solutions of (SIM)}

Our main asymptotic result, stated below, shows that when $\gamma>0$ is large, $n$ satisfies:
\[n(t,x,y)\underset{\gamma\gg 1}{\sim}N(t,x)\Gamma_A(y-Z(t,x)),\]
and the couple $(N,Z)$, given by \eqref{NZ}, is close to the solution of (KBM)  with initial data 
\[\Big (N(0,\cdot),Z(0,\cdot)\Big)=\left(\ra{N^0},\ra{Z^0}\right).\]

\begin{theo}\label{Thm:macro}
Let $y_{opt},\,n^0$ satisfying Assumption~\ref{Assumption} and $A>0$. For $\gamma>0$, let  $n_\gamma\in L^\infty(\mathbb R_+\times \mathbb T^d,L^1(\mathbb R))$ the solution of (SIM) with initial data $n^0$. 

There exist  $\bar \gamma>0$, $C>0$ and $\theta\in(0,1)$ such that for any $\gamma>\bar \gamma$,
there exist $\varphi_{N,\gamma},\,\varphi_{Z,\gamma}: \mathbb R_+\times\mathbb T^d\to\mathbb R$ satisfying
\begin{equation}\label{thm:est2}
\|\varphi_{N,\gamma}(t,\cdot)\|_{L^\infty(\mathbb T^d)}+\|\varphi_{Z,\gamma}(t,\cdot)\|_{L^\infty(\mathbb T^d)}\leq \frac C{\gamma^\theta}+C \mathds 1_{\left[0,C\ln\gamma/\gamma\right]}(t)
\end{equation}
such that the functions $N_\gamma$ and $Z_\gamma$ defined from $n_\gamma$ by \eqref{NZ} satisfy the following equations as equalities between continuous functions:
\begin{equation}\label{eq:thm}
\left\{\begin{array}{l}
\partial_t N_\gamma(t,x)-\Delta_x N_\gamma(t,x)= \left[1-\frac 12(Z_\gamma(t,x)-y_{opt}(t,x))^2-N_\gamma(t,x)+\varphi_{N,\gamma}(t,x)\right]N_\gamma(t,x),\\ \\
\partial_t Z_\gamma(t,x)-\Delta_x Z_\gamma(t,x)=2\frac{\nabla_xN_\gamma\cdot\nabla_x Z_\gamma}{N_\gamma}(t,x)-A(Z_\gamma(t,x)-y_{opt}(t,x))+\varphi_{Z,\gamma}(t,x),
\end{array}\right.
\end{equation}
where $(t,x)\in\mathbb R_+\times \mathbb T^d$. Moreover,
\begin{equation}\label{thm:est1}
\max_{(t,x)\in \left[\theta\ln\gamma/\gamma,\infty\right)\times \mathbb T^d}W_2\left(\frac{n_\gamma(t,x,\cdot)}{N_\gamma(t,x)},\Gamma_{A}(\cdot-Z_\gamma(t,x))\right)\leq \frac{C}{\gamma^\theta},
\end{equation}
where $W_2$ stands for the Wasserstein distance, as defined in Section~\ref{subsec:defWasserstein}.
\end{theo}
We show in Section~\ref{Appendix:uniquenessSIM} that this theorem implies the convergence of $(N_\gamma,Z_\gamma)$ to the solution $(\bar N,\bar Z)$ of (KBM):
\begin{equation}\label{eq:cv}
\begin{array}{l}
N\underset{\gamma\to\infty}\longrightarrow\bar N\quad \textrm{ in }\quad L^\infty_{loc}((0,+\infty),L^\infty(\mathbb T^d)),\\
\\
Z\underset{\gamma\to\infty}\longrightarrow\bar Z\quad \textrm{ in }\quad L^\infty_{loc}((0,+\infty),L^\infty(\mathbb T^d)).
\end{array}
\end{equation}

\begin{rem} \label{rem:thm}
The estimates given by Theorem~\ref{Thm:macro} are global in time, even though $N(t,\cdot)$ may converge to $0$ when $t\to\infty$. This is possible because the last term of (SIM) (ie the "kinetic" operator) scales linearly with $n$. It therefore also holds for populations that are going extinct. 

\gr{Estimate \eqref{eq:thm} is a stronger result (more quantitative) than what is typically obtained in macroscopic limits, through e.g. a truncated Hilbert expansion \cite{Caflisch} or compactness arguments \cite{Bardos}. This is possible thanks to the diffusion operator in $x$ present in (SIM) that brings some regularity to solutions.}

\end{rem}

In Section~\ref{sec:existence}, we prove the existence of global solutions for (SIM).  In Section~\ref{subsec:highermoments}, we show that an $L^\infty([0,\tau]\times \mathbb T^d)$ bound on $Z$ (with $\tau\geq 0$) implies an estimate on the fourth moment of $y\mapsto n(t,x,\cdot)$ for $t\in[0,\tau+\bar \tau]$, with $\bar \tau>0$. In Section~\ref{subsec:RegularityNZ} we show that $Z$ is H\"older continuous, provided we have a bound on $\|Z\|_{L^\infty}$. This regularity is used in Section~\ref{subsec:Wasserstein} together with a Tanaka-type inequality  (see Theorem~\ref{thm:Tanaka} in the Appendix) to show that $\frac{n(t,x,\cdot)}{N(t,x)}$ is close to $\Gamma_{A}(\cdot-Z(t,x))$ for the Wasserstein distance $W_2$ when $\gamma>0$ is large enough. Finally in Section~\ref{subsec:refinedLinfty} we use the estimates mentioned above to obtain a uniform bound on $\|Z\|_{L^\infty(\mathbb R_+\times \mathbb T^d)}$\rb{, through a contradiction argument showing that the maximal time $\tau$ where this estimate hold is actually $\tau=\infty$}. This estimate  implies the macroscopic limit described in Theorem~\ref{Thm:macro}.

\subsection{Biological interpretation of the model and impact for ecology}\label{subsec:biology}


The first term on the right-hand side of (SIM), $\Delta_x n(t,x,y)$, represents the dispersion of individuals in space. The  term $\left(1+\frac A2-\frac 12(y-y_{opt}(t,x))^2\right)n(t,x,y)$ represents the effect of natural selection: the individuals whose phenotypic trait $y$ is far from the optimal trait $y_{opt}(t,x)$ have a high mortality rate. The  function $y_{opt}$ should therefore be considered as a description of the environment and is a given function. For instance the trait $y$ is could be the temperature to which an individual is best adapted to, and $y_{opt}$ is then the predicted map of temperatures. The term $-\left(\int n(t,x,z)\,dz\right)n(t,x,y)$ in  (SIM) represents competition: all individuals present at time $t\geq 0$ and location $x\in\mathbb T^d$ are competing for resources. The last term describes the effect of sexual reproductions: when parents give birth to an offspring, the phenotypic trait of the offspring is drawn from a normal distribution with a fixed variance $A/2$ centred on the mean of the parents' traits. This model for the effect of sexual reproduction on a continuous phenotypic trait is known as the \emph{Infinitesimal Model}. It was introduced by Fisher in 1919 \cite{Fisher}, and is used in population genetics either for theoretical purpose \cite{Bulmer,Turelli,Barton} or for practical purposes \cite{Martinez, Verrier}. The $\gamma\gg1$ limit corresponds to a short generation time and it can be seen as the implicit assumption behind the classical \emph{Linkage Equilibrium} assumption used in population genetics (see for instance \cite{Bulmer}): in the framework of the \emph{Infinitesimal Model} the Linkage Equilibrium assumption  implies that the population distribution $\tilde n(t,x,\cdot)$ is Gaussian with fixed variance. Numerical simulations (see \cite{Mirrahimi}) suggest that the macroscopic limit model (KBM) provides a good description of the dynamics of solutions of (SIM) even when $\gamma$ is not very large: for $\gamma=2$, the information provided by (KBM) is already relevant.

We expect (SIM) to be related to a well-chosen Individual Based Model via a large number of individuals argument, but as far as we know, no such asymptotics exists at present. This type of derivation exists for asexual models \cite{Champagnat}, but here an additional difficulty arises: describing (SIM) as a large population limit of an Individual Based Model will require a precise understanding of the link between explicit genetic models and the \emph{Infinitesimal Model} (which is at the root of the reproduction operator appearing in (SIM)). We refer to \cite{Barton} for more information on this limit.

(KBM) was introduced by Kirkpatrick and Barton in 1997 \cite{Kirkpatrick}, and is used to model the range dynamics of populations, particularly when these populations are subject to climate change, see for example \cite{Bridle,Aguilee}. The success of (KBM) comes from to the complex dynamics it presents \cite{Kirkpatrick,Mirrahimi}: even for a very simple environment described by $y_{opt}(t,x)=Bx$ (and $x\in\mathbb R$), the population can either become extinct, survive without spreading, or spread (see \cite{Kirkpatrick,Miller2}). From a mathematical point of view, these dynamics raise a number of difficult questions. Several simplified models exist (see \cite{Pease,Mirrahimi}), and we refer to \cite{Miller,Mirrahimi} for analysis of some of these simplified models.

A good understanding of the connections between (SIM) and (KBM) (and other connections with stochastic models) has practical implications: the different scales (such as the mesoscopic scale of (SIM) and the  macroscopic scale of the (KBM)) are not clearly distinct in most biological systems, and  easy navigation between the different scales of description is an essential feature of the theory. This can be seen in \cite{Aguilee}, where the macroscopic limit from (SIM) to (KBM) plays an important role. We believe that these models will play an important role in understanding the effect of climate change on species, and are a valuable complement to Species Distribution Models (see e.g. \cite{Guisan}) that currently prevail.


\subsection{Preliminary: equations satisfied by solutions of the (SIM)}\label{subsec:preliminary}
Let $n:(t,x,y)\mapsto n(t,x,y)$ a solution of (SIM). In this section we derive heuristically equations satisfied by $n$ normalized by $\int n(t,x,y)\,dy$ and other moments of $n$. The existence and regularity results stated in Section~\ref{subsec:existencemain} (see Proposition~\ref{prop:existSIM} and Remark~\ref{rem:thm}) will provide a rigorous framework for these partial differential equations.

\medskip

If we integrate (SIM) along the variable $y$, we get that the population size $N$ (see \eqref{NZ} for its definition) satisfies, for $t\geq 0$ and $x\in\mathbb T^d$,
\begin{equation}\label{eq:Nmod}
\partial_t N(t,x)-\Delta_x N(t,x)= \left[1+\frac A2-N(t,x)\right]N(t,x)-\frac 12 \int_{\mathbb R} (y-y_{opt}(t,x))^2n(t,x,y)\,dy.
\end{equation}
We define the normalized profile of the population,
\begin{equation}\label{def:tilden}
\tilde n(t,x,y)=\frac{n(t,x,y)}{N(t,x)},
\end{equation}
which satisfies
\begin{align}
&\partial_t \tilde n(t,x,y)-\Delta_x \tilde n(t,x,y)\nonumber\\
&\quad =2\frac{\nabla_x N(t,x)}{N(t,x)}\cdot \nabla_x \tilde n(t,x,y)+\gamma\left( T(\tilde n(t,x,\cdot))-\tilde n(t,x,y)\right)\nonumber\\
&\qquad+\frac 12 \tilde n(t,x,y)\left(\int_{\mathbb R} (z-y_{opt}(t,x))^2\tilde n(t,x,z)\,dz-(y-y_{opt}(t,x))^2\right),\label{eq:tildenmod}
\end{align}
where $T$, the Infinitesimal operator, \rb{is defined by
\begin{equation}\label{def:T}
T(\tilde n)(y):=\int_{\mathbb R} \Gamma_{A/2}\left(y-\frac{y_*+y_*'}2\right)\tilde n(t,y_*)\tilde n(t,y_*')\,dy_*\,dy_*'.
\end{equation}
In the Appendix (Section~\ref{subsec:infinitesimaloperator}), we detail the important properties of this operator.} From this expression, we can deduce the following equation on the mean phenotypic trait of the population $Z$ (see \eqref{NZ} for its definition):
\begin{align}
&\partial_t Z(t,x)-\Delta_x Z(t,x)\nonumber\\
&\quad =2\frac{\nabla_xN(t,x)}{N(t,x)}\cdot\nabla_x Z(t,x)-\frac 12\int_{\mathbb R} \left(y-Z(t,x)\right)(y-y_{opt}(t,x))^2 \tilde n(t,x,y)\,dy.\label{eq:Zmod}
\end{align}
We define 
\begin{equation}\label{def:V}
 V(t,x):= \int_{\mathbb R} |y|^4\tilde n(t,x,y)\,dy,
\end{equation}
and thanks to \eqref{eq:tildenmod}, we show that $V$ satisfies
\begin{align}
&\partial_t V(t,x)-\Delta_x  V(t,x)\nonumber\\
&\quad =2\frac{\nabla_x N(t,x)}{N(t,x)}\cdot\nabla_x V(t,x)+\frac 12\int_{\mathbb R} \left(V(t,x)-|y|^4\right) (y-y_{opt}(t,x))^2\tilde n(t,x,y)\,dy\nonumber\\
&\qquad+\gamma\left(\int_{\mathbb R} |y|^4T(\tilde n(t,x,\cdot))(y)\,dy-V(t,x)\right).\label{eq:V}
\end{align}

\section{Existence of solutions for (SIM)}\label{sec:existence}

In this section, we prove Proposition~\ref{prop:existSIM}. For $R>0$, we denote by $n_R$ the solution of \eqref{eq:nRsol} provided by Proposition~\ref{prop:existencenR}. 

\noindent{\textbf{Step 1: We show that the sequence $(n_R)_{R>0}$ converges to a limit $n$}}

An integration of \eqref{eq:nRsol} shows that $N_R(t,x):=\int n_R(t,x,y)\,dy$ satisfies
\begin{align}
&\partial_t N_R(t,x)-\Delta_x N_R(t,x)= \left[1+\frac A2-N_R(t,x)\right]N_R(t,x)-\frac 12 \int_{\mathbb R} (y-y_{opt}(t,x))^2n_R(t,x,y)\,dy\nonumber\\
&\quad +\gamma\left(\iint_{\mathbb R^2} \left(\int_{-R}^R\Gamma_{A/2}\left(y-\frac{y_*+y_*'}2\right)\,dy\right)\frac{n_R(t,x,y_*)n_R(t,x,y_*')}{\int_{\mathbb R} n_R(t,x,z)\,dz}\,dy_*\,dy_*'-N_R(t,x)\right).\label{eq:nR}
\end{align}
\rb{Thanks to \eqref{est:uppery8} (that also implies a bound on $N_R$), the right hand side of this equation is uniformly bounded (independently from $R>0$):
\begin{align*}
&\left|\partial_t N_R(t,x)-\Delta_x N_R(t,x)\right|\leq \left[1+\frac A2+N_R(t,x)\right]N_R(t,x)+\frac 12 \int_{\mathbb R} (1+y^2)n_R(t,x,y)\,dy\nonumber\\
&\quad +\gamma\left(2R\Gamma_{A/2}(0)\left(\int_{\mathbb R} \frac{n_R(t,x,y_*)}{\int_{\mathbb R} n_R(t,x,z)\,dz}\,dy_*\right)\left(\int_{\mathbb R}n_R(t,x,y_*')\,dy_*'\right)+N_R(t,x)\right)<C.
\end{align*}}
 Theorem 5  in Section 7.1.3 of  \cite{Evans} implies that $N_R$ is a \emph{strong} solution  (in the sense of \cite{Evans}) of this heat equation with the right hand side as a $0-$th order term: for any $T>0$, there exists a constant $C_\gamma>0$, independent from $R>0$ (but that depends on $\gamma>0$) such that 
\begin{align}\label{est:regNR}
&\textrm{sup}_{t\in[0,T]}\|N_R(t,\cdot)\|_{H^1(\mathbb T^d)}+\|N_R\|_{L^2([0,T],H^2(\mathbb T^d))}+\|\partial_t N_R\|_{L^2([0,T]\times \mathbb T^d)} \leq C_\gamma.
\end{align}
We recall that $n_R$ also satisfies a regularity estimate independent from $R>0$, see \eqref{est:regnR}.

\medskip

Thanks to \eqref{est:uppery2}, the family of measures $n_R$ is tight on $[0,T]\times\mathbb T^d\times \mathbb R$, for any bounded time interval $[0,T]$. Thanks to Prokhorov's theorem, $(n_R)_{R>1}$ converges when $R\to+\infty$, up to an extraction $(R_k)$, for the weak topology of measures, to a limit $n$. \gr{Thanks to \eqref{est:uppery2}, the sequence $(n_{R_k})_{k}$ is uniformly bounded in $L^\infty([0,T]\times\mathbb T^d\times \mathbb R)$, and thus equi-integrable. We can therefore apply the Dunford-Pettis Theorem to show that $(n_{R_k})_{k}$ converges strongly to $n$ in $L^1([0,T]\times\mathbb T^d\times \mathbb R)$, when $k\to+\infty$, up to an extraction. Note that we can use a diagonal argument to show that this convergence holds for any $T>0$. This implies in particular the convergence of $N_{R_k}$ to $N(t,x):=\int_{\mathbb R}n(t,x,y)\,dy$ in $L^1([0,T]\times\mathbb T^d)$:
\[\int_0^T\int_{\mathbb T^d} |N_{R_k}(t,x)-N(t,x)|\,dx\,dt\leq \int_0^T\int_{\mathbb T^d}\int_{\mathbb R} |n_{R_k}(t,x,y)-n(t,x,y)|\,dy\,dx\,dt.\]}

The uniformity of \eqref{est:regnR} implies that the limit $n$ of $n_{R_k}$ satisfies \eqref{est:regnSIM}, and similarly the estimate \eqref{est:regNR} on $N_{R_k}$ implies that $N$ satisfies the same estimate. Next, we prove a lower bound estimate on $n_{R_k}$. Thanks to \eqref{eq:nRsol}, for $(t,x)\in\mathbb R_+\times\mathbb T^d$ and $|y|\leq 1$,
\begin{align}
&\partial_t n_{R_k}(t,x,y)-\Delta_x n_{R_k}(t,x,y)\nonumber\\
&\quad \geq -\left(\frac 12\left(R+\|y_{opt}\|_{L^\infty(\mathbb R_+\times\mathbb T^d)}\right)^2+\|N_{R_k}\|_{L^\infty(\mathbb R_+\times\mathbb T^d)}+\gamma\right)n_{R_k}(t,x,y),\label{eq:eqsubsol}
\end{align}
and we notice that for any fixed $y\in[-1,1]$, 
\[(t,x)\mapsto \left(\min_{\mathbb R_+\times\mathbb T^d\times[-1,1]}n^0\right)e^{-\left(\frac 12\left(R+\|y_{opt}\|_{L^\infty(\mathbb R_+\times\mathbb T^d)}\right)^2+\|N_{R_k}\|_{L^\infty(\mathbb R_+\times\mathbb T^d)}+\gamma\right)t}\]
is a sub-solution of \eqref{eq:eqsubsol}, and we can use the comparison principle to show that for any $(t,x,y)\in \mathbb R_+\times\mathbb T^d\times[-1,1]$,
\begin{equation}\label{eq:lowern}
n_{R_k}(t,x,y)\geq \left(\min_{\mathbb R_+\times\mathbb T^d\times[-1,1]}n^0\right)e^{-\left(\frac 12\left(R+\|y_{opt}\|_{L^\infty(\mathbb R_+\times\mathbb T^d)}\right)^2+\|N_{R_k}\|_{L^\infty(\mathbb R_+\times\mathbb T^d)}+\gamma\right)t}.
\end{equation}
Since this estimate is uniform in $k\in\mathbb N$, the limit $n$ of $n_{R_k}$ satisfies the same estimate, which proves \eqref{eq:lowernprop}.

\medskip

\noindent{\textbf{Step 2: We show that $n$ is a solution of (SIM)}}

For $\phi\in C^0_c([0,\infty)\times\mathbb T^d\times\mathbb R)$ such that $\partial_t \phi,\nabla_x \phi,\Delta_x\phi\in C^0_c(\mathbb R_+\times\mathbb T^d\times\mathbb R)$. Since $\phi$ is compactly supported, there is $B>0$ such that $\textrm{supp }\phi\subset[0,B]\times\mathbb T^d\times [-B,B]$ for some $B>0$. Since $\partial_t\phi,\Delta_x \phi\in L^1([0,\infty)\times\mathbb T^d\times\mathbb R)$,
\begin{align}
&\int_{\mathbb R_+}\int_{\mathbb T^d}\int_{\mathbb R} \left(\partial_t\phi(t,x,y)+\Delta_x\phi(t,x,y)\right) n_{R_k}(t,x,y)\,dy\,dx\,dt\nonumber\\
&\quad -\int_{\mathbb R_+}\int_{\mathbb T^d}\int_{\mathbb R} \left(\partial_t\phi(t,x,y)+\Delta_x\phi(t,x,y)\right) n(t,x,y)\,dy\,dx\,dt\underset{k\to\infty}{\longrightarrow}0.\label{est:weaksol}
\end{align}
\gr{We can also  estimate the following quantity:
\begin{align}
&I_0:=\bigg|\int_{\mathbb R_+}\int_{\mathbb T^d}\int_{\mathbb R} \phi(t,x,y)\bigg(\left(1+\frac A2-\frac 12(y-y_{opt}(t,x))^2-\int_{\mathbb R} n_{R_k}(t,x,z)\,dz\right)n_{R_k}(t,x,y)\nonumber\\
&\qquad -\left(1+\frac A2-\frac 12(y-y_{opt}(t,x))^2-\int_{\mathbb R} n(t,x,z)\,dz\right)n(t,x,y)\bigg) n(t,x,y)\,dy\,dx\,dt\bigg|\nonumber\\
&\quad \leq C\left|\int_{\mathbb R_+}\int_{\mathbb T^d}\int_{\mathbb R}\phi(t,x,y)\left(1+\frac A2-\frac 12(y-y_{opt}(t,x))^2-N(t,x)\right)\left(n_{R_k}(t,x,y)-n(t,x,y)\right)\,dy\,dx\,dt\right|\nonumber\\
&\qquad + \left|\int_{\mathbb R_+}\int_{\mathbb T^d}\left(\int_{\mathbb R}\phi(t,x,y)n_{R_k}(t,x,y)\,dy\right)\left(N(t,x)-N_{R_k}(t,x)\right)\,dx\,dt\right|=|I_1|+|I_2|.\label{est:II1I2I3}
\end{align}
To estimate  $|I_1|$, we notice that $(t,x,y)\mapsto\phi(t,x,y)\left(1+\frac A2-\frac 12(y-y_{opt}(t,x))^2-N(t,x)\right)$ is a compactly supported bounded function that does not depend on $k\in\mathbb N$.  The convergence  of $n_{R_k}$ to $n$ in $L^1([0,T]\times\mathbb T^d\times\mathbb R)$ then implies $|I_1|\to 0$ as $k\to\infty$. To estimate $|I_2|$, we take advantage of $\textrm{supp }\phi\subset[0,B]\times\mathbb T^d\times [-B,B]$:
\begin{align*}
|I_2|&\leq \|\phi n_{R_k}\|_{L^\infty(\mathbb R_+\times\mathbb T^d, L^1(\mathbb R))}\|N-N_{R_k}\|_{L^1([0,B]\times\mathbb T^d)}\underset{k\to\infty}{\longrightarrow}0,
\end{align*}
thanks to the fact that $\phi\in C^0_c([0,\infty)\times\mathbb T^d\times\mathbb R)$, the  uniform bound \eqref{est:uppery8} on $n_{R_k}$, and the convergence of $N_{R_k}$ to $N$ in $L^1_{loc}(\mathbb R_+\times\mathbb T^d)$. We can therefore conclude the estimate \eqref{est:II1I2I3}:
\begin{equation}\label{est:IStep3}
I_0\underset{k\to\infty}{\longrightarrow}0.
\end{equation}
To estimate the difference between the birth terms for $n_{R_k}$ and $n$, we notice that
\begin{align}
&I_3:=\bigg|\int_{\mathbb R_+}\int_{\mathbb T^d}\int_{\mathbb R} \phi(t,x,y)\bigg[1_{|y|\leq R_k}\iint_{\mathbb R^2} \Gamma_{A/2}\left(y-\frac{y_*+y_*'}2\right)\frac{n_{R_k}(t,x,y_*)n_{R_k}(t,x,y_*')}{\int_{\mathbb R} n_{R_k}(t,x,z)\,dz}\,dy_*\,dy_*'\nonumber\\
&\quad -\iint_{\mathbb R^2} \Gamma_{A/2}\left(y-\frac{y_*+y_*'}2\right)\frac{n(t,x,y_*)n(t,x,y_*')}{\int_{\mathbb R} n(t,x,z)\,dz}\,dy_*\,dy_*'\bigg]\,dy\,dx\,dt\bigg|\nonumber\\
&\quad = \bigg|\int_{\mathbb R_+}\int_{\mathbb T^d}\int_{\mathbb R} \phi(t,x,y)\iint_{\mathbb R^2} \Gamma_{A/2}\left(y-\frac{y_*+y_*'}2\right)\bigg[\left(1_{|y|\leq R_k}-1\right)\frac{n_{R_k}(t,x,y_*)n_{R_k}(t,x,y_*')}{\int_{\mathbb R} n_{R_k}(t,x,z)\,dz}\nonumber\\
&\qquad +\frac{\left(n_{R_k}(t,x,y_*)-n(t,x,y_*)\right)n_{R_k}(t,x,y_*')}{\int_{\mathbb R} n_{R_k}(t,x,z)\,dz}\nonumber\\
&\qquad +\frac{n(t,x,y_*)n_{R_k}(t,x,y_*')}{\left(\int_{\mathbb R} n(t,x,z)\,dz\right)\left(\int_{\mathbb R} n_{R_k}(t,x,z)\,dz\right)}\left(N(t,x)-N_{R_k}(t,x)\right)\,dy_*\,dy_*'\nonumber\\
&\qquad +\frac{n(t,x,y_*)\left(n_{R_k}(t,x,y_*')-n(t,x,y_*')\right)}{\int_{\mathbb R} n(t,x,z)\,dz}\bigg]\,dy_*\,dy_*'\,dy\,dx\,dt\bigg|\leq |I_4+I_5+I_6+I_7|.\label{est:I4I7}
\end{align}
The test function $\phi$ is compactly supported and then $\phi(t,x,y)\left(1_{|y|\leq R_k}-1\right)\equiv 0$ if $k\in\mathbb N$ is large enough, which implies $I_4=0$. To estimate $I_6$, we notice that
\begin{align}
|I_6|&=\bigg|\int_{\mathbb R_+}\int_{\mathbb T^d}\bigg(\int_{\mathbb R} \phi(t,x,y)\frac{\iint_{\mathbb R^2} \Gamma_{A/2}\left(y-\frac{y_*+y_*'}2\right) n(t,x,y_*)n_{R_k}(t,x,y_*')\,dy_*\,dy_*'}{N(t,x)N_{R_k}(t,x)}\,dy\bigg)\nonumber\\
&\qquad \left(N(t,x)-N_{R_k}(t,x)\right)\,dx\,dt\bigg|\leq \|N-N_{R_k}\|_{L^1([0,B]\times\mathbb T^d)}\nonumber\\
&\qquad \left\|(t,x,y)\mapsto \phi(t,x,y)\frac{\iint_{\mathbb R^2} \Gamma_{A/2}\left(y-\frac{y_*+y_*'}2\right) n(t,x,y_*)n_{R_k}(t,x,y_*')\,dy_*\,dy_*'}{N(t,x)N_{R_k}(t,x)}\,dy\right\|_{L^1([0,B]\times\mathbb T^d,L^1([-B,B]))}\nonumber\\
&\leq \sqrt 2B\|\phi\|_{L^\infty(\mathbb R_+\times\mathbb T^d\times\mathbb R)}\Gamma_{A/2}(0)\|N-N_{R_k}\|_{L^2([0,B]\times\mathbb T^d)}\underset{k\to\infty}{\longrightarrow} 0,\label{est:I6}
\end{align}
where we have used the fact that $\frac{\int_{\mathbb R}n(t,x,y_*)\,dy_*}{N(t,x)}=\frac{\int_{\mathbb R}n_{R_k}(t,x,y_*')\,dy_*'}{N_{R_k}(t,x)}=1$ and the convergence of $N_{R_k}$ to $N$ in $L^1_{loc}(\mathbb R_+\times\mathbb T^d)$. To estimate $I_5$, we notice that
\begin{align}
|I_5|&= \bigg|\int_{\mathbb R_+}\int_{\mathbb T^d}\int_{-B}^B\int_{\mathbb R} \phi(t,x,y)\frac{n_{R_k}(t,x,y_*')}{N_{R_k}(t,x)}\nonumber\\
&\qquad \left(\int_{\mathbb R}\Gamma_{A/2}\left(y-\frac{y_*+y_*'}2\right)\left(n_{R_k}(t,x,y_*)-n(t,x,y_*)\right)\,dy_*\right)\,dy_*'\,dy\,dx\,dt\bigg|\nonumber\\
&\leq \left\|(t,x,y_*')\mapsto \left(\int_{-B}^B\phi(t,x,y)\,dy\right)\frac{n_{R_k}(t,x,y_*')}{N_{R_k}(t,x)}\right\|_{L^\infty([0,B]\times\mathbb T^d\times\mathbb R)}\nonumber\\
&\qquad\Gamma_{A/2}(0)\|n_{R_k}-n\|_{L^1([0,B]\times\mathbb T^d\times \mathbb R)}\underset{k\to\infty}{\longrightarrow} 0,\label{est:I5}
\end{align}
where the first factor on the right hand side of that inequality is bounded  thanks to \eqref{est:uppery2} and \eqref{eq:lowern}, while the last one converges to $0$. Finally, the argument above (to estimate $I_5$) can be reproduced to show that $I_7\to 0$ when $k\to\infty$.

\medskip

These estimates on $I_3$ (see \eqref{est:I4I7}, as well as \eqref{est:II1I2I3}, \eqref{est:IStep3} and \eqref{est:weaksol} show that for any $\phi\in C^0_c(\mathbb R_+\times\mathbb T^d\times\mathbb R)$,
\begin{align*}
0&=\lim_{k\to\infty}\int_{\mathbb R_+}\int_{\mathbb T^d}\int_{\mathbb R} \left(\partial_t\phi(t,x,y)+\Delta_x\phi(t,x,y)\right) n_{R_k}(t,x,y)\,dy\,dx\,dt\\
&\quad +\int_{\mathbb R_+}\int_{\mathbb T^d}\int_{\mathbb R} \phi(t,x,y)\bigg[\left(1+\frac A2-\frac 12(y-y_{opt}(t,x))^2-\int_{\mathbb R} n_{R_k}(t,x,z)\,dz\right)n_{R_k}(t,x,y)\\
&\quad +\gamma\left(1_{|y|\leq R}\iint_{\mathbb R^2} \Gamma_{A/2}\left(y-\frac{y_*+y_*'}2\right)\frac{n_{R_k}(t,x,y_*)n_{R_k}(t,x,y_*')}{\int_{\mathbb R} n_{R_k}(t,x,z)\,dz}\,dy_*\,dy_*'-n_{R_k}(t,x,y)\right)\bigg]\,dy\,dx\,dt\\
&=\int_{\mathbb R_+}\int_{\mathbb T^d}\int_{\mathbb R} \left(\partial_t\phi(t,x,y)+\Delta_x\phi(t,x,y)\right) n(t,x,y)\,dy\,dx\,dt\\
&\quad +\int_{\mathbb R_+}\int_{\mathbb T^d}\int_{\mathbb R} \phi(t,x,y)\bigg[\left(1+\frac A2-\frac 12(y-y_{opt}(t,x))^2-\int_{\mathbb R} n(t,x,z)\,dz\right)n(t,x,y)\\
&\quad +\gamma\left(1_{|y|\leq R}\iint_{\mathbb R^2} \Gamma_{A/2}\left(y-\frac{y_*+y_*'}2\right)\frac{n(t,x,y_*)n(t,x,y_*')}{\int_{\mathbb R} n(t,x,z)\,dz}\,dy_*\,dy_*'-n(t,x,y)\right)\bigg]\,dy\,dx\,dt.
\end{align*}
This equality and the regularity estimate \eqref{est:regnSIM} imply that $n$ is a solution of (SIM) in the sense of Definition~\ref{def:sol}. Estimate \eqref{est:uppery6} is a consequence of \eqref{est:uppery8} and the convergence of $n_{R_k}$ to $n$ in $L^1([0,T]\times \mathbb T^d\times \mathbb R)$, for any $T>0$.

}

\section{Macroscopic limits of solutions of (SIM)}\label{sec:macro}

\subsection{Tail estimates uniform in $\gamma>0$ for solutions of (SIM)}\label{subsec:highermoments}

In this section, we show that a bound on $\|Z\|_{L^\infty([0,\tau)\times \rb{\mathbb T^d})}$ implies a bound on $\|V\|_{L^\infty([0,\tau)\times \mathbb T^d)}$. \gr{This is the beginning of a bootstrap argument that will unfold in Proposition~4.5: we assume $\|Z\|_{L^\infty([0,\tau)\times \mathbb T^d)}\leq \kappa$ and will show that it implies a stronger estimate on $\|Z\|_{L^\infty([0,\tau)\times \mathbb T^d)}$. This bootstrap will imply a uniform estimate on $\|Z\|_{L^\infty([0,\tau)\times \rb{\mathbb T^d})}$, validating all the estimates obtained in Proposition 4.1, 4.3, and 4.4.} 

\begin{prop}\label{prop:moment4}
Let $\alpha>0$, $A>0$ and $\kappa>0$. There exist $\bar \gamma>0$, $C_\kappa>0$ and $\bar\tau_\kappa>0$ such that if $y_{opt},\,n^0$ satisfies Assumption~\ref{Assumption} and $\gamma>\bar \gamma$, then the following statement holds.

If the solution $n\in L^\infty(\mathbb R_+\times \mathbb T^d,L^1(\mathbb R))$ of (SIM) with initial condition $n^0$ and if it  satisfies $\|Z\|_{L^\infty([0,\tau)\times \mathbb T^d)}\leq \kappa$  for some $\tau\in[0,+\infty]$, then
\[\forall (t,x)\in[0,\tau+\bar\tau_\kappa)\times \mathbb T^d,\quad \int_{\mathbb R} |y|^4\frac{n(t,x,y)}{\int_{\mathbb R} n(t,x,z)\,dz}\,dy\leq C_\kappa.\]
\end{prop}
\begin{rem}\label{est:moment4T}
\ra{Note that the estimate proven in Proposition~\ref{prop:moment4} is uniform in $\gamma>0$ and is therefore not implied by \eqref{est:uppery6}. To obtain it, it is necessary to assume that this quantity is finite at $t=0$, but this is implied by the assumption $n^0(x,y)\leq \frac {C_0}{1+y^{10}}$ made in Assumption~\ref{Assumption}.}

Under the assumptions of the proposition above, \eqref{est:0T} and Proposition~\ref{prop:moment4} imply the following estimate, that will be useful on several occasions in the manuscript:
\begin{equation*}
\int_{\mathbb R} |y|^4 T(\tilde n(t,x,\cdot))(y)\,dy\leq C_\kappa
\end{equation*}

\end{rem}

\begin{proof}[Proof of Proposition~\ref{prop:moment4}]
The dynamics of $V$ is given by \eqref{eq:V}, and to estimate the last term of that equation, \ra{we take advantage of the fact that $T(\Gamma_{A}(\cdot-Z))=\Gamma_{A}(\cdot-Z)$} (see \eqref{eq:Maxwellian2A}), and Corollary~\ref{cor:Tanaka4}: for $(t,x)\in\mathbb R_+\times \mathbb T^d$,
\begin{align}
&\int_{\mathbb R} |y|^4 T(\tilde n(t,x,\cdot))(y)\,dy=W_4(T(\tilde n(t,x,\cdot)),\delta_0)^4\nonumber\\
&\quad \leq\left[W_4\left(T(\tilde n(t,x,\cdot)),T(\Gamma_{A}(Z(t,x)-\cdot))\right)+W_4\left(\Gamma_{A}(Z(t,x)-\cdot),\delta_0\right)\right]^4\nonumber\\
&\quad \leq \left[\frac 1{2^{1/4}} W_4(\tilde n(t,x,\cdot),\Gamma_{A}(Z(t,x)-\cdot))+W_4\left(\Gamma_{A}(Z(t,x)-\cdot),\delta_0\right)\right]^4\nonumber\\
&\quad \leq \left(\frac 1{2^{1/4}} W_4(\tilde n(t,x,\cdot),\delta_0)+2 W_4(\delta_0,\Gamma_{A}(Z(t,x)-\cdot))\right)^4\leq \left(\frac 1{2^{1/4}} W_4(\tilde n(t,x,\cdot),\delta_0)+2 Z(t,x)+C\right)^4\nonumber\\
&\quad \leq \frac 23 W_4^4(\tilde n(t,x,\cdot),\delta_0)+C\left(Z(t,x)^4+1\right),\label{est:0T}
\end{align}
for some constant $C>0$, thanks to a Young inequality. The last term of \eqref{eq:V} then satisfies
\begin{align*}
&\gamma\left(\int_{\mathbb R} |y|^4T(\tilde n(t,x,\cdot))(y)\,dy-V(t,x)\right) \leq \gamma\left(C\left(|Z(t,x)|^4+1\right)-\frac 13 V(t,x)\right).
\end{align*}
To estimate the second term on the right hand side of \eqref{eq:V}, we use a Cauchy-Schwarz inequality as follows
\begin{align*}
&\int_{\mathbb R} \left(V(t,x)-|y|^4\right) (y-y_{opt}(t,x))^2\tilde n(t,x,y)\,dy\leq V(t,x)\int_{\mathbb R} (y-y_{opt}(t,x))^2\tilde n(t,x,y)\,dy\\
&\quad\leq CV(t,x)\int \left(|y|^2+1\right)\tilde n(t,x,y)\,dy\leq C\left(1+\sqrt {V(t,x)}\right)V(t,x).
\end{align*}
Thanks to these estimates, \eqref{eq:V} becomes, for $(t,x)\in [0,\tau]\times \mathbb T^d$,
\begin{eqnarray}
\partial_t V(t,x)-\Delta_x V(t,x)&\leq& 2\frac{\nabla_x N(t,x)}{N(t,x)}\cdot\nabla_x V(t,x)+  C\left(1+\sqrt {V(t,x)}\right)V(t,x)\nonumber\\
&&+ \gamma\left(C\left(|Z(t,x)|^4+1\right)-\frac 13 V(t,x)\right).\label{est:V1}
\end{eqnarray}
Let 
\[\bar V:= \max\left(\|V(0,\cdot)\|_{L^\infty(\mathbb T^d)}, 7C\left((\kappa+1)^4+1\right)\right).\]
As soon as $\gamma\geq  C \left(1+\sqrt {\bar V}\right)$, we have 
\[C\left(1+\sqrt {\bar V}\right)\bar V+ \gamma\left(C\left((\kappa+1)^4+1\right)-\frac 13 \bar V\right)\leq 0,\]
and $\phi(t,x)\equiv \bar V$ satisfies $V(0,x)\leq \phi(0,x)$ for $x\in\mathbb T^d$, as well as
\begin{align*}
\partial_t \phi(t,x)-\Delta_x \phi(t,x)&\geq 2\frac{\nabla_x N(t,x)}{N(t,x)}\cdot\nabla_x \phi(t,x)+  C\left(1+\sqrt {\phi(t,x)}\right)\phi(t,x)\\
&\quad+\gamma\left(C\left((\kappa+1)^4+1\right)-\frac 13 \phi(t,x)\right),
\end{align*}
and $\phi$ is a super-solution of \eqref{est:0T} for $x\in\mathbb T^d$ and $t\in[0,\tau')$, where $\tau'$ is such that $\|Z\|_{L^\infty([0,\tau')\times\mathbb T^d)}\leq \kappa+1$. Note that the assumption $\|Z\|_{L^\infty([0,\tau)\times\mathbb T^d)}\leq\kappa$ implies $\kappa'\geq \kappa$. We may apply the parabolic comparison principle (see Remark~\ref{rem:comparaison}), and then, for $(t,x)\in [0,\tau')\times \mathbb T^d$,  $V(t,x)\leq \bar V$.

Thanks to \eqref{eq:Zmod}, we have
\begin{align*}
&\partial_t Z(t,x)-\Delta_x Z(t,x)-2\frac{\nabla_x N(t,x)\cdot\nabla_x Z(t,x)}{N(t,x)} \leq C\left(1+V(t,x)\right),
\end{align*}
and we define $\psi(t,x)=\kappa+C\left(1+\bar V\right)(t-\tau)$ which satisfies $Z(\tau,x)\leq \psi(\tau,x)$ for $x\in\mathbb T^d$, as well as
\begin{align*}
&\partial_t \psi(t,x)-\Delta_x \psi(t,x)-2\frac{\nabla_x N(t,x)}{N(t,x)}\cdot\nabla_x \psi(t,x) =C\left(1+\bar V\right)\geq C\left(1+V(t,x)\right),
\end{align*}
for $t\in[\tau,\tau')$. We may thus apply the parabolic comparison principle (see Remark~\ref{rem:comparaison}) to show $Z(t,x)\leq \kappa+C\left(1+\bar V\right)(t-\tau)$ for $(t,x)\in [\tau,\tau')\times \mathbb T^d$. We may then choose 
\[\tau'=\tau+\frac 1{C(1+\bar V)}=\tau+\frac 1{C\left(1+\max\left(\|V(0,\cdot)\|_{L^\infty(\mathbb T^d)}, 7C\left((\kappa+1)^4+1\right)\right)\right)}.\]

\end{proof}

\subsection{Regularity estimates uniform in $\gamma>0$ for $N$ and $Z$}\label{subsec:RegularityNZ}

\rb{In the proposition below, we prove some regularity estimates on $N$ and $Z$ that are uniform in $\gamma>0$. It is this uniformity that sets them apart from the regularity results obtained in Section~\ref{subsec:regNV}. }

\begin{prop}\label{prop:regularityNZ}
Let $\alpha>0$, $A>0$ and $\kappa>0$. There exist $\bar \gamma>0$, $\theta\in(0,1)$ and $C_\kappa>0$ such that if $y_{opt},\,n^0$ satisfies Assumption~\ref{Assumption} and $\gamma>\bar \gamma$, then the following statement holds.

\rb{Let $n\in L^\infty(\mathbb R_+\times \mathbb T^d,L^1(\mathbb R))$ the solution of (SIM) with initial condition $n^0$, and $N$ and $Z$ are defined by \eqref{NZ}. If} it  satisfies $\|Z\|_{L^\infty([0,\tau)\times \mathbb T^d)}\leq \kappa$ for some $\tau\in(0,+\infty]$, then for any 
$s,t\in[0,\tau)$ and $x,y\in \mathbb T^d$,
\[\frac {|Z(t,x)-Z(s,y)|}{(|t-s|+|x-y|)^\theta}+\frac {|N(t,x)-N(s,y)|}{(|t-s|+|x-y|)^\theta}\leq C_{\kappa}.\]
\rb{Moreover},
\begin{equation}\label{est:gradNN}
\left\|\frac{\nabla_x N}N\right\|_{L^{d+3}([0,\tau)\times \mathbb T^d)}\leq C_\kappa. 
\end{equation}
\end{prop}
\rb{Note that $C_\kappa$ in this statement is independent from $\tau$. This is a consequence of Step 3 of the proof. It relies on the application of a Harnack inequality (that was used to obtain (35)), and of (39), that is a corrollary of the Harnack inequality. This uniformity in $\tau$ is essential for to conclude the bootstrap argument described at the beginning of Section~\ref{subsec:highermoments}}

\begin{proof}[Proof of Proposition~\ref{prop:regularityNZ}] Let $\bar\gamma>0$ as in Proposition~\ref{prop:moment4}.

\textbf{Step 1: Lower bound on $N(t,x)$} \\
Since $\|Z\|_{L^\infty([0,\tau)\times \mathbb T^d)}\leq \kappa$, Proposition~\ref{prop:moment4} implies that $\int |y|^4\tilde n(t,x,y)\,dy$ is uniformly bounded on $[0,\tau)\times \mathbb T^d$, and there exists a constant $C_\kappa>0$ such that for $(t,x)\in[0,\tau)\times \mathbb T^d$,
\begin{equation}\label{est:CN}
\left|\left[1+\frac A2-N(t,x)\right]N(t,x)-\frac 12 \int_{\mathbb R} (y-y_{opt}(t,x))^2n(t,x,y)\,dy\right|\leq C_\kappa N(t,x),
\end{equation}
where we have also used the uniform bound on $N$ provided by Proposition~\ref{prop:existSIM}.  Thanks to \eqref{est:CN} and the comparison principle used in \eqref{eq:Nmod}, for $t\in[0,1]\cap [0,\tau)$,
\begin{equation}\label{est:lowerN1}
N(t,x)\geq e^{-C_\kappa t}\inf_{\mathbb T^d}N(0,\cdot)\geq C_\kappa,
\end{equation}
thanks to Assumption~\ref{Assumption}. 
Estimate \eqref{est:CN} also provides a uniform bound on the coefficients of~\eqref{eq:Nmod}, we can apply the Harnack inequality for $t\in [0,\tau)\setminus [0,1]$ (see \cite{Moser}, or Theorem~3 in \cite{Aronson}): there exists $C_{\kappa}>0$ such that for any $t\in [0,\tau)\setminus [0,1]$,
\begin{equation*}
\max_{(s,x)\in [t-3/4,t-1/2]\times \mathbb T^d} N(s,x)\leq C_{\kappa}\min_{(s,x)\in [t-1/3,t]\times\mathbb T^d}N(s,x).
\end{equation*}
Since $\partial_t N-\Delta_x N\leq \rb{(1+A/2)} N$, we may consider the super-solution 
\[(s,x)\mapsto \left(\max_{x\in \mathbb T^d} N(t-1/2,x)\right)e^{\rb{(1+A/2)}(s-(t-1/2))},\]
 and the comparison principle implies, for $t\in [0,\tau)\setminus [0,1]$,
\begin{equation}\label{est:lowerN2}
\max_{(s,x)\in [t-3/4,t]\times \mathbb T^d} N(s,x)\leq C_{\kappa}\min_{ (s,x)\in [t-1/3,t]\times \mathbb T^d} N(s,x).
\end{equation}

\noindent \textbf{Step 2: $L^{d+3}$ estimate on $\frac{\nabla_x N(t,x)}{N(t,x)}$ for $t\in[0,1]$}\\

We notice that for  $(t,x)\in(-\infty,\tau)\times\mathbb R$, $N(t,x)=\left(N(0,x)+\mathcal N(t,x)\right)1_{t\geq 0}$, where $\mathcal N$ is a solution of
\begin{equation}\label{eq:mathcalN}
\partial_t\mathcal N(t,x)-\Delta_x\mathcal N(t,x)=\mu_N(t,\pi(x))1_{t\geq 0},\quad (t,x)\in(-\infty,\tau)\times\mathbb R^d,
\end{equation}
where $\pi(x)$ is the standard projection of $x\in\mathbb R^d$ on $\mathbb T^d$, and
\[\mu_N(t,x)=\Delta_x N^0(x)+\left(1+\frac A2-\frac 12 \int_{\mathbb R} (y-y_{opt}(t,x))^2\tilde n(t,x,y)\,dy-N(t,x)\right)N(t,x).\]
Note that $\mathcal N(t,\cdot)\equiv 0$ for $t\leq 0$. Thanks to \eqref{est:CN} and Assumption~\ref{Assumption}, we have  $\|\mu_N\|_{L^\infty([0,\tau)\times \mathbb T^d)}<C_\kappa$, and we can  apply Theorem~7.22 of \cite{Liebermann} to obtain
\begin{equation}\label{est:grad}
\|\partial_{x}\mathcal N\|_{L^{d+3}([t-1/4,t]\times \mathbb T^d)}\leq C_{\kappa}\left(\|\mathcal N\|_{L^{d+3}([t-1/3,t]\times \mathbb T^d)}+1\right),
\end{equation}
for any $t\in\mathbb R$. For $t\in[0,1]$, since $N$ is uniformly bounded and thanks to the lower estimate \eqref{est:lowerN1}, we obtain
\begin{equation}\label{est:L5b}
\left\|\frac{\nabla_x N}N\right\|_{L^{d+3}([0,1]\times \mathbb T^d)}\leq C_\kappa. 
\end{equation}

\noindent \textbf{Step 3: $L^{d+3}$ estimate on $\frac{\nabla_x N(t,x)}{N(t,x)}$ for $t\in[0,\tau)\setminus[0,1]$}\\

The argument here is similar to the one developed for step~2, but on equation \eqref{eq:Nmod} instead of \eqref{eq:mathcalN}. Theorem~7.22 of \cite{Liebermann} applied to  \eqref{eq:Nmod} implies that for $t\geq 1$,
\begin{equation}\label{est:grad-bis}
\|\nabla_{x}N\|_{L^{d+3}([t-1/4,t]\times \mathbb T^d)}\leq C_{\kappa}\| N\|_{L^{d+3}([t-1/3,t]\times \mathbb T^d)},
\end{equation}
which we combine to \eqref{est:lowerN2} to obtain, for $t\geq 1$,
\begin{align}
\left\|\frac{\nabla_x N}N\right\|_{L^{d+3}(([t-1/4,t+1/4]\cap[0,\tau))\times \mathbb T^d)}&\rb{\leq \frac {\left\|\nabla_x N\right\|_{L^{d+3}(([t-1/4,t+1/4]\cap[0,\tau))\times \mathbb T^d)}}{\min_{(s,x)\in [t-1/3,t]\times\mathbb T^d}N(s,x)}}\nonumber\\
&\leq C_{\kappa}\frac{\|N\|_{L^{d+3}(([t-1/3,t])\times \mathbb T^d)}}{\|N\|_{L^\infty(([t-3/4,t])\times \mathbb T^d)}}\leq C_\kappa,\label{est:L5}
\end{align}
Since  $\|N\|_{L^{d+3}(([t-1/3,t])\times \mathbb T^d)}\leq \|N\|_{L^\infty(([t-3/4,t])\times \mathbb T^d)}$.

\medskip

\noindent \textbf{Step 4: Regularity of $N$ and $Z$}\\

Just as we have done for $\mathcal N(t,x)=N(t,x)-N^0(x)$ (see \eqref{eq:mathcalN}), we can define $\mathcal Z=\left(Z(t,x)-Z(0,x)\right)1_{t\geq 0}$, solution of
\begin{equation*}
\partial_t\mathcal Z(t,x)-\Delta_x\mathcal Z(t,x)=2\frac{\nabla_x N(t,x)}{N(t,x)}\cdot\nabla_x\mathcal Z(t,x)+\mu_Z(t,\pi(x))1_{t\geq 0},\quad (t,x)\in(-\infty,\tau)\times\mathbb R^d,
\end{equation*}
where $\|\mu_Z\|_{L^\infty([0,\tau)\times \mathbb T^d)}<C_\kappa$ thanks to Proposition~\ref{prop:moment4} and Assumption~\ref{Assumption}, and $\frac{\nabla_x N}{N}$ satisfies \eqref{est:L5b}, \eqref{est:L5}. This equation then has the structure of equation (5) in \cite{Aronson} and we can apply  satisfy Theorem~4 from that reference to obtain a H\"older estimate on  $\mathcal Z$ to prove the Hölder continuity of $\mathcal Z$. This theorem cxan also be applied to  $\mathcal N$ (since all its coefficients are bounded), which concludes the proof of the proposition.

\end{proof}

\subsection{Distance of solutions of (SIM) to local Maxwellians}\label{subsec:Wasserstein}

\begin{prop}\label{prop:Wassersteincontraction}
Let $\alpha>0$, $A>0$ and $\kappa>0$. There exist $\bar \gamma>0$, $\theta\in (0,1)$ and $C_\kappa>0$  such that if $y_{opt},\,n^0$ satisfies Assumption~\ref{Assumption} and $\gamma>\bar \gamma$, then the following statement holds.

\rb{Let $n\in L^\infty(\mathbb R_+\times \mathbb T^d,L^1(\mathbb R))$  the solution of (SIM) with initial condition $n^0$, and $Z$ defined by \eqref{NZ}. If} it satisfies $\|Z\|_{L^\infty([0,\tau)\times \mathbb T^d)}\leq \kappa$ for some $\tau\in(0,+\infty]$, then
\begin{equation}\label{est:Wass}
\forall t\in\left[C_\kappa\frac {\ln\gamma}{\gamma},\tau\right),\quad \max_{x\in \mathbb T^d}W_2^2\Big(\tilde n(t,x,\cdot),\Gamma_{A}(\cdot-Z(t,x))\Big)\leq \frac{C_{\kappa}}{\gamma^\theta},
\end{equation}
where $\tilde n$ is given by \eqref{def:tilden}\rb{ and} $\Gamma_{A}$ is defined by \eqref{def:Gamma}.
\end{prop}
\rb{Note that in this statement, the constant $C_\kappa>0$ is independent from $\gamma\geq \bar\gamma$, so that after a boundary layer $(0,C_\kappa\ln\gamma/\gamma)$, the functions $y\mapsto \tilde n(t,x,y)$ becomes close to Gaussian distributions in the trait space, for any $x\in\mathbb T^d$.}

\begin{proof}[Proof of Proposition~\ref{prop:Wassersteincontraction}]

In this proof, we use the linear problems and estimates presented in Section~\ref{subsec:tecnical} of the Appendix. \gr{In particular, we define $(t,x)\mapsto \phi_{s,z,y}(t,x)$ as the solution of}
\begin{equation}\label{eq:phi}
\left\{
\begin{array}{l}
\partial_t \phi_{s,z,y}(t,x)-\Delta_x \phi_{s,z,y}(t,x)\\
\qquad = 2\frac{\nabla_x N(t,x)}{N(t,x)}\cdot\nabla_x \phi_{s,z,y}(t,x)-\frac 12(y-y_{opt}(t,x))^2\phi_{s,z,y}(t,x), \;(t,x)\in [s,\tau)\times \mathbb T^d,\\
\phi_{s,z,y}(s,x)=\delta_z(x),\; x\in \mathbb T^d.
\end{array}
\right.
\end{equation}
\gr{This solution exists since $\frac{\nabla_x N(t,x)}{N(t,x)}$ is a continuous function (see Section~\ref{subsec:regNV}) and $y$ is a parameter here, therefore all coefficients of this linear parabolic equation are bounded and continuous. Alternatively, it is possible to build explicit solutions from a heat equation, we refer to \eqref{eq:barphi} for this argument}. For $t\geq 0$,  we can use a Duhamel formula to  write $\tilde n$ (we recall that $\tilde n$ satisfies \eqref{eq:tildenmod}) as follows
\begin{align*}
&\tilde n(t,x,y)= e^{-\gamma t}\int_{\mathbb R} \phi_{0,z,y}(t,x)\tilde n(0,z,y)\,dz\\
&\quad +\frac 12 \int_0^{t}e^{-\gamma(t-s)}\int_{\mathbb R}\phi_{s,z,y}(t,x) \tilde n(s,z,y)\left(\int_{\mathbb R} (w-y_{opt}(s,z))^2\tilde n(s,z,w)\,dw\right)\,dz\,ds\\
&\qquad+\gamma \int_0^{t}e^{-\gamma(t-s)}\int_{\mathbb R}\phi_{s,z,y}(t,x)T(\tilde n(s,z,\cdot))(y)\,dz\,ds.
\end{align*}
Since $\tilde n(t,x,\cdot)$ is a probability measure, the $y-$integral of the right hand size of the equation above sums up to one \ra{and the right hand side can be seen as a convex combinations of three probability distributions. } \rb{The convexity properties of the squared Wasserstein distance $W_2^2$, that we detail in Section~\ref{subsec:defWasserstein} in the Appendix (see \eqref{eq:convexity}), then implies:}
\begin{align}
&W_2^2\left(\tilde n(t,x,\cdot),\Gamma_{A}(\cdot-Z(t,x))\right) \leq e^{-\gamma t}\int_{\mathbb R} \left(\int_{\mathbb R}  \phi_{0,z,y}(t,x)\tilde n(0,z,y)\,dy\right)\nonumber\\
&\quad W_2^2\left(\frac{\phi_{0,z,\cdot}(t,x)\tilde n(0,z,\cdot)}{\int_{\mathbb R}  \phi_{0,z,y}(t,x)\tilde n(0,z,y)\,dy},\Gamma_{A}(\cdot-Z(t,x))\right)\,dz\nonumber\\
&\quad +\frac 12 \int_0^te^{-\gamma(t-s)}\int_{\mathbb R} \left(\int_{\mathbb R}\phi_{s,z,y}(t,x) \tilde n(s,z,y)\,dy\right)\left(\int_{\mathbb R} (w-y_{opt}(s,z))^2\tilde n(s,z,w)\,dw\right)\nonumber\\
&\qquad W_2^2\left(\frac{\phi_{s,z,\cdot}(t,x) \tilde n(s,z,\cdot)}{\int_{\mathbb R}\phi_{s,z,y}(t,x) \tilde n(s,z,y)\,dy},\Gamma_{A}(\cdot-Z(t,x))\right)\,dz\,ds\nonumber\\
&\quad+\gamma\int_0^t  e^{-\gamma(t-s)}\int_{\mathbb R} \left(\int_{\mathbb R} \phi_{s,z,y}(t,x)T(\tilde n(s,z,\cdot))(y)\,dy\right)\nonumber\\
&\quad W_2^2\left(\frac{\phi_{s,z,\cdot}(t,x)T(\tilde n(s,z,\cdot))}{\int_{\mathbb R} \phi_{s,z,y}(t,x)T(\tilde n(s,z,\cdot))(y)\,dy}, T\left(\Gamma_{A}(\cdot-Z(t,x))\right)\right)\,dz\,ds.\label{est:Was}
\end{align}
Note that we have used that $\Gamma_{A}(\cdot-Z(t,x))$ is  a fixed point for $T$ (see  \eqref{eq:Maxwellian2A}). To estimate the first two terms on the right hand side of \eqref{est:Was}, a rough estimate is sufficient: for any $(s,z)\in[0,\infty)\times \mathbb T^d$ and $(t,x)\in[s,\infty)\times \mathbb T^d$,
\begin{align}
&W_2^2\left(\frac{\phi_{s,z,\cdot}(t,x) \tilde n(s,z,\cdot)}{\int_{\mathbb R}\phi_{s,z,y}(t,x) \tilde n(s,z,y)\,dy},\Gamma_{A}(\cdot-Z(t,x))\right)\nonumber\\
&\leq \left(W_2\left(\frac{\phi_{s,z,\cdot}(t,x) \tilde n(s,z,\cdot)}{\int_{\mathbb R}\phi_{s,z,y}(t,x) \tilde n(s,z,y)\,dy},\delta_0\right)+W_2\left(\delta_0,\Gamma_{A}(\cdot-Z(t,x))\right)\right)^2\nonumber\\
&\quad \leq 2\int_{\mathbb R} |y|^2\frac{\phi_{s,z,y}(t,x) \tilde n(s,z,y)}{\int\phi_{s,z,y'}(t,x) \tilde n(s,z,y')\,dy'}\,dy+2\int_{\mathbb R} |y|^2\Gamma_{A}(y-Z(t,x))\,dy\leq C _\kappa ,\label{est:rough0}
\end{align}
where the final estimate follows from Section~\ref{subsec:tecnical} in the Appendix: if we define $R$ by \eqref{def:R} and $R'$ as in \eqref{def:rho} (note that $|R'|\leq C_\kappa$), then \eqref{est:phi1}, \eqref{est:phi2} and Proposition~\ref{prop:moment4} imply
\begin{align}
&\int_{\mathbb R} |y|^2\frac{\phi_{s,z,y}(t,x) \tilde n(s,z,y)}{\int_{\mathbb R}\phi_{s,z,y'}(t,x) \tilde n(s,z,y')\,dy'}\,dy\nonumber\\
&\quad \leq\int_{[-R',R']^c} |y|^2\frac{\left(\min_{|\tilde y|\leq R}\phi_{s,z,\tilde y}(t,x)\right)\tilde n(s,z,y)}{\int_{-R}^R\phi_{s,z,y'}(t,x) \tilde n(s,z,y')\,dy'}\,dy+(R')^2 \int_{-R'}^{R'} \frac{\phi_{s,z,y}(t,x) \tilde n(s,z,y)}{\int_{\mathbb R}\phi_{s,z,y'}(t,x) \tilde n(s,z,y')\,dy'}\,dy\nonumber\\
&\quad \leq \int_{\mathbb R} |y|^2\frac{\tilde n(s,z,y)}{1/2}\,dy+(R')^2\leq C_\kappa.\label{est:rough}
\end{align}
We repeat the estimate \eqref{est:rough0} (using additionally the estimate of Remark~\ref{est:moment4T}) to control the last term of \eqref{est:Was} for $s\leq t-\varepsilon$, for some $\varepsilon>0$ that we will define later on. We obtain then, for $s\leq t-\varepsilon$,
\begin{equation}\label{est:rough1}
W_2^2\left(\frac{\phi_{s,z,\cdot}(t,x)T(\tilde n(s,z,\cdot))}{\int_{\mathbb R} \phi_{s,z,y}(t,x)T(\tilde n(s,z,\cdot))(y)\,dy}, T\left(\Gamma_{A}(\cdot-Z(t,x))\right)\right)\leq C_\kappa.
\end{equation}
 For $s\in[t-\varepsilon,t]$, we need a more precise estimate, which we will obtain with the following coupling $\pi$. We define $\bar \phi_{s,z}(t,x)$by \eqref{eq:barphi}, and
\begin{eqnarray*}
\pi(y_1,y_2)&=&\frac{\phi_{s,z,y_1}(t,x)}{\bar \phi_{s,z}(t,x)} T(\tilde n(s,z,\cdot))(y_1)\delta_{y_1=y_2}\\
&&+\left(1-\frac{\phi_{s,z,y_1}(t,x)}{\bar \phi_{s,z}(t,x)}\right)T(\tilde n(s,z,\cdot))(y_1)\frac{\phi_{s,z,y_2}(t,x) T(\tilde n(s,z,\cdot))(y_2)}{\int_{\mathbb R}\phi_{s,z,y'}(t,x) T(\tilde n(s,z,\cdot))(y')\,dy'}.
\end{eqnarray*}
$\pi$ is then a probability measure on $\mathbb R\times \mathbb R$ (note that $\phi_{s,z,y_1}(t,x)\leq \bar \phi_{s,z}(t,x)$, thanks to \eqref{est:barphi1}), with marginals 
\[\pi|_1(y_1)=T(\tilde n(s,z,\cdot))(y_1)\quad\textrm{ and }\quad\pi|_2(y_2)=\frac{\phi_{s,z,y_2}(t,x) T(\tilde n(s,z,\cdot))(y_2)}{\int_{\mathbb R}\phi_{s,z,y'}(t,x) T(\tilde n(s,z,\cdot))(y')\,dy'}.\]
Then,
\begin{align}
&W_2^2\left(\frac{\phi_{s,z,y}(t,x) T(\tilde n(s,z,\cdot))(y)}{\int_{\mathbb R}\phi_{s,z,y'}(t,x) T(\tilde n(s,z,\cdot))(y')\,dy'},T(\tilde n(s,z,\cdot))\right)\leq \iint_{\mathbb R^2}|y_1-y_2|^2\,d\pi(y_1,y_2)\nonumber\\
&\quad\leq \iint_{\mathbb R^2}|y_1-y_2|^2\left(1-\frac{\phi_{s,z,y_1}(t,x)}{\bar \phi_{s,z}(t,x)}\right)T(\tilde n(s,z,\cdot))(y_1)\frac{\phi_{s,z,y_2}(t,x) T(\tilde n(s,z,\cdot))(y_2)}{\int_{\mathbb R}\phi_{s,z,y'}(t,x) T(\tilde n(s,z,\cdot))(y')\,dy'}\,dy_1\,dy_2\nonumber\\
&\quad \leq 2\iint_{\mathbb R^2}\left(y_1^2+y_2^2\right)\left(1-\frac{\phi_{s,z,y_1}(t,x)}{\bar \phi_{s,z}(t,x)}\right)T(\tilde n(s,z,\cdot))(y_1)\frac{\phi_{s,z,y_2}(t,x) T(\tilde n(s,z,\cdot))(y_2)}{\int_{\mathbb R}\phi_{s,z,y'}(t,x) T(\tilde n(s,z,\cdot))(y')\,dy'}\,dy_1\,dy_2\nonumber\\
&\quad \leq 2\int_{\mathbb R} y_1^2\left(1-\frac{\phi_{s,z,y_1}(t,x)}{\bar \phi_{s,z}(t,x)}\right)T(\tilde n(s,z,\cdot))(y_1)\,dy_1\nonumber\\
&\qquad +2\left(1-\frac {\int_{\mathbb R}\phi_{s,z,y'}(t,x) T(\tilde n(s,z,\cdot))(y')\,dy'}{\bar\phi_{s,z}(t,x)}\right)\int_{\mathbb R} y_2^2\frac{\phi_{s,z,y_2}(t,x) T(\tilde n(s,z,\cdot))(y_2)}{\int_{\mathbb R}\phi_{s,z,y'}(t,x) T(\tilde n(s,z,\cdot))(y')\,dy'}\,dy_2.\label{est:distW2}
\end{align}
We estimate below  the first integral term of \eqref{est:distW2}, where $s<t$. We estimate the integral by separating it into two integral terms. The first integral one can then be controlled thanks to a Chebyshev's inequality (we recall  Remark~\ref{est:moment4T}), while we use the estimate \eqref{est:barphi1}, derived in the Appendix, to estimate the second integral term:
\begin{align}
&\int_{\mathbb R} y_1^2\left(1-\frac{\phi_{s,z,y_1}(t,x)}{\bar \phi_{s,z}(t,x)}\right)T(\tilde n(s,z,\cdot))(y_1)\,dy_1\leq \int_{|y_1|\geq (t-s)^{-1/3}}y_1^2T(\tilde n(s,z,\cdot))(y_1)\,dy_1\nonumber\\
&\quad + \int_{|y_1|\leq (t-s)^{-1/3}}y_1^2\left(1-e^{-(t-s)\frac 12(y_1+\mathcal O(1))^2}\right)T(\tilde n(s,z,\cdot))(y_1)\,dy_1\nonumber\\
&\quad\leq (t-s)^{2/3}\int_{|y_1|\geq (t-s)^{-1/3}}y_1^4T(\tilde n(s,z,\cdot))(y_1)\,dy_1\nonumber\\
&\qquad +\left(1-e^{-(t-s)^{2/3}}\right)\int_{|y_1|\leq (t-s)^{-1/3}}y_1^2T(\tilde n(s,z,\cdot))(y_1)\,dy_1\nonumber\\
&\quad \leq C_\kappa (t-s)^{2/3}+C_\kappa\left(1-e^{-(t-s)^{2/3}}\right) \leq C_\kappa (t-s)^{2/3},\label{est:truc}
\end{align}
provided $t-s>0$ is small enough. We estimate the last term of \eqref{est:distW2} as follows, provided $|t-s|$ is small enough:
\begin{align}
&2\left(1-\frac {\int_{\mathbb R}\phi_{s,z,y'}(t,x) T(\tilde n(s,z,\cdot))(y')\,dy'}{\bar\phi_{s,z}(t,x)}\right)\int_{\mathbb R} y_2^2\frac{\phi_{s,z,y_2}(t,x) T(\tilde n(s,z,\cdot))(y_2)}{\int_{\mathbb R}\phi_{s,z,y'}(t,x) T(\tilde n(s,z,\cdot))(y')\,dy'}\,dy_2\nonumber\\
&\quad\leq \left(\int_{\mathbb R}\left(1-\frac{\phi_{s,z,y'}(t,x)}{\bar\phi_{s,z}(t,x)}\right) T(\tilde n(s,z,\cdot))(y')\,dy'\right) C_\kappa\leq C_\kappa (t-s)^{2/3},\label{est:trac}
\end{align}
where the first inequality is justified by \eqref{est:rough}, and the second inequality can be obtained through the argument performed in \eqref{est:truc} (with $1$ instead of $y_1^2$). Thanks to \eqref{est:truc} and \eqref{est:trac}, the estimate \eqref{est:distW2} becomes
\begin{equation*}
W_2^2\left(\frac{\phi_{s,z,y}(t,x) T(\tilde n(s,z,\cdot))(y)}{\int_{\mathbb R}\phi_{s,z,y'}(t,x) T(\tilde n(s,z,\cdot))(y')\,dy'},T(\tilde n(s,z,\cdot))\right)\leq C_\kappa (t-s)^{2/3}.
\end{equation*}
This estimate combined to the regularity estimates on $N$ and $Z$ obtained in Proposition~\ref{prop:regularityNZ} lead to
\begin{align}
&W_2\left(\frac{\phi_{s,z,\cdot}(t,x) T(\tilde n(s,z,\cdot))}{\int_{\mathbb R}\phi_{s,z,y}(t,x) \tilde n(s,z,y)\,dy},T(\Gamma_{A}(\cdot-Z(t,x)))\right)\nonumber\\
&\quad\leq W_2\left(\frac{\phi_{s,z,\cdot}(t,x) T(\tilde n(s,z,\cdot))}{\int_{\mathbb R}\phi_{s,z,y}(t,x) \tilde n(s,z,y)\,dy},T(\tilde n(s,z,\cdot))\right)\nonumber\\
&\qquad + W_2\left(T(\tilde n(s,z,\cdot)),T(\Gamma_{A}(\cdot-Z(s,z)))\right)+|Z(t,x)-Z(s,z)|\nonumber\\
&\leq W_2^2\Big(T(\tilde n(s,z,\cdot)),T(\Gamma_{A}(\cdot-Z(s,z)))\Big)+C_{\kappa}|t-s|^\theta+C_{\kappa}|x-z|^\theta,\label{est:refined}
\end{align}
for some $\theta\in (0,1)$, provided $\gamma>0$ is large enough. We are now ready to consider the original estimate \eqref{est:Was}: thanks to \eqref{est:rough0}, \eqref{est:rough1} and \eqref{est:refined}, the estimate \eqref{est:Was} implies\begin{align*}
&W_2^2\left(\tilde n(t,x,\cdot),\Gamma_{A}(\cdot-Z(t,x))\right)\leq  e^{-\gamma t}\int_{\mathbb R} \left(\int_{\mathbb R}  \phi_{0,z,y}(t,x)\tilde n(0,z,y)\,dy\right)C_\kappa\,dz \\
&\qquad +\frac {C_\kappa}2 \int_0^te^{-\gamma(t-s)}\int_{\mathbb R} \left(\int_{\mathbb R}\phi_{s,z,y}(t,x) \tilde n(s,z,y)\left(\int_{\mathbb R} (w-y_{opt}(s,z))^2\tilde n(s,z,w)\,dw\right)\,dy\right)\,dz\,ds\\
&\qquad +\gamma\int_0^{t-\varepsilon}  e^{-\gamma(t-s)}\int_{\mathbb R} \left(\int_{\mathbb R} \phi_{s,z,y}(t,x)T(\tilde n(s,z,\cdot))(y)\,dy\right) C_\kappa \,dz\,ds\\
&\qquad +\gamma\int_{t-\varepsilon}^t  e^{-\gamma(t-s)}\int_{\mathbb R} \left(\int_{\mathbb R} \phi_{s,z,y}(t,x)T(\tilde n(s,z,\cdot))(y)\,dy\right)\\
&\qquad \left(W_2^2\Big(T(\tilde n(s,z,\cdot)),T(\Gamma_{A}(\cdot-Z(s,z)))\Big)+C_{\kappa}|t-s|^\theta+C_{\kappa}|x-z|^\theta\right)\,dz\,ds.
\end{align*}
We can now use the estimate \eqref{est:barphi1} (and Proposition~\ref{prop:moment4}) to obtain
\begin{align*}
&W_2^2\left(\tilde n(t,x,\cdot),\Gamma_{A}(\cdot-Z(t,x))\right)\leq  e^{-\gamma t}\left(\int_{\mathbb R} \bar \phi_{0,z}(t,x)\,dz\right)C_\kappa \\
&\qquad +\frac 12 \int_0^te^{-\gamma(t-s)}\left(\int_{\mathbb R}\bar \phi_{s,z}(t,x)\,dz\right) C_\kappa \,ds +\gamma\int_0^{t-\varepsilon}  e^{-\gamma(t-s)}\left(\int_{\mathbb R}\bar \phi_{s,z}(t,x)\,dz\right)  C_\kappa \,dz\,ds\\
&\qquad +\gamma\int_{t-\varepsilon}^t  e^{-\gamma(t-s)}\left(\int_{\mathbb R}\bar \phi_{s,z}(t,x)\,dz\right) \max_{z\in \mathbb T^d}W_2^2\Big(T(\tilde n(s,z,\cdot)),T(\Gamma_{A}(\cdot-Z(s,z)))\Big)\,ds\\
&\qquad +\gamma\int_{t-\varepsilon}^t  e^{-\gamma(t-s)}\int_{\mathbb R} \bar \phi_{s,z}(t,x)\left(C_{\kappa}|t-s|^\theta+C_{\kappa}|x-z|^\theta\right)\,dz\,ds.
\end{align*}
Thanks to \eqref{est:barphi2}, we have $\int \bar \phi_{0,z}(t,x)\,dz=1$, while \eqref{est:barphi3} shows that $\int \bar \phi_{s,z}(t,x)|x-z|^\theta\leq C_{\kappa}|t-s|^{\frac\theta 2}$. Then,
\begin{align}
&W_2^2\Big(\tilde n(t,x,\cdot),\Gamma_{A}(\cdot-Z(t,x))\Big)\leq  C_\kappa e^{-\gamma t}+\frac{C_\kappa}{\gamma}+C_\kappa e^{-\gamma\varepsilon}+\frac{C_{\kappa}}{\gamma^{\theta/2}}\nonumber\\
&\qquad +\gamma\int_{t-\varepsilon}^t  e^{-\gamma(t-s)}\max_{z\in \mathbb T^d}W_2^2\Big(T(\tilde n(s,z,\cdot)),T(\Gamma_{A}(\cdot-Z(s,z)))\Big)\,ds,\label{est:abc}
\end{align}
where we have used the change of variable $\tilde s=\gamma(t-s)$ to show
\[\gamma\int_{t-\varepsilon}^t  e^{-\gamma(t-s)}(t-s)^{\theta/2}\,ds=\int_0^{\gamma\varepsilon}e^{-s}\left(\frac s\gamma\right)^{\theta/2}\,ds\leq C{\gamma^{\theta/2}}.\]
Since the right hand side of \eqref{est:abc} is independent of $x\in\mathbb T^d$, we can consider the maximum over that variable. If moreover we apply the Tanaka inequality (see Theorem~\ref{thm:Tanaka}), we obtain
\begin{align*}
&I(t)\leq  C_\kappa e^{-\gamma t}+\frac{C_\kappa}{\gamma}+C_\kappa e^{-\gamma\varepsilon}+\frac{C_{\kappa}}{\gamma^{\theta/2}} +\frac{\gamma}2\int_{t-\varepsilon}^t  e^{-\gamma(t-s)}I(s)\,ds,
\end{align*}
where $I(s):=\max_{x\in \mathbb T^d}W_2^2\Big(\tilde n(s,x,\cdot),\Gamma_{A}(\cdot-Z(s,x))\Big)$.
 Thanks to a Gr\"onwall inequality (see e.g. \cite{Dragomir}),
\begin{eqnarray*}
I(t)&\leq&  C_\kappa e^{-\gamma t}+\frac{C_\kappa}{\gamma}+C_\kappa e^{-\gamma\varepsilon}+\frac{C_{\kappa}}{\gamma^{\theta/2}}\\
&&+\frac{\gamma}2e^{-\gamma t}\int_{t-\varepsilon}^t\left( C_\kappa e^{-\gamma s}+\frac{C_\kappa}{\gamma}+C_\kappa e^{-\gamma\varepsilon}+\frac{C_{\kappa}}{\gamma^{\theta/2}}\right) e^{\gamma s} e^{\frac{\gamma}2(t-s)}\,ds\\
& \leq& C_\kappa e^{-\gamma t}+\frac{C_\kappa}{\gamma}+C_\kappa e^{-\gamma\varepsilon}+\frac{C_{\kappa}}{\gamma^{\theta/2}}+
\left( C_\kappa e^{-\gamma(t-\varepsilon/2)}+\frac{C_\kappa}{\gamma}+C_\kappa e^{-\gamma\varepsilon}+\frac{C_{\kappa}}{\gamma^{\theta/2}}\right).
\end{eqnarray*}
We can chose $\varepsilon:=\frac {\theta\ln\gamma}{2\gamma}$ to obtain
\[I(t)\leq C_\kappa e^{-\gamma t}\gamma^{\theta/4}+\frac{C_{\kappa}}{\gamma^{\theta/2}},\]
so that finally, for any $\gamma>0$ large enough,
\[\max_{t\in[ \theta\ln\gamma/\gamma,\tau)} I(t)\leq \frac{C_{\kappa}}{\gamma^{\theta/2}}.\]
The result follows (note that we need to define a slightly different parameter $\theta$: $\tilde \theta:=\theta /2>0$).
\end{proof}

\subsection{Macroscopic limit from (SIM) to (KBM)}\label{subsec:refinedLinfty}

Let $\alpha>0$, $A>0$ and $\kappa>0$. There exist $\bar \gamma>0$, $C_\kappa>0$, and $\theta\in (0,1)$  such that if $y_{opt},\,n^0$ satisfies Assumption~\ref{Assumption} and $\gamma>\bar \gamma$, then the following statement holds.

If $n\in L^\infty(\mathbb R_+\times \mathbb T^d,L^1(\mathbb R)$ is the solution of (SIM) with initial condition $n^0$ and if it satisfies $\|Z\|_{L^\infty([0,\tau)\times \mathbb T^d)}\leq \kappa$, then

\begin{prop}\label{prop:Linftyboundwithoutalpha}
Let $\alpha>0$, $A>0$ and $\kappa>0$. There exist $\bar \gamma>0$ such that if $y_{opt},\,n^0$ satisfies Assumption~\ref{Assumption} and $\gamma>\bar \gamma$, then the solution $n\in L^\infty(\mathbb R_+\times \mathbb T^d,L^1(\mathbb R))$ of (SIM) with initial condition $n^0$ satisfies
\begin{equation}\label{est:prop310}
\|Z\|_{L^\infty(\mathbb R_+\times \mathbb T^d)}\leq \|Z(0,\cdot)\|_{L^\infty(\mathbb T^d)}+\|y_{opt}\|_{L^\infty(\mathbb R_+\times \mathbb T^d)}+1,
\end{equation}
where $Z$ is defined by \eqref{NZ}.
\end{prop}

\begin{proof}[Proof of Proposition~\ref{prop:Linftyboundwithoutalpha}]

Let
\begin{equation}\label{def:kappa}
\kappa:=\|Z(0,\cdot)\|_{L^\infty(\mathbb T^d)}+\|y_{opt}\|_{L^\infty(\mathbb R_+\times \mathbb T^d)}+1,
\end{equation}
and 
\begin{equation}\label{def:tau}
\tau=\max\{t\geq 0;\,\|Z\|_{L^\infty([0,t]\times \mathbb T^d)}\leq\kappa\}.
\end{equation}
\rb{Note that if $\tau=+\infty$, the estimate \eqref{est:prop310} holds and the proof is completed. We can therefore consider the other case, where $\tau<\infty$. More precisely, we will use a contradiction argument: we assume that $\tau<\infty$ and prove that it is not possible.}

We have $\|Z(0,\cdot)\|_{L^\infty(\mathbb T^d)}\leq \kappa$. Thanks to Proposition~\ref{prop:moment4}, for some $\bar\tau_\kappa>0$ independent from $\gamma>\bar\gamma$,
\begin{equation}\label{est:y4fin}
\int_{\mathbb R} |y|^4\frac{n(t,x,y)}{\int_{\mathbb R} n(t,x,z)\,dz}\,dy\leq C_\kappa,
\end{equation}
for $(t,x)\in[0,\tau+\bar\tau_\kappa]\times\mathbb T^d$, as soon as $\gamma>\bar \gamma$. Our goal is to show that $\|Z(t)\|_{L^\infty(\mathbb T^d)}<\kappa$ for $t\in [0,\tau+\bar\tau]$, for some $\bar \tau$.

\medskip

If $\tau<\bar \tau_\kappa$, then we can then apply Proposition~\ref{prop:regularityNZ}, which ensures the H\"older regularity of $Z$ (uniformly for $\gamma>0$ large enough), and then, in particular,
\[\|Z(t,\cdot)\|_{L^\infty(\mathbb T^d)}\leq \left\|Z(0,\cdot)\right\|_{L^\infty(\mathbb T^d)}+Ct^\theta,\]
and then, up to a reduction of $\bar\tau_\kappa>0$ into $\bar \tau>0$,
\begin{equation}\label{est:hypini}
\|Z\|_{L^\infty([0,\tau+\bar\tau]\times\mathbb T^d)}<\|Z(0,\cdot)\|_{L^\infty(\mathbb T^d)}+\|y_{opt}\|_{L^\infty(\mathbb R_+\times \mathbb T^d)}+1=\kappa.
\end{equation}
In particular, $\|Z(t)\|_{L^\infty(\mathbb T^d)}<\kappa$ for $t\in[0,\bar \tau]$, which completes this initialisation step of this proof.

\medskip

From \eqref{eq:Zmod} we get, for $(t,x)\in[0,\tau+\bar \tau]\times\mathbb T^d$,
\begin{align}
&\partial_t Z(t,x)-\Delta_x Z(t,x)-2\frac{\nabla_xN(t,x)\cdot\nabla_x Z(t,x)}{N(t,x)}\nonumber\\
&\quad =-\frac 12\int_{\mathbb R} \left(y-Z(t,x)\right)(y-y_{opt}(t,x))^2 \Gamma_{A}(y-Z(t,x))\,dy\nonumber\\
&\qquad+\int_{\mathbb R} \left(y-Z(t,x)\right)(y-y_{opt}(t,x))^2 \left(\Gamma_{A}(y-Z(t,x))-\tilde n(t,x,y)\right)\,dy.\label{eq:estZ2}
\end{align}
The first term on the right hand side of this equation can be simplified as follows
\begin{align}
&-\frac 12\int_{\mathbb R} \left(y-Z(t,x)\right)(y-y_{opt}(t,x))^2 \Gamma_{A}(y-Z(t,x))\,dy\nonumber\\
&\quad = -\left(Z(t,x)-y_{opt}(t,x)\right)\int_{\mathbb R} |y|^2\Gamma_{A}(y)\,dy=-A\left(Z(t,x)-y_{opt}(t,x)\right),\label{est:reacZA}
\end{align}
and to estimate the last term of \eqref{eq:estZ2}, we introduce for some $R>0$ and a Lipschitz function $\phi_R:\mathbb R\mapsto[0,1]$ such that $\phi_R|_{[-R,R]}=1$, $\phi_R|_{[-R-1,R+1]}=0$ and $\|\phi_R'\|_{L^\infty(\mathbb R)}<2$. Then,
\begin{align*}
&\left|\int_{\mathbb R} \left(y-Z(t,x)\right)(y-y_{opt}(t,x))^2 \left(\Gamma_{A}(y-Z(t,x))-\tilde n(t,x,y)\right)\,dy\right|\\
&\quad\leq \left|\int_{\mathbb R} \phi_R(y)\left(y-Z(t,x)\right)(y-y_{opt}(t,x))^2 \left(\Gamma_{A}(y-Z(t,x))-\tilde n(t,x,y)\right)\,dy\right|\nonumber\\
&\qquad+\left|\int_{\mathbb R} \left(1-\phi_R(y)\right)\left(y-Z(t,x)\right)(y-y_{opt}(t,x))^2 \left(\Gamma_{A}(y-Z(t,x))-\tilde n(t,x,y)\right)\,dy\right|\nonumber\\
&\quad\leq \max_{y\in\mathbb R }\left|\frac d{dy}\left[\phi_R(y)\left(y-Z(t,x)\right)(y-y_{opt}(t,x))^2\right]\right|W_1\left(\tilde n(t,x,\cdot),\Gamma_{A}(\cdot-Z(t,x))\right)\nonumber\\
&\qquad+C_\kappa\int_{|y|\geq R} |y+\kappa|^3\tilde n(t,x,y)\,dy+C_\kappa\int_{|y|\geq R} |y+\kappa|^3\Gamma_{A}(y-Z(t,x))\,dy,
\end{align*}
where $\kappa>0$ is defined by \eqref{def:kappa} and where we have used the  Kantorovich-Rubinstein estimate (see Section~\ref{subsec:defWasserstein} in the Appendix) to obtain the first term on the right hand side of the estimate above. We use next the fact that $\phi_R$ \rb{as well as its derivative $\phi_R'$,} is supported in $[-R-1,R+1]$ and the Chebyshev's inequality to obtain
\begin{align*}
&\left|\int_{\mathbb R} \left(y-Z(t,x)\right)(y-y_{opt}(t,x))^2 \left(\Gamma_{A}(y-Z(t,x))-\tilde n(t,x,y)\right)\,dy\right|\\
&\quad\leq C(R+\kappa)^3W_2\left(\tilde n(t,x,\cdot),\Gamma_{A}(\cdot-Z(t,x))\right)+\frac{C}R\int |y|^4\tilde n(t,x,y)\,dy\nonumber\\
&\qquad+\frac{C}{R}\int_{\mathbb R} |y|^4\Gamma_{A}(y-Z(t,x))\,dy,
\end{align*}
To estimate the three terms that appear in the estimate above, we use Proposition~\ref{prop:Wassersteincontraction} and \eqref{est:y4fin} to obtain
\begin{equation}
\left|\int_{\mathbb R} \left(y-Z(t,x)\right)(y-y_{opt}(t,x))^2 \left(\Gamma_{A}(y-Z(t,x))-\tilde n(t,x,y)\right)\,dy\right|\leq \frac{C_\kappa R^3}{\gamma^\theta}+\frac {C_\kappa} R\leq \frac{C_\kappa}{\gamma^{\theta/4}},\label{est:reacZA2}
\end{equation}
for $(t,x)\in [\theta\ln\gamma/\gamma,\tau+\bar\tau]\times \mathbb T^d$, provided we chose $R=\gamma^{\theta/4}$. Note that for $\gamma>0$ large enough, $C_\kappa\frac {\ln \gamma}\gamma<k\sigma$, so that $[\bar\tau,\tau+\bar\tau]\subset[\theta\ln\gamma/\gamma,\tau+\bar\tau]$. Thanks to \eqref{est:reacZA} and \eqref{est:reacZA2}, we obtain that for $t\in\left[\theta\ln\gamma/\gamma,\tau+\bar\tau\right]$ and $\gamma\geq \bar\gamma$ (this may require to increase the value of $\bar\gamma>0$, but this new value of $\bar\gamma$ remains independent of $\tau\geq\bar\tau$),
\[\partial_t Z(t,x)-\Delta_x Z(t,x)=2\frac{\nabla_xN(t,x)\cdot\nabla_x Z(t,x)}{N(t,x)}-A(Z(t,x)-y_{opt}(t,x))+\mathcal O(1),\]
where $\left|\mathcal O(1)\right|\leq A$. This estimate combined to \eqref{est:hypini} and a comparison of $Z(t,x)$ with $\varphi(t,x)\equiv \pm\kappa$ thanks to the comparison principle (see Remark~\ref{rem:comparaison}) proves that $\|Z\|_{L^\infty([0,(k+1)\sigma]\times\mathbb T^d)}<\kappa$. This is in contradiction with \eqref{def:tau} if $\tau\neq +\infty$, which concludes the proof.
\end{proof}

We are now ready to prove Theorem~\ref{Thm:macro}:

\begin{proof}[Proof of Theorem~\ref{Thm:macro}]
Thanks to Proposition~\ref{prop:Linftyboundwithoutalpha}, there exists a solution $n\in L^\infty(\mathbb R_+\times\mathbb T^d,L^1((1+|y|^4)\,dy))$ of the SIM with initial condition $n^0$ such that 
\begin{equation}\label{est:Zinf}
\|Z\|_{L^\infty(\mathbb R_+\times \mathbb T^d)}\leq \kappa:=\|Z(0,\cdot)\|_{L^\infty(\mathbb T^d)}+\|y_{opt}\|_{L^\infty(\mathbb R_+\times \mathbb T^d)}+1.
\end{equation}
Thanks to \eqref{eq:Nmod} and \eqref{eq:Zmod}, we get the following expressions for the functions $\varphi_N$ and $\varphi_Z$ appearing in \eqref{eq:thm}:
\begin{equation}\label{def:phiN}
\varphi_N(t,x)=\left(-\frac 12\int_{\mathbb R} (y-y_{opt}(t,x))^2\tilde n(t,x,y)\,dy+\frac A2+\frac 12(Z(t,x)-y_{opt}(t,x))^2\right)N(t,x).
\end{equation}
\begin{equation}\label{def:phiZ}
\varphi_Z(t,x)= -\frac 12\int_{\mathbb R} (y-Z(t,x))(y-y_{opt}(t,x))^2\tilde n(t,x,y)\,dy+A(Z(t,x)-y_{opt}(t,x)).
\end{equation}

Thanks to \eqref{est:Zinf}, we can apply Proposition~\ref{prop:moment4} with $[0,\tau)=[0,\infty)$, and there exists a constant $C>0$ such that $\int |y|^4\tilde n(t,x,y)\,dy\leq C$, for any $(t,x)\in\mathbb R_+\times \mathbb T^d$  and $\gamma>\bar\gamma$ large enough. This combined to the boundedness of $Z$ provided by \eqref{est:Zinf} implies the existence of a constant $C>0$ such that
\[\forall t\geq 0,\quad \|\varphi_N(t,\cdot)\|_{L^\infty(\mathbb T^d)}+\|\varphi_Z(t,\cdot)\|_{L^\infty(\mathbb T^d)}\leq C.\]
To show \eqref{thm:est2}, we need to show that after an initial layer, this estimate can be improved. For $\varphi_Z$, we can use an estimate derived in the proof of Proposition~\ref{prop:Linftyboundwithoutalpha}:  \eqref{est:reacZA2} and \eqref{est:reacZA} imply
\[\forall t\geq C\frac{\ln\gamma}{\gamma},\quad \|\varphi_Z(t,\cdot)\|_{L^\infty(\mathbb T^d)}\leq \frac C{\gamma^{\theta/4}}.\]
To estimate $\|\varphi_N(t,\cdot)\|_{L^\infty(\mathbb T^d)}$, we note that
\[\frac 12\int_{\mathbb R} (y-y_{opt}(t,x))^2\Gamma_A\left(y-Z(t,x)\right)\,dy=\frac A2+\frac 12\left(Z(t,x)-y_{opt}(t,x)\right)^2,\]
and then
\[\varphi_N(t,x)=\frac {1}2\left(\int_{\mathbb R} \left(y-y_{opt}(t,x)\right)^2\left(\Gamma_A\left(y-Z(t,x)\right)-\tilde n(t,x,y)\right)\,dy\right)N(t,x).\]
We can repeat the argument developed in \eqref{est:reacZA}-\eqref{est:reacZA2} to estimate the integral term, and then, 
\[\forall t\geq C\frac{\ln\gamma}{\gamma},\quad \|\varphi_N(t,\cdot)\|_{L^\infty(\mathbb T^d)}\leq \frac C{\gamma^{\theta/4}}.\]
To conclude the proof, we notice that \eqref{thm:est1} is a consequence of Proposition~\ref{prop:Wassersteincontraction}. To obtain estimate \eqref{thm:est2}, we define a slightly different parameter $\theta$: $\tilde \theta:=\frac \theta 4>0$.
Finally, $(N,Z)=(N(t,x),Z(t,x))$, for some $\gamma>0$ large, are $C^1$ in $t$ and $C^2$ in $x$ thanks to Section~\ref{subsec:RegularityNZ}, and satisfy \eqref{eq:thm}, according to \eqref{eq:Nmod} and \eqref{eq:Zmod} and the definitions \eqref{def:phiN}, \eqref{def:phiZ} of $\varphi_N$, $\varphi_Z$.
\end{proof}

\section{Appendix}

\subsection{Wasserstein distances}\label{subsec:defWasserstein}

In this section, we review the definition of the Wasserstein distance and several useful formula. We refer to \cite{Villani} for more on this topic. Let $p\geq 1$, and $\mathcal P_p(\mathbb R)$ the set of probability measures with finite $p-$moment, that is the set of probability measures $\mu$ over $\mathbb R$ such that
\begin{equation}\label{def:finite-moment}
\int_{\mathbb R} |y|^d\,d\mu(y)<\infty.
\end{equation}
If $\pi$ is a probability measure over $\mathbb R^2$, we call marginals the probability measures $\pi|_1$ and $\pi|_2$ such that for any Borelian $A\subset\mathbb R$,
\[\pi(A\times \mathbb R)=\pi|_1(A),\quad \pi(\mathbb R \times A)=\pi|_2(A).\]
For $\tilde n,\tilde m\in\mathcal P_2(\mathbb R)$, we call transference plans the probability measures $\pi$ over $\mathbb R^2$ such that $\pi|_1=\tilde n$ and $\pi|_2=\tilde m$, and $\Pi(\tilde n,\tilde m)$ the set of such plans:
\begin{equation}\label{def:plan}
\Pi(\tilde n,\tilde m):=\left\{\pi\in\mathcal P(\mathbb R^2);\, \pi|_1=\tilde n,\,\pi|_2=\tilde m\right\}.
\end{equation}
We can now define the $p-$ Wasserstein distance between two measures $\tilde n,\tilde m\in\mathcal P_p(\mathbb R)$ as follows
\begin{equation*}
W_p(\tilde n,\tilde m)=\left( \inf_{\pi\in \Pi(\tilde n,\tilde m)}\iint_{\mathbb R^2} |y_1-y_2|^p\,d\pi(y_1,y_2)\right)^{\frac 1p}.
\end{equation*}
Note that $W_p(\tilde n,\delta_{\bar y})=\int |y-\bar y|^p\,d\tilde n(y)$, for any $\bar y\in\mathbb R$ and $\tilde n\in \mathcal P_p(\mathbb R)$. 

\medskip

For $\tilde n,\tilde m\in\mathcal P_2(\mathbb R)$ and $f\in W^{1,\infty}(\mathbb R)$, the Kantorovich-Rubinstein is the following useful estimate:
\[\left|\int_{\mathbb R} f(y)d\tilde n(y)-\int_{\mathbb R} f(y)d\tilde m(y)\right|\leq \|f'\|_{L^\infty(\mathbb R)} W_1(\tilde n,\tilde m).\]

For $\tilde n,\tilde m\in\mathcal P_p(\mathbb R)$ (with $p\geq 1$), the Kantorovich duality provides the following equality
\begin{equation}\label{eq:Kantorovich}
W_p(\tilde n,\tilde m)=\left( \sup_{(\varphi,\psi)\in F}\int_{\mathbb R} \varphi(y)\,d\tilde n(y)+\int_{\mathbb R} \psi(Y)\,d\tilde m(Y)\right)^{\frac 1p},
\end{equation}
where $F=\left\{(\varphi,\psi)\in (C^0_b(\mathbb R,\mathbb R))^2;\, \forall y,Y\in\mathbb R, \varphi(y)+\psi(Y)\leq |y-Y|^p \right\}$.

\medskip

Finally, we will also use the convexity of the squared Wasserstein distance $W_2$. Let $\tilde n_1,\tilde m\in\mathcal P_2(\mathbb R)\cap L^1(\mathbb R)$ and, $\tilde n_2\in L^\infty([0,t]\times \mathbb T^d,\mathcal P_2\left(\mathbb R)\cap L^1(\mathbb R)\right)$, for some $t>0$.
For any $\alpha\in[0,1]$ and $\beta\in L^1([0,t]\times\mathbb T^d)$ such that $\int_{[0,t]\times\mathbb T^d} \beta=1-\alpha$, we have
\begin{align}
&W_2^2\left(\alpha \tilde n_1+\int_0^t\int_{\mathbb T^d} \beta(\sigma,x)\tilde n_2(t,x,\cdot)\,dx\,d\sigma,\tilde m\right)\nonumber\\
&\quad \leq \alpha W_2^2(\tilde n_1,\tilde m)+\int_0^t\int_{\mathbb T^d} \beta(\sigma,x)W_2^2\left(\tilde n_2(\sigma,x,\cdot),\tilde m\right)\,dx\,d\sigma.\label{eq:convexity}
\end{align}

To obtain this estimate, let $(\varphi,\psi)\in F$ with $p=2$. Then,
\begin{align*}
&\int_{\mathbb R} \varphi(y) \left(\alpha \tilde n_1(y)+\int_0^t\int_{\mathbb T^d} \beta(\sigma,x)\tilde n_2(\sigma,x,y)\,dx\,d\sigma\right)\,dy+\int_{\mathbb R}\psi(Y) \tilde m(Y)\,dY\\
&\quad\leq \alpha\left(\int_{\mathbb R} \varphi(y)\tilde n_1(y)\,dy+\psi(Y)\tilde m(Y)\,dY\right)\\
&\qquad+\int_0^t\int_{\mathbb T^d} \beta(\sigma,x)\left(\int_{\mathbb R} \varphi(y)\tilde n_2(\sigma,x,y)\,dy+\int_{\mathbb R} \psi(Y)\tilde m(Y)\,dY\right)\,dx\,d\sigma\\
&\quad\leq \alpha W_2^2\left(\tilde n_1,\tilde m\right)+\int_0^t\int_{\mathbb T^d} \beta(\sigma,x)W_2^2\left(\tilde n_2(\sigma,x,\cdot),\tilde m\right)\,dx\,d\sigma,
\end{align*}
and \eqref{eq:convexity} follows thanks to \eqref{eq:Kantorovich}, if we consider the suppremum over $(\varphi,\psi)\in F$.
\subsection{\rb{Properties of the} Infinitesimal operator}
\label{subsec:infinitesimaloperator}
\rb{In \eqref{def:T}, we have defined the  Infinitesimal operator $T$. More precisely, we define this operator on the space $\mathcal P_2(\mathbb R)$ (see Section~\ref{subsec:defWasserstein}) by \eqref{def:T}.} Then, for any $\tilde n\in\mathcal P_2(\mathbb R)$,
\[\int_{\mathbb R} T(\tilde n)(y)\,dy=\int_{\mathbb R} \tilde n(y)\,dy=1,\quad  \int_{\mathbb R} y\, T(\tilde n)(y)\,dy=\int_{\mathbb R} y\, \tilde n(y)\,dy,\]
and for any $Z\in\mathbb R$,
\begin{equation}\label{eq:Maxwellian2A}
\forall y\in\mathbb R,\quad T\left(\Gamma_{A}(\cdot-Z)\right)(y)=\Gamma_{A}(y-Z).
\end{equation}
where 
\begin{equation}\label{def:Gamma}
\Gamma_{A}(y)=\frac 1{\sqrt{2\pi A}}e^{-\frac {|y|^2}{2A}}.
\end{equation}
$T$ induces a contraction for the Wasserstein distance $W_2$, which can be seen as a version of the Tanaka inequality \cite{Tanaka} (see also \cite{Bassetti,Bolley}):
\begin{theo}[A Tanaka inequality]\label{thm:Tanaka}
Let $A>0$, $\tilde n,\,\tilde m\in\mathcal P_2(\mathbb R)$ such that $\int y\tilde n(y)\,dy=\int y\tilde m(y)\,dy$, and $T$ defined by~\eqref{def:T}. Then
\[W_2(T(\tilde n),T(\tilde m))\leq \frac 1{\sqrt 2}W_2(\tilde n,\tilde m).\]
\end{theo}

\begin{proof}[Proof of the Theorem~\ref{thm:Tanaka}]
We consider $\varphi,\psi$ such that for any $y,Y\in\mathbb R$, $\varphi(y)+\psi(Y)\leq |y-Y|^2$, and $\pi\in \Pi(\tilde n,\tilde m)$. Then,
\begin{align}
&\int_{\mathbb R} \varphi(y)T(\tilde n)(y)\,dy+\int_{\mathbb R} \psi(Y)T(\tilde m)(Y)\,dY\nonumber\\
&\quad = \iiint_{\mathbb R^3} \varphi(y)\Gamma_{A/2}\left(y-\frac {y_*+y_*'}2\right)\tilde n(y_*)\tilde n(y_*')\,dy_*\,dy_*'\,dy\nonumber\\
&\qquad +\iiint_{\mathbb R^3} \psi(Y)\Gamma_{A/2}\left(Y-\frac {Y_*+Y_*'}2\right)\tilde n(Y_*)\tilde n(Y_*')\,dY_*\,dY_*'\,dY\nonumber\\
&\quad = \iiint_{\mathbb R^3} \varphi\left(y+\frac {y_*+y_*'}2\right)\Gamma_{A/2}\left(y\right)\tilde n(y_*)\tilde n(y_*')\,dy_*\,dy_*'\,dy\nonumber\\
&\qquad +\iiint_{\mathbb R^3} \psi\left(y+\frac {Y_*+Y_*'}2\right)\Gamma_{A/2}\left(Y\right)\tilde n(Y_*)\tilde n(Y_*')\,dY_*\,dY_*'\,dY\nonumber\\
&\quad= \int_{\mathbb R} \Gamma_{A/2}(y)\iiiint_{\mathbb R^4}  \varphi\left(y+\frac {y_*+y_*'}2\right)+\psi\left(y+\frac {Y_*+Y_*'}2\right)\,d\pi(y_*,Y_*)\,d\pi(y_*',Y_*')\,dy\nonumber\\
&\quad \leq \int_{\mathbb R} \Gamma_{A/2}(y)\iiiint_{\mathbb R^4}  \left|\left(y+\frac {y_*+y_*'}2\right)-\left(y+\frac {Y_*+Y_*'}2\right)\right|^2\,d\pi(y_*,Y_*)\,d\pi(y_*',Y_*')\,dy\nonumber\\
&\quad \leq \frac 14\iiiint_{\mathbb R^4}  \left|(y_*-Y_*)+(y_*'-Y_*')\right|^2\,d\pi(y_*,Y_*)\,d\pi(y_*',Y_*').\label{eq:annexe1}
\end{align}
We notice that
\[\iiiint_{\mathbb R^4}(y_*-Y_*)(y_*'-Y_*')\,d\pi(y_*,Y_*)\,d\pi(y_*',Y_*')=\left(\int_{\mathbb R} y\tilde n(y)\,dy-\int_{\mathbb R} y\tilde m(Y)\,dY\right)^2=0,\]
and then
 \begin{align*}
&\int_{\mathbb R} \varphi(y)T(\tilde n)(y)\,dy+\int_{\mathbb R} \psi(Y)T(\tilde m)(Y)\,dY\\
&\quad \leq \frac 14\iiiint_{\mathbb R4}  \left[(y_*-Y_*)^2+2(y_*-Y_*)(y_*'-Y_*')+(y_*'-Y_*')^2\right]\,d\pi(y_*,Y_*)\,d\pi(y_*',Y_*')\\
&\quad =\frac 12\iint_{\mathbb R^2} (y-Y)^2\,d\pi(y,Y).
\end{align*}
Since this inequality holds for any $\pi\in\Pi(\tilde n,\tilde m)$, we can consider the infinum of over these, to obtain, thanks to the definition of the Wasserstein distance:
\[\int_{\mathbb R} \varphi(y)T(\tilde n)(y)\,dy+\int_{\mathbb R} \psi(Y)T(\tilde m)(Y)\,dY\leq\frac 12 W_2^2(\tilde n,\tilde m). \]
We can now take the supremum of this inequality over the functions $\varphi,\psi$ satisfying $\varphi(y)+\psi(Y)\leq |y-Y|^2$ and conclude, thanks to the Kantorovich duality formula \eqref{eq:Kantorovich}.
\end{proof}

\begin{cor}[A Tanaka inequality for $W_4$]\label{cor:Tanaka4}
Let $A>0$, $\tilde n,\,\tilde m\in\mathcal P_4(\mathbb R)$ such that $\int y\tilde n(y)\,dy=\int y\tilde m(y)\,dy$, and $T$ defined by~\eqref{def:T}. Then
\[W_4(T(\tilde n),T(\tilde m))\leq \frac 1{2^{1/4}}W_4(\tilde n,\tilde m).\]
\end{cor}

\begin{proof}[Proof of the Corollary~\ref{cor:Tanaka4}]
We can reproduce the proof of Theorem~\ref{thm:Tanaka} until \eqref{eq:annexe1}, and obtain that for any $\varphi,\psi$ satisfying $\varphi(y)+\psi(Y)\leq |y-Y|^4$ and $\pi\in \Pi(\tilde n,\tilde m)$,
\begin{align}
&\int_{\mathbb R} \varphi(y)T(\tilde n)(y)\,dy+\int_{\mathbb R} \psi(Y)T(\tilde m)(Y)\,dY\nonumber\\
&\quad \leq \frac 1{16}\iiiint_{\mathbb R^4}\int  \left|(y_*-Y_*)+(y_*'-Y_*')\right|^4\,d\pi(y_*,Y_*)\,d\pi(y_*',Y_*')\nonumber\\
&\quad=\frac 1{16}\iiiint_{\mathbb R^4}  \Big[(y_*-Y_*)^4+4(y_*-Y_*)^3(y_*'-Y_*')+6(y_*-Y_*)^2(y_*'-Y_*')^2\nonumber\\
&\qquad \phantom{dsgregszrtgsz}+4(y_*-Y_*)(y_*'-Y_*')^3+(y_*'-Y_*')^4\Big]\,d\pi(y_*,Y_*)\,d\pi(y_*',Y_*')\nonumber\\
&\quad=\frac 1{8}\left(\iint_{\mathbb R^2}  (y-Y)^4\,d\pi(y,Y)\right)+\frac 38\left(\iint_{\mathbb R^2} (y-Y)^2\,d\pi(y,Y)\right)^2\nonumber\\
&\quad\leq\frac 1{2}\left(\iint_{\mathbb R^2}  (y-Y)^4\,d\pi(y,Y)\right).\nonumber
\end{align}
The rest of the proof is similar to the proof of Theorem~\ref{thm:Tanaka}.
\end{proof}

\subsection{Existence theory for a truncated version of (SIM)}\label{subsec:existencetruncated}
\gr{
In this section, we prove the existence of global solutions to \eqref{eq:nRsol}, the truncated version of (SIM). To do so, we first construct local in time (i.e. for $t\in[0,\bar t]$, $\bar t>0$) solutions of the truncated equation \eqref{eq:nRsol}:

\begin{lem}\label{lem:exist}
Let $y_{opt}\in W^{1,\infty}(\mathbb R_+\times\mathbb T^d,\mathbb R)$, $A>0$, $\gamma\geq 1$ and  $n^0$ satisfying Assumption~\ref{Assumption3}. There is $C>0$ such that if $R>1$ and if $n^0\in C^0(\mathbb T^d\times \mathbb R,\mathbb R_+)$ satisfies $n^0(x,y)>0$ for $(x,y)\in\mathbb T^d\times [-R,R]$, then there is a unique solution $n_R=n_R(t,x,y)$ of \eqref{eq:nRsol} for $t\in[0,\bar t]$ together with the initial data $n_R(0,x,y)=n^0(x,y)1_{|y|\leq R}$, where
\[\bar t=\frac 1{C((\|n^0\|_{L^\infty(\mathbb T^d\times[-R,R])}+\gamma)R+1)}.\]
More precisely, $n_R$ satisfies the following estimate for some constant $C_R>0$ may depend on $R>0$ and $\gamma>0$:
\begin{align*}
&\sup_{y\in\mathbb R}\left(\sup_{t\in[0,\bar t]}\|n_R(t,\cdot,y)\|_{H^1(\mathbb T^d)}+\|n_R(\cdot,\cdot,y)\|_{L^2_{loc}([0,\bar t),H^2(\mathbb T^d))}+\|\partial_t n_R(\cdot,\cdot,y)\|_{L^2_{loc}([0,\bar t)\times \mathbb T^d)}\right) \leq C_{R,\gamma}.
\end{align*}
\end{lem}

%
}\gr{

\begin{proof}[Proof of Lemma~\ref{lem:exist}]

\noindent{\textbf{Step 1: Definition of the  set $\mathcal F_R$ and the application $F_{R,\bar t}$}}

For $0<\gr{\bar t}<\frac 1{3\gamma(1+A/2)}$, let
\begin{align}
\mathcal F_{R,\bar t}&:=\bigg\{m\in L^\infty([0,\gr{\bar t}]\times\mathbb T^d\times\mathbb R);\, m\geq 0,\, m(t,x,y)=0\textrm{ if }|y|\geq R,\nonumber\\
&\qquad \|m\|_{L^\infty([0,\gr{\bar t}]\times \mathbb T^d\times\mathbb R)}\leq 3\|n^0\|_{L^\infty(\mathbb T^d\times[-R,R])}\bigg\}.\label{eq:FR}
\end{align}

We introduce the operator $F_{R,\bar t}$, that is defined by $F_{R,\bar t}(m)=n$ for $m\in \mathcal F_{R,\bar t}$ and where $n$ is the solution of 
\begin{align}
&\partial_t n(t,x,y)=\Delta_x n(t,x,y)+\left(1+\frac A2-\frac 12(y-y_{opt}(t,x))^2-\int_{\mathbb R} m(t,x,z)\,dz\right)n(t,x,y)\nonumber\\
&\quad +\gamma\left(1_{|y|\leq R}\iint_{\mathbb R^2} \Gamma_{A/2}\left(y-\frac{y_*+y_*'}2\right)\frac{m(t,x,y_*)m(t,x,y_*')}{\int_{\mathbb R} m(t,x,z)\,dz}\,dy_*\,dy_*'-n(t,x,y)\right),\label{eq:nk} 
\end{align}
together with $n(0,x,y)=n^0(x,y)\,1_{|y|\leq R}$ for $(x,y)\in\mathbb T^d\times\mathbb R$, so that $n(t,x,y)=0$ if $|y|>R$.\rb{Note that a different application could probably be used}. Notice that in \eqref{eq:nk}, $y$ can be seen as a simple coefficient, and the equation can be solved independently for each $y\in\mathbb R$. We also notice that for any fixed $y\in[-R,R]$, the coefficients of the parabolic equation \eqref{eq:nk} are bounded when $m\in \mathcal F_{R,\bar t}$ (we recall that $\gamma>0$ is here a fixed constant), so that a non-negative weak solution $\left((t,x)\mapsto n(t,x,y)\right)\in L^2([0,T]_+,H^1(\mathbb T^d))$ (such that $\left((t,x)\mapsto \partial_t n(t,x,y)\right)\in L^2([0,T],H^{-1}(\mathbb T^d))$) exists thanks to standard arguments (see Theorem 3  in Section 7.1.2 of  \cite{Evans}). The bounded coefficients actually imply that this is a \emph{strong} solution (in the sense of \cite{Evans}), thanks to Theorem 5  in Section 7.1.3 of \cite{Evans} applied for each $y\in\mathbb R$. More precisely,
\begin{align}\label{est:regn}
&\textrm{sup}_{t\in[0,\gr{\bar t}]}\|n(t,\cdot,y)\|_{H^1(\mathbb T^d)}+\|n(\cdot,\cdot,y)\|_{L^2([0,\gr{\bar t}],H^2(\mathbb T^d))}+\|\partial_t n(\cdot,\cdot,y)\|_{L^2([0,\gr{\bar t}]\times \mathbb T^d)} \leq C_R,
\end{align}
where the constant $C_R$ is related to a bound on the coefficient $1_{|y|\leq R}\frac 12(y-y_{opt}(t,x))^2$, and it therefore depends on $R>0$ in this estimate.

\medskip

\noindent{\textbf{Step 2: The set $\mathcal F_{R,\bar t}$ is stable under $F_{R,\bar t}$ provided $\bar t\leq 1/C$}}

 We notice that $n$ satisfies
\ra{\begin{equation}\label{eq:supersoln}
\partial_t n(t,x,y) -\Delta_x n(t,x,y) \leq \left(1+\frac A2\right) n(t,x,y)+\gamma \Gamma_{A/2}(0)\|m\|_{L^\infty([0,\bar t]\times\mathbb T^d,L^1([-R,R]))}.
\end{equation}
Let 
\begin{align*}
\phi(t,x,y)&:= 1_{|y|\leq R}\bigg[\left(\textrm{sup}_{\mathbb T^d\times[-R,R]} n^0\right)e^{\left(1+\frac A2\right)t}\\
&\quad +\left(\|m\|_{L^\infty([0,\bar t]\times\mathbb T^d,L^1([-R,R]))}\right)\gamma \Gamma_{A/2}(0)\frac{e^{\left(1+\frac A2\right)t}-1}{1+A/2}\bigg],
\end{align*}
that is a super-solution of \eqref{eq:supersoln}. Since additionally $\phi(0,x,y)\geq n^0(x,y)1_{|y|\leq R}$ for $(x,y)\in\mathbb T^d\times\mathbb R$, we can use the maximum principle (see Corollary 7.4 p. 159 in \cite{Liebermann}) to \gr{compare} $(t,x)\mapsto n(t,x,y)$ and $(t,x)\mapsto \phi(t,x,y)$, for any $y\in[-R,R]$. We then show that for $(t,x,y)\in[0,\gr{\bar t}]\times\mathbb T^d\times \mathbb R$,
\begin{align}
n(t,x,y)\leq  \phi(t,x,y) \leq \left(1_{|y|\leq R}e^{\left(1+\frac A2\right)t}+3\gamma \Gamma_{A/2}(0)\left(e^{\left(1+\frac A2\right)t}-1\right)\right)\|n^0\|_{L^\infty(\mathbb T^d\times[-R,R])},\label{eq:esti1}
\end{align}}
\gr{where we have used the estimate on $\|m\|_{L^\infty([0,\gr{\bar t}]\times \mathbb T^d,L^1(\mathbb R))}$ provided by the definition of $\mathcal F_{R,\bar t}$.} If 
\begin{equation}\label{eq:est:bart}
\bar t\leq  \frac{\ln \left(1+\min(1/(3\gamma \Gamma_{A/2}(0)),1)\right)}{1+A/2},
\end{equation}
then $\left(1_{|y|\leq R}e^{\left(1+\frac A2\right)t}+3\gamma \Gamma_{A/2}(0)\left(e^{\left(1+\frac A2\right)t}-1\right)\right)\leq 3$, which implies
\begin{align}
\|n\|_{L^\infty([0,\bar t]\times\mathbb T^d\times[-R,R])}&\leq 3\|n^0\|_{L^\infty(\mathbb T^d\times[-R,R])}.\label{eq:estnlinfty}
\end{align} 
We have proven that for any $m\in\mathcal F_{R,\bar t}$, we have $F_{R,\bar t}(m)=n\in \mathcal F_{R,\bar t}$ provided \eqref{eq:est:bart} is satisfied.

\medskip

\noindent{\textbf{Step 3: $F_{R,\bar t}$ is a contraction for the norm $\|\cdot\|_{L^\infty([0,T]\times\mathbb T^d\times\mathbb R)}$ on $\mathcal F_{R,\bar t}$}}

Let $m,\,\tilde m\in\mathcal F_{R,\bar t}$ (see \eqref{eq:FR}), and $n:=F_{R,\bar t}(m)$, $\tilde n:=F_{R,\bar t}(\tilde m)$. Then $(n-\tilde n)(0,x,y)=0$ for $(x,y)\in\mathbb T^d\times\mathbb R$, $(n-\tilde n)(t,x,y)=0$ for $(t,x)\in[0,\gr{\bar t}]\times \mathbb T^d$ and $y\notin[-R,R]$. We therefore consider $(t,x)\in[0,\gr{\bar t}]\times \mathbb T^d$ and $y\in[-R,R]$ from now on. We can estimate the difference between birth terms as follows:
\begin{align*}
&\left|\iint_{\mathbb R^2} \Gamma_{A/2}\left(y-\frac{y_*+y_*'}2\right)\left[\frac{m(t,x,y_*)m(t,x,y_*')}{\int_{\mathbb R} m(t,x,z)\,dz}-\frac{\tilde m(t,x,y_*)\tilde m(t,x,y_*')}{\int_{\mathbb R} \tilde m(t,x,z)\,dz}\right]\,dy_*\,dy_*'\right|\\
&\quad =\bigg|\iint_{\mathbb R^2} \Gamma_{A/2}\left(y-\frac{y_*+y_*'}2\right)\bigg[\frac{\left(m(t,x,y_*)-\tilde m(t,x,y_*)\right)m(t,x,y_*')}{\int_{\mathbb R}  m(t,x,z)\,dz}\\
&\qquad +\frac{\tilde m(t,x,y_*)m(t,x,y_*')}{\left(\int_{\mathbb R} m(t,x,z)\,dz\right)\left(\int_{\mathbb R} \tilde m(t,x,z)\,dz\right)}\left(\int_{\mathbb R} \tilde m(t,x,z)-m(t,x,z)\,dz\right)\\
&\qquad +\frac{\tilde m(t,x,y_*)\left(m(t,x,y_*')-\tilde m(t,x,y_*')\right)}{\int_{\mathbb R} \tilde m(t,x,z)\,dz}\bigg]\,dy_*\,dy_*'\bigg|\\
&\quad \leq 3\Gamma_{A/2}(0)\int_{\mathbb R} \left|\tilde m(t,x,z)-m(t,x,z)\right|\,dz=3\Gamma_{A/2}(0)\int_{-R}^R \left|\tilde m(t,x,z)-m(t,x,z)\right|\,dz.
\end{align*}
\rb{$(n-\tilde n)$ satisfies:
\begin{align}
&\partial_t (n-\tilde n)(t,x,y)-\Delta_x (n-\tilde n)(t,x,y)\leq \left(1+\frac A2\right)(n-\tilde n)(t,x,y)\nonumber\\
&\qquad +\left(\int_{\mathbb R} \left(m(t,x,z)- \tilde m(t,x,z)\right)\,dz\right)\tilde n(t,x,y)  +3\gamma\Gamma_{A/2}(0)\int_{-R}^R |m(t,x,z)- \tilde m(t,x,z)|\,dz\nonumber\\
&\quad \leq \left(1+\frac A2\right)(n-\tilde n)(t,x,y)+6R\left(\|n^0\|_{L^\infty(\mathbb T^d\times[-R,R])}+\gamma \Gamma_{A/2}(0)\right)\|m-\tilde m\|_{L^\infty([0,\tau]\times\mathbb T^d\times[-R,R])},\label{est:ntildenexist}
\end{align}
where we have used \eqref{eq:estnlinfty} to estimate $0\leq \tilde n(t,x,y)$  from above. We notice that 
\[(t,x)\mapsto 6R\left(\|n^0\|_{L^\infty(\mathbb T^d\times[-R,R])}+\gamma \Gamma_{A/2}(0)\right)\|m-\tilde m\|_{L^\infty([0,\tau]\times\mathbb T^d\times[-R,R])}e^{(1+A/2)t}\]
is a super-solution of the parabolic equation \eqref{est:ntildenexist} for any fixed $y\in[-R,R]$, and then the comparison principle implies 
\begin{align}
&\rb{\max_{[0,\gr{\bar t}]\times\mathbb T^d\times[-R,R]}(n-\tilde n)}\nonumber\\
&\quad\leq 6R\left(\|n^0\|_{L^\infty(\mathbb T^d\times[-R,R])}+\gamma \Gamma_{A/2}(0)\right)\|m-\tilde m\|_{L^\infty([0,\tau]\times\mathbb T^d\times[-R,R])}\frac{e^{(1+A/2)\bar t}-1}{1+A/2}\nonumber\\
&\quad \leq\frac 12\|m-\tilde m\|_{L^\infty([0,\tau]\times\mathbb T^d\times[-R,R])},\label{eq:condT}
\end{align}}
provided we set $\bar t$ as follows 
\begin{equation}\label{eq:timeT}
\gr{\bar t}= \frac 1{C((\|n^0\|_{L^\infty(\mathbb T^d\times[-R,R])}+\gamma)R+1)},
\end{equation}
where $C>0$ is chosen large enough for \eqref{eq:est:bart} and \eqref{eq:condT} to hold.
\rb{If $(n-\tilde n)(t,x,y)\leq 0$, a similar argument can made on $(\tilde n-n)$, since $n$ and $\tilde n$ have symmetric properties. We then obtain, for $\bar t$ defined by \eqref{eq:timeT},}
\gr{\begin{align*}
\|n-\tilde n\|_{L^\infty([0,\gr{\bar t}]\times\mathbb T^d\times[-R,R])}&\leq \frac 12\|m-\tilde m\|_{L^\infty([0,\gr{\bar t}]\times\mathbb T^d\times[-R,R])}.
\end{align*}}
The Banach fixed-point theorem then shows that there is a unique fixed point $n_R$ of $F_{R,\bar t}$ in $\mathcal F_{R,\bar t}$. That fixed point $n_R\in\mathcal F_{R,\bar t}$  is a solution of \eqref{eq:nRsol} that satisfies \eqref{est:regn}.


\end{proof}
We can now  construct global (i.e. for $t\in\mathbb R_+$) solutions of the truncated model \eqref{eq:nRsol}:

\begin{prop}\label{prop:existencenR}
Let $y_{opt}\in W^{1,\infty}(\mathbb R_+\times\mathbb T^d,\mathbb R)$, $A>0$, $\gamma>\rb{2+A+\|y_{opt}\|^2_{L^\infty(\mathbb R_+\times\mathbb T^d)}}$, and $n^0$ satisfying Assumption~\ref{Assumption3}. For $R>1$, there is a unique global solution $n_R$ of \eqref{eq:nRsol} together with the initial data $n_R(0,x,y)=n^0(x,y)1_{|y|\leq R}$. More precisely, for $T>0$, there is a constant $C_\gamma>0$ independent from $R>0$ (but that depends on $\gamma$) such that $n_R$ satisfies
\begin{align}\label{est:regnR}
&\sup_{y\in\mathbb R}\left(\sup_{t\in[0,T)}\|n_R(t,\cdot,y)\|_{H^1(\mathbb T^d)}+\|n_R(\cdot,\cdot,y)\|_{L^2([0,T),H^2(\mathbb T^d))}+\|\partial_t n_R(\cdot,\cdot,y)\|_{L^2([0,T)\times \mathbb T^d)}\right) \leq C_\gamma,
\end{align}
and for $(t,x,y)\in [0,\infty)\times\mathbb T^d\times\mathbb R$, there is $\bar C>0$ independent from both $R$ and $\gamma$, such that
\begin{equation}\label{est:uppery2}
 n_R(t,x,y)\leq \frac {\bar C\gamma}{1+y^2}.
\end{equation}
\end{prop}
}\gr{

\begin{proof}[Proof of Proposition~\ref{prop:existencenR}]

Let $\tau\geq 0$. We assume that $n_R$ is a solution of \eqref{eq:nRsol} for $t\in[0,\tau]$ with initial data $(x,y)\mapsto n^0(x,y)1_{|y|\leq R}$. Thanks to an integration of \eqref{eq:nRsol} along $y\in\mathbb R$, $N_R(t,x)=\int n_R(t,x,y)\,dy$ is a \emph{strong} solution of the following parabolic equation (in the sense of Theorem 5(i) in Section 7.1.3 of \cite{Evans}) with $n_R$ as a given coefficient:
\begin{align}
\partial_t N_R(t,x)-\Delta_x N_R(t,x)&\leq  \left[1+\frac A2-N_R(t,x)\right]N_R(t,x)-\frac 12 \int_{-R}^R (y-y_{opt}(t,x))^2n_R(t,x,y)\,dy,\nonumber\\
&\leq \left[1+\frac A2-N_R(t,x)\right]N_R(t,x).\label{eq:subsnR}
\end{align}
We notice that  \rb{$\phi:(t,x)\mapsto \max\left(\int_{\mathbb R} \frac{C_0}{1+y^2}\,dy,1+\frac A2\right)$} satisfies $ N_R(0,x)\leq \phi(0,x)$ thanks to Assumption~\ref{Assumption3}, \rb{and is a super-solution of \eqref{eq:subsnR}. We can then apply the comparison principle} to show that $N_R(t,x)\leq \phi(t,x)$ for $(t,x)\in[0,\tau]\times\mathbb T^d$, that provides a uniform bound $\int_{\mathbb R} n_R(t,x,y)\,dy\leq C$ for some constant $C>0$ that only depends on $A$ and the constant $C_0$ from Assumption~\ref{Assumption3}. Then, 
\begin{align}
\partial_t n_R(t,x,y) -\Delta_x n_R(t,x,y) &\leq \left(1+\frac A2-\frac 12\left(y_{opt}(t,x)^2-2y\,y_{opt}(t,x)+y^2/2\right)-\frac \gamma 2\right)n_R(t,x,y)\nonumber\\
&\quad -\left(\frac \gamma 2+\frac {y^2}4 \right)n_R(t,x,y)+\gamma \Gamma_{A/2}(0)\int_{\mathbb R} n_R(t,x,y)\,dy\label{eq:supersoln2bis}\\
&= \left(\left(1+\frac A2+\frac {y_{opt}(t,x)^2}2 -\frac \gamma 2\right)-\frac 12\left(\sqrt 2y_{opt}(t,x)-\frac y{\sqrt 2}\right)^2\right)n_R(t,x,y)\nonumber\\
&\quad -\left(\frac \gamma 2+\frac {y^2}4 \right)n_R(t,x,y)+\gamma \Gamma_{A/2}(0)N_R(t,x)\nonumber\\
&\leq -\left(\frac \gamma 2+\frac {y^2}4 \right) n_R(t,x,y)+C\gamma,\label{eq:supersoln2}
\end{align}
since we have assumed $\gamma\geq \gr{2}\left(1+\frac A2+\frac 12\|y_{opt}\|_{L^\infty(\mathbb R_+\times\mathbb T^d)}^2\right)$. We notice that for any $y\in [-R,R]$, $(t,x)\mapsto \psi(\gr{x,}y)\gr{:=} \frac{4 C}{1+y^2}\gamma$ is a super-solution of \eqref{eq:supersoln2} that satisfies $\psi(x\gr{, y})\geq n^0(x,y)$ for $(x,y)\in\mathbb T^d\times [-R,R]$, provided $C>C_0/\gamma$, thanks to Assumption~\ref{Assumption3}. For $y\in[-R,R]$, the comparison principle applied to \eqref{eq:supersoln2} shows $n_R(t,x,y)\leq \psi(y)$ for $(t,x)\in[0,\tau]\times\mathbb T^d$. Since this holds for any $y\in[-R,R]$ and $n_R(t,x,y)=0$ if $|y|>R$, we have
\begin{equation}\label{eq:upb}
\forall (t,x,y)\in [0,\tau]\times\mathbb T^d\times\mathbb R,\quad n_R(t,x,y)\leq \frac{4 C}{1+y^2}\gamma.
\end{equation}
Thanks to Lemma~\ref{lem:exist}, we can extend the solution $n_R=n_R(t,x,y)$ into a solution of \eqref{eq:nRsol} on $[0,\tau+\bar t]\times\mathbb T^d\times[-R,R]$, with $n_R(\tau +t,x,y):=\tilde n_R(t,x,y)$, and $\bar t$ as in Lemma~\ref{lem:exist}. We notice that $\bar t$ is independent of $\tau\geq 0$, so that this extension argument can be iterated to construct a global solution of \eqref{eq:nRsol} satisfying \eqref{eq:upb}. Thanks to \eqref{est:uppery2}, for any fixed $\bar y\in[-R,R]$, we can see $(t,x)\mapsto n_R(t,x,\bar y)$ (where $n_R$ is the solution of \eqref{eq:nRsol}) as a solution of the heat equation $\partial_t n_R(t,x,\bar y)-\Delta_x n_R(t,x,\bar y)=f_{\bar y}(t,x)$ with a bounded $0-$order term: $|f_{\bar y}(t,x)|\leq C_\gamma$, where $C_\gamma>0$ is independent from $R$. $(t,x)\mapsto n_R(t,x,\bar y)$ is then a \emph{strong} solution (in the sense of \cite{Evans}) of that heat equation and Theorem 5 in Section 7.1.3 of  \cite{Evans} implies the regularity estimate \eqref{est:regnR}.

\end{proof}
}

\gr{

In Proposition~\ref{prop:existencenR}, we have proven the bound \eqref{est:uppery2}, that is independent from $R>0$, which will be very useful to consider the limit $R\to\infty$ of $n_R$ to construct solutions of (SIM) and prove Proposition~\ref{prop:existSIM}, see Section~\ref{sec:existence}. We can actually improve this estimates on tails of $n_R$, as we show in the following proposition:

\begin{prop}\label{prop:tailsnR}
Let $y_{opt}\in W^{1,\infty}(\mathbb R_+\times\mathbb T^d,\mathbb R)$, $A>0$, $\gamma>2+A+\|y_{opt}\|^2_{L^\infty(\mathbb R_+\times\mathbb T^d)}$, and  $n^0$ satisfying Assumption~\ref{Assumption3}. There is $\bar C>0$ (independent from $\gamma>0$) such that for any $R>1$, the global solution $n_R$ of \eqref{eq:nRsol}  with the initial data $n_R(0,x,y)=n^0(x,y)1_{|y|\leq R}$ satisfies 
\begin{equation}\label{est:uppery8}
 n_R(t,x,y)\leq \frac {\bar C\gamma}{1+y^{10}}.
\end{equation}
Moreover, for $T>0$, there is a constant $C_{T,\gamma}$ that may depend on $T$ and $\gamma$, but that is uniform in $R$, such that
\begin{align}\label{est:regn2}
&\textrm{sup}_{t\in[0,T]}\|n_R(t,\cdot,y)\|_{H^1(\mathbb T^d)}+\|n_R(\cdot,\cdot,y)\|_{L^2([0,T],H^2(\mathbb T^d))}+\|\partial_t n_R(\cdot,\cdot,y)\|_{L^2([0,T]\times \mathbb T^d)} \leq \frac {C_{T,\gamma}}{1+y^8}.
\end{align}
\end{prop}

\begin{proof}[Proof of Proposition~\ref{prop:tailsnR}]

Note that estimate \eqref{est:uppery2} provides an upper bound on $y^2n_R(t,x,y)$ that is uniform in $(t,x)\in\mathbb R_+\times\mathbb T^d$ and that does not depend on $R>0$. To improve this tail estimate further, we decompose the birth term as follows:
\begin{align}
&\iint_{\mathbb R^2} \Gamma_{A/2}\left(y-\frac{y_*+y_*'}2\right)\frac{n_R(t,x,y_*)n_R(t,x,y_*')}{\int_{\mathbb R} n_R(t,x,z)\,dz}\,dy_*\,dy_*'\nonumber\\
&\quad =\iint_{[-|y|/4,|y|/4]^2} \Gamma_{A/2}\left(y-\frac{y_*+y_*'}2\right)\frac{n_R(t,x,y_*)n_R(t,x,y_*')}{\int_{\mathbb R} n_R(t,x,z)\,dz}\,dy_*\,dy_*'\nonumber\\
&\qquad +\iint_{\mathbb R^2\setminus[-|y|/4,|y|/4]^2} \Gamma_{A/2}\left(y-\frac{y_*+y_*'}2\right)\frac{n_R(t,x,y_*)n_R(t,x,y_*')}{\int_{\mathbb R} n_R(t,x,z)\,dz}\,dy_*\,dy_*'.\label{est:I1I2}
\end{align}
We can estimate the first term on the right hand side of \eqref{est:I1I2} as follows:
\begin{align}
&\iint_{[-|y|/4,|y|/4]^2} \Gamma_{A/2}\left(y-\frac{y_*+y_*'}2\right)\frac{n_R(t,x,y_*)n_R(t,x,y_*')}{\int_{\mathbb R} n_R(t,x,z)\,dz}\,dy_*\,dy_*'\nonumber\\
&\quad \leq \left(\max_{y_*,y_*'\in[-|y|/4,|y|/4]}\Gamma_{A/2}\left(y-\frac{y_*+y_*'}2\right)\right)\iint_{[-|y|/4,|y|/4]^2} \frac{n_R(t,x,y_*)n_R(t,x,y_*')}{\int n_R(t,x,z)\,dz}\,dy_*\,dy_*'\nonumber\\
&\quad \leq \Gamma_{A/2}\left(\frac{3|y|}4\right)\int n_R(t,x,z)\,dz\leq Ce^{-\frac {9y^2}{16A}},\label{est:I1}
\end{align}
since $y_*,y_*'\in[-|y|/4,|y|/4]$ implies $\Gamma_{A/2}\left(y-\frac{y_*+y_*'}2\right)\leq \Gamma_{A/2}\left(3|y|/4\right)$. To estimate the last term of \eqref{est:I1I2}, we take advantage of \eqref{est:uppery2} to show:
\begin{align}
&\iint_{\mathbb R^2\setminus[-|y|/4,|y|/4]^2} \Gamma_{A/2}\left(y-\frac{y_*+y_*'}2\right)\frac{n_R(t,x,y_*)n_R(t,x,y_*')}{\int_{\mathbb R} n_R(t,x,z)\,dz}\,dy_*\,dy_*'\nonumber\\
&\quad \leq\int_{\mathbb R}\int_{|y_*|\geq |y|/4} \Gamma_{A/2}\left(y-\frac{y_*+y_*'}2\right)\frac{n_R(t,x,y_*)n_R(t,x,y_*')}{\int_{\mathbb R} n_R(t,x,z)\,dz}\,dy_*\,dy_*'\nonumber\\
&\qquad +\int_{|y_*'|\geq |y|/4}\int_{\mathbb R} \Gamma_{A/2}\left(y-\frac{y_*+y_*'}2\right)\frac{n_R(t,x,y_*)n_R(t,x,y_*')}{\int_{\mathbb R} n_R(t,x,z)\,dz}\,dy_*\,dy_*'\nonumber\\
&\quad \leq 2\int_{\mathbb R} \left(\int\Gamma_{A/2}\left(y-\frac{y_*+y_*'}2\right)\,dy_*\right)\left(\max_{z\in [-|y|/4,|y|/4]^c}|n_R(t,x,z)|\right)\frac{n_R(t,x,y_*')}{\int_{\mathbb R} n_R(t,x,z)\,dz}\,dy_*'\nonumber\\
&\quad \leq C\max_{z\in [-|y|/4,|y|/4]^c}|n_R(t,x,z)|\leq  \frac {C\gamma}{1+(|y|/4)^2}\leq \frac {C\gamma}{1+y^2},\label{est:I2}
\end{align}
for some constant $C>0$. We can reproduce here the argument developed in \eqref{eq:supersoln2bis}-\eqref{eq:supersoln2}, and use the estimates \eqref{est:I1I2}, \eqref{est:I1} and \eqref{est:I2} on the birth term to show 
\begin{align}
&\partial_t n_R(t,x,y) -\Delta_x n_R(t,x,y)\nonumber\\ 
&\quad \leq -\left(\frac \gamma 2+\frac {y^2}4 \right) n_R(t,x,y)+\iint_{\mathbb R^2} \Gamma_{A/2}\left(y-\frac{y_*+y_*'}2\right)\frac{n_R(t,x,y_*)n_R(t,x,y_*')}{\int_{\mathbb R} n_R(t,x,z)\,dz}\,dy_*\,dy_*'\nonumber\\
&\quad \leq   -\left(\frac \gamma 2+\frac {y^2}4 \right) n_R(t,x,y)+\frac {C\gamma}{1+y^2}.\label{eq:supersoln3}
\end{align}
We notice that for any $y\in[-R,R]$, $(t,x)\mapsto n_R(t,x,y)$ satisfies \eqref{eq:supersoln3}, while 
\[(t,x)\mapsto \psi(y):=\frac{\max(4 C\gamma,\bar C)}{1+y^4}\]
is a super-solution of \eqref{eq:supersoln3} that satisfies $\psi(y)\geq n^0(x,y)$ for $x\in\mathbb T^d$ thanks to Assumption~\ref{Assumption}. The comparison principle then implies $n_R(t,x,y)\leq \frac{C}{1+y^4}\gamma$ for $(t,x,y)\in\mathbb R_+\times\mathbb T^d\times [-R,R]$ and a constant $C>0$ independent from $R>0$. Since this holds for any $y\in\mathbb R$, we have
\begin{equation}\label{est:uppery40}
\forall (t,x,y)\in \mathbb R_+\times\mathbb T^d\times\mathbb R,\quad n_R(t,x,y)\leq \frac{C\gamma}{1+y^4}.
\end{equation}
This estimate can be used to obtain a better estimate \eqref{est:I2}. It then becomes
\begin{align*}
I_2(t,x,y)&\leq C\max_{z\in [-|y|/4,|y|/4]^c}|n_R(t,x,z)|\leq  \frac {C\gamma}{1+(|y|/4)^4}\leq \frac {C\gamma}{1+y^4},
\end{align*}
the argument above can be repeated to show that $n_R(t,x,y)\leq \frac{C\gamma}{1+y^6}$ for some constant $C>0$ independent of $R$. This estimate can be used again twice to improve the estimate \eqref{est:I2}, and iterate the argument once more to show \eqref{est:uppery8}.

\medskip

In the last iteration above, we have derived the following bound on the birth term:
\[\left|\gamma \iint_{\mathbb R^2} \Gamma_{A/2}\left(y-\frac{y_*+y_*'}2\right)\frac{n_R(t,x,y_*)n_R(t,x,y_*')}{\int_{\mathbb R} n_R(t,x,z)\,dz}\,dy_*\,dy_*'\right|\leq \frac {C\gamma^2}{1+y^{8}},\]
while the estimate \eqref{est:uppery8} implies
\[\left|\left(1+\frac A2-\gamma-\frac 12 (y-y_{opt}(t,x))^2-\int_{\mathbb R}n_R(t,x,z)\,dz\right)n_R(t,x,y)\right|\leq \frac {C\gamma(1+\gamma)}{1+y^8},\]
so that for $y\in[-R,R]$ and $T>0$, we have on $[0,T]\times\mathbb T^d$: 
\begin{equation}\label{est:nR}
\left|\partial_t n_R(t,x,y)-\Delta_x n_R(t,x,y)\right|\leq \frac {C\gamma(1+\gamma)}{1+y^8}.
\end{equation}
We can then apply Theorem 5  in Section 7.1.3 of \cite{Evans} to \eqref{est:nR} for each $y\in\mathbb R$, which implies \eqref{est:regn2}.

\end{proof}

\begin{rem}\label{rem:regnR}
In \eqref{est:nR}, we have an $L^\infty$ bound on $\partial_t n_R(t,x,y)-\Delta_x n_R(t,x,y)$, which is stronger than the $L^2$ estimate necessary to apply Theorem 5  in Section 7.1.3 of \cite{Evans}. We can take advantage of this to obtain a stronger regularity result: Thanks to Theorem 7.22 in \cite{Liebermann}, we have, for $y\in[-R,R]$, that the function $(t,x)\mapsto n_R(t,x,y)$ belongs to $W^{2,1}_{d+2}([0,T]\times\mathbb T^d)$ (we use here the notation $W^{2,1}_{d+2}$ from \cite{Liebermann}: the time derivative belongs to $L^{d+2}$, and its second spatial derivative belongs to $L^{d+2}$.). Moreover, \eqref{est:nR} implies that the $W^{2,1}_{d+2}([0,T]\times\mathbb T^d)$ norm of $(t,x)\mapsto n_R(t,x,y)$ is dominated by $\frac C{1+y^8}$, with a constant $C>0$ that may depend on $T$ and $\gamma>0$, but that is independent from $R>0$. This weighted regularity estimate will be useful in Section~\ref{subsec:regNV}
\end{rem}
}

\subsection{Regularity of the macroscopic quantities $N$, $Z$ and $V$}\label{subsec:regNV}
\gr{
Let $n$ a solution of (SIM). Thanks to the uniform regularity of $n_R$ described in Remark~\ref{rem:regnR} and an integration against $1$, $y$ and $y^4$ respectively, $N(t,x)$, $Y(t,x):=\int_{\mathbb R} yn(t,x,y)\,dy$ and $W(t,x):=\int_{\mathbb R} y^4n(t,x,y)\,dy$ satisfy
\[N,Y,W\in W^{2,1}_{d+2}([0,T]\times\mathbb T^d),\]
where we use here the notation $W^{2,1}_{d+2}$ from \cite{Liebermann}: the time derivative of the function belongs to $L^{d+2}$, and its second spatial derivative belongs to $L^{d+2}$.

Let $\varphi\in L^\infty([0,T]\times\mathbb T^d)$, and $R>0$. If $n$ is a solution of (SIM) and $\varphi\in C^2([0,T]\times\mathbb T^d)$, we can use $\varphi(t,x)1_{[-R,R]}$ (resp. $\varphi(t,x)y1_{[-R,R]}$, $\varphi(t,x)y^41_{[-R,R]}$) as a test function, and let $R\to \infty$ to show that $N$ solves \eqref{eq:Nmod}, while $Y$ and $W$ solve
\begin{align}\label{eq:Y}
\partial_t Y(t,x)-\Delta_x Y(t,x)=\left(1+\frac A2-N(t,x)\right)Y(t,x)-\frac 12\int_{\mathbb R}y(y-y_{opt})^2n(t,x,y)\,dy,
\end{align}
\begin{align}\label{eq:W}
&\partial_t W(t,x)-\Delta_x W(t,x)=\left(1+\frac A2-N(t,x)\right)W(t,x)-\frac 12\int_{\mathbb R}y^4(y-y_{opt})^2n(t,x,y)\,dy\\
&\quad +\gamma\left(N(t,x)\int_{\mathbb R} |y|^4 T(\tilde n(t,x,\cdot))(y)\,dy-W(t,x)\right).
\end{align}

We can notice that the right hand side of the equations \eqref{eq:Nmod}, \eqref{eq:Y} and \eqref{eq:W} after a derivation in $t$ or in $x$, belong to $L^{d+2}([0,T]\times\mathbb T^d)$ uniformly in $R$, thanks to Remark~\ref{rem:regnR}. We can then apply Theorem 7.22  in \cite{Liebermann} to show that $\partial_t N,\,\partial_t Y,\,\partial_t W,\,\nabla_x N,\, \nabla_x Y$ and $\nabla_x W$  belong to $W^{2,1}_{d+2}([0,T]\times\mathbb T^d)$. Morrey's inequality then implies that $N,Y,W$ are $C^{1}$ functions in the $t$ variable, and $C^2$ functions in the $x$ variable. They are then \emph{classical solutions} of \eqref{eq:Nmod}, \eqref{eq:Y} and \eqref{eq:W}: they satisfy these equalities pointwise, as an equality between continuous functions. In particular, $\partial_t$ and $\Delta_x$ are the actual differential operators, and we can compute $(\partial_t-\Delta_x)\frac ZN$ to show \eqref{eq:Zmod}.

\medskip

In the above argument, we use interior parabolic estimates, \rb{specifically Theorem 7.22  in \cite{Liebermann}}. These estimates do not hold up to the boundary $\{t=0\}$ directly, but it is possible to use traditional ideas to get around this difficulty: it is possible to transform these problems into equations that hold for $t\in\mathbb R$ thanks to the regularity of the initial condition (provided by Assumption~\ref{Assumption}), we refer to Step~4 of the proof of Proposition~\ref{prop:regularityNZ} where we develop the details of a similar argument.

}

\subsection{Proof of Proposition~\ref{prop:uniqueness}: uniqueness of solutions of (SIM)}
\gr{
Since $n$ and $\tilde n$ are  solutions of (SIM), we can reproduce the lower bound argument made in the proof of Proposition~\ref{prop:existSIM} (see \eqref{eq:eqsubsol}-\eqref{eq:lowern}), that is based on the comparison principle, to show that the exists $C_T>0$ such that $\int_{\mathbb R}n(t,x,y)\,dy\geq C_T$ and $\int_{\mathbb R}\tilde n(t,x,y)\,dy\geq C_T$.

Thanks to \cite{Evans} (Theorem 3, p. 287) and using the fact that $n(0,x,y)=\tilde n(0,x,y)=n^0(x,y)$ for $(x,y)\in\mathbb T^d\times\mathbb R$, we obtain that for $t\in[0,T]$,
\begin{align}
&\int_{\mathbb R}\int_{\mathbb T^d} (n(t,x,y)-\tilde n(t,x,y))^2(1+y^4)\,dx \,dy+\int_0^t
\int_{\mathbb R}\int_{\mathbb T^d} |\nabla_xn(s,x,y)-\nabla_x\tilde n(s,x,y)|^2(1+y^4)\,dx \,dy\,ds\nonumber\\
&\quad =\int_0^t\int_{\mathbb T^d}\int_{\mathbb R} \left(1+\frac A2-\gamma-\frac 12(y-y_{opt}(s,x))^2-\int_{\mathbb R} n(s,x,z)\,dz\right)(n(s,x,y)-\tilde n(s,x,y))^2(1+y^4)\,dy\,dx\,ds\nonumber\\
&\qquad +\int_0^t\int_{\mathbb T^d}\int_{\mathbb R}\left(\int_{\mathbb R} n(s,x,z)-\tilde n(s,x,z)\,dz\right)\tilde n(s,x,y)(n(s,x,y)-\tilde n(s,x,y))(1+y^4)\,dy\,dx\,ds\nonumber\\
&\qquad + \gamma\int_0^t\int_{\mathbb R} \int_{\mathbb T^d}(n(s,x,y)-\tilde n(s,x,y))\bigg[-(n(s,x,y)-\tilde n(s,x,y))\nonumber\\
&\qquad \iint_{\mathbb R^2} \Gamma_{A/2}\left(y-\frac{y_*+y_*'}2\right)\left(\frac{n(s,x,y_*)n(s,x,y_*')}{\int_{\mathbb R} n(s,x,z)\,dz}-\frac{\tilde n(s,x,y_*)\tilde n(s,x,y_*')}{\int_{\mathbb R} \tilde n(s,x,z)\,dz}\right)\,dy_*\,dy_*'\bigg](1+y^4)\,dx\,dy\,ds\nonumber\\
&\quad \leq \left(1+\frac A2- \gamma\right) \int_0^t\int_{\mathbb R}\int_{\mathbb T^d} (n(s,x,y)-\tilde n(s,x,y))^2(1+y^4)\,dx \,dy\,ds\nonumber\\
&\qquad + \left(\max_{s\in[0,t],x\in\mathbb T^d}\int_{\mathbb R}n(s,x,y)\sqrt{1+y^4}\,dy \right)^{\frac 12}\nonumber\\
&\qquad \int_0^t\left(\int_{\mathbb R}\int_{\mathbb T^d} (n(s,x,y)-\tilde n(s,x,y))^2(1+y^4)\,dx \,dy\right)^{\frac 12}\left(\int_{\mathbb R}\int_{\mathbb T^d} |n(s,x,y)-\tilde n(s,x,y)|\,dx \,dy\right)\,ds\label{est:uni2}\\
&\qquad  +\frac \gamma 2\int_0^t\left\{\int_{\mathbb R}\int_{\mathbb T^d} (n(s,x,y)-\tilde n(s,x,y))^2(1+y^4)\,dx \,dy\right\}^{\frac 12}\bigg\{\int_{\mathbb R}\int_{\mathbb T^d} \bigg[\iint_{\mathbb R^2} \Gamma_{A/2}\left(y-\frac{y_*+y_*'}2\right)\nonumber\\
&\qquad \bigg(\frac{n(s,x,y_*)n(s,x,y_*')}{\int_{\mathbb R} n(s,x,z)\,dz}-\frac{\tilde n(s,x,y_*)\tilde n(s,x,y_*')}{\int_{\mathbb R} \tilde n(s,x,z)\,dz}\bigg)\,dy_*\,dy_*'\bigg]^2(1+y^4)\,dx\,dy\bigg\}^{\frac 12}\,ds,\label{est:uni}
\end{align}
and we can estimate the last term appearing in brackets in \eqref{est:uni} as follows:
\begin{align*}
\{\cdot\}&\leq C\int_{\mathbb T^d} \left[\max_{y\in\mathbb R}\iint_{\mathbb R^2} \Gamma_{A/2}\left(y-\frac{y_*+y_*'}2\right)(1+y^2)\bigg|\frac{n(t,x,y_*)n(t,x,y_*')}{\int_{\mathbb R} n(t,x,z)\,dz}-\frac{\tilde n(t,x,y_*)\tilde n(t,x,y_*')}{\int_{\mathbb R} \tilde n(t,x,z)\,dz}\bigg|\,dy_*\,dy_*'\right]\\
&\quad \left[\int_{\mathbb R}\iint_{\mathbb R^2} \Gamma_{A/2}\left(y-\frac{y_*+y_*'}2\right)(1+y^2)\bigg|\frac{n(t,x,y_*)n(t,x,y_*')}{\int_{\mathbb R} n(t,x,z)\,dz}-\frac{\tilde n(t,x,y_*)\tilde n(t,x,y_*')}{\int_{\mathbb R} \tilde n(t,x,z)\,dz}\bigg|\,dy_*\,dy_*'\,dy\right]\,dx.
\end{align*}
We notice the following estimates satisfied by the reproduction kernel $\Gamma_{A/2}$:
\begin{align*} \Gamma_{A/2}\left(y-\frac{y_*+y_*'}2\right)(1+y^2)&\leq  \Gamma_{A/2}\left(y-\frac{y_*+y_*'}2\right)C\left(1+\left(y-\frac{y_*+y_*'}2\right)^2+y_*^2+(y_*')^2\right)\\
&\leq C\left(1+y_*^2+(y_*')^2\right),
\end{align*}
for any $y\in\mathbb R$, as well as
\begin{align*}
\int_{\mathbb R}\Gamma_{A/2}\left(y-\frac{y_*+y_*'}2\right)(1+y^2)\,dy&\leq  \int_{\mathbb R}\Gamma_{A/2}\left(y-\frac{y_*+y_*'}2\right)C\left(1+\left(y-\frac{y_*+y_*'}2\right)^2+y_*^2+(y_*')^2\right)\,dy\\
&\leq C\left(1+y_*^2+(y_*')^2\right),
\end{align*}
and then,
\begin{align*}
&\iint_{\mathbb R^2}C\left(1+y_*^2+(y_*')^2\right)\bigg|\frac{n(t,x,y_*)n(t,x,y_*')}{\int_{\mathbb R} n(t,x,z)\,dz}-\frac{\tilde n(t,x,y_*)\tilde n(t,x,y_*')}{\int_{\mathbb R} \tilde n(t,x,z)\,dz}\bigg|\,dy_*\,dy_*'\\
&\quad \leq \iint_{\mathbb R^2}C\left(1+y_*^2+(y_*')^2\right)\bigg|\frac{\left(n(t,x,y_*)-\tilde n(t,x,y_*)\right)n(t,x,y_*')}{\int_{\mathbb R} n(t,x,z)\,dz}\nonumber\\
&\qquad +\frac{\tilde n(t,x,y_*)n(t,x,y_*')}{\left(\int_{\mathbb R} n(t,x,z)\,dz\right)\left(\int_{\mathbb R} \tilde n(t,x,z)\,dz\right)}\left(\int_{\mathbb R} \tilde n(t,x,z)- n(t,x,z)\,dz\right)\nonumber\\
&\qquad + \frac{\tilde n(t,x,y_*)\left(n(t,x,y_*')-\tilde n(t,x,y_*')\right)}{\int_{\mathbb R} \tilde n(t,x,z)\,dz}\bigg|\,dy_*\,dy_*'\\
&\quad \leq C\int_\mathbb R (1+y^2)\left|n(t,x,y)-\tilde n(t,x,y)\right|\,dy.
\end{align*}
We can then continue the estimation of the last term appearing in brackets in \eqref{est:uni}:
\begin{align*}
\{\cdot\}&\leq C\int_{\mathbb T^d}\left[\int_\mathbb R (1+y^2)\left|n(t,x,y)-\tilde n(t,x,y)\right|\,dy\right]^2\,dx\\
&\leq C\int_{\mathbb T^d}\left[\int_\mathbb R \frac 1{1+y^2}\,dy\right]\left[\int_\mathbb R (1+y^2)^2\left|n(t,x,y)-\tilde n(t,x,y)\right|^2\,dy\right]\,dx\\
&\leq C \int_{\mathbb T^d}\int_\mathbb R \left|n(t,x,y)-\tilde n(t,x,y)\right|^2\left(1+y^4\right)\,dy\,dx
\end{align*}
Similarly, we can estimate the last factor of \eqref{est:uni2}:
\[\int_{\mathbb R}\int_{\mathbb T^d} |n(s,x,y)-\tilde n(s,x,y)|\,dx \,dy\leq \left(\int_{\mathbb R}\int_{\mathbb T^d} |n(s,x,y)-\tilde n(s,x,y)|^2(1+y^4)\,dx \,dy\right)^{\frac 12}.\]
Brought together, these estimates imply
\begin{align*}
&\int_{\mathbb R}\int_{\mathbb T^d} (n(t,x,y)-\tilde n(t,x,y))^2(1+y^4)\,dx \,dy+\int_0^t
\int_{\mathbb R}\int_{\mathbb T^d} |\nabla_xn(s,x,y)-\nabla_x\tilde n(s,x,y)|^2(1+y^4)\,dx \,dy\,ds\nonumber\\
&\quad \leq C\int_0^t \int_{\mathbb R}\int_{\mathbb T^d} (n(s,x,y)-\tilde n(s,x,y))^2(1+y^4)\,dx \,dy\,ds.
\end{align*}
In particular $y(t):=\int_{\mathbb R}\int_{\mathbb T^d} (\tilde n(t,x,y)- n(t,x,y))^2\left(1+y^4\right)\,dx\,dy$ satisfies $y(t)\leq C\int_0^t y(s)\,ds$, and a Gronwall estimate shows that $y(t)=0$ for $t\in[0,T]$, that is $n=\tilde n$, which concludes the proof of Proposition~\ref{prop:uniqueness}.
}

\subsection{Technical estimates for some linear problems}\label{subsec:tecnical}
In this section, we derive estimates on solutions of linear parabolic problems that are used in Section~\ref{subsec:Wasserstein} (proof of Proposition~\ref{prop:Wassersteincontraction}). We consider the assumption made in Proposition~\ref{prop:Wassersteincontraction}, and in particular: $y_{opt},\,n^0$ satisfying Assumption~\ref{Assumption}, $n\in L^\infty(\mathbb R_+\times\mathbb T^d,L^1((\mathbb R))$ a solution of  (SIM) with initial condition $n^0$, and $\tilde n$, $N$, $Z$ defined by \eqref{def:tilden} and \eqref{NZ}, and we assume that $\|Z\|_{L^\infty([0,\tau)\times \mathbb T^d)}\leq\kappa$, for some $\tau>0$ and $\kappa>0$.

\medskip

\noindent\textbf{Some linear parabolic equations}

For $(s,z,y)\in [0,\tau)\times \mathbb T^d\times \mathbb R$, let $\phi_{s,z,y}(t,x)$ the solution of \eqref{eq:phi}. Let $(t,x)\mapsto \psi_{s,z,y}(t,x):=\phi_{s,z,y}(t,x)N(t,x)$, which satisfies the following linear parabolic equation:
\begin{equation*}
\left\{
\begin{array}{l}
\partial_t\psi_{s,z,y}(t,x)-\Delta_x\psi_{s,z,y}(t,x)\\
\qquad =\left(1+\frac A 2-N(t,x)-\frac 12 (y-y_{opt}(t,x))^2-\frac 12\int (y-y_{opt}(t,x))^2\tilde n(t,x,y)\,dy\right)\psi_{s,z,y}(t,x) , \\
\qquad (t,x)\in [s,\tau)\times \mathbb T^d,\\
\psi_{s,z,y}(s,x)=N(s,z)\delta_z(x),\; x\in \mathbb T^d.
\end{array}
\right.
\end{equation*}
Since the factor on the right hand side of the equation satisfied by $\psi_{s,z,y}$ is bounded (see Proposition~\ref{prop:moment4}), the existence and uniqueness of $\psi_{s,z,y}$ derives from standards methods (see e.g. Theorem~7.3 and Theorem~7.4 in \cite{Evans}), and this implies the existence and uniqueness of the solution $\phi_{s,z,y}$ of \eqref{eq:phi}.

We also define $\bar \phi_{s,z}(t,x)$ as the solution of:
\begin{equation}\label{eq:barphi}
\left\{
\begin{array}{l}
\partial_t \bar \phi_{s,z}(t,x)-\Delta_x \bar\phi_{s,z}(t,x) = 2\frac{\nabla_x N(t,x)}{N(t,x)} \cdot\nabla_x \bar \phi_{s,z}(t,x), \;(t,x)\in [s,\tau)\times \mathbb T^d,\\
\bar \phi_{s,z}(s,x)=\delta_z(x),\; x\in \mathbb T^d.
\end{array}
\right.
\end{equation}
\gr{Thanks to Section~\ref{subsec:regNV}, the quotient $\frac{\nabla_x N(t,x)}{N(t,x)}$ is continuous, the existence and uniqueness of solution then follows from standard arguments.} Alternatively, we can notice that $\bar \psi_{s,z}(t,x):=\bar \phi_{s,z}N(t,x)$ satisfies
\begin{equation*}
\left\{
\begin{array}{l}
\partial_t\bar\psi_{s,z}(t,x)-\Delta_x\bar \psi_{s,z}(t,x)\\
\quad =\left(1+\frac A 2-N(t,x)-\frac 12\int (y-y_{opt}(t,x))^2\tilde n(t,x,y)\,dy\right)\bar\psi_{s,z}(t,x), \;(t,x)\in [s,\tau)\times \mathbb T^d,\\
\bar \psi_{s,z}(s,x)=N(s,z)\delta_z(x),\; x\in \mathbb T^d,
\end{array}
\right.
\end{equation*}
which is a linear heat equation, and the unique solution is then given by  an explicit Duhamel formula. It can then be used to construct $\bar \phi_{s,z}$.

\medskip

\noindent\textbf{Estimate 1}

Thanks to Proposition~\ref{prop:moment4}, there exists $C_\kappa>0$ such that $\int |y|^4 \tilde n(t,x,y)\,dy\leq C_\kappa$ for any $(t,x)\in [0,\tau)\times \mathbb T^d$, and we can define
\begin{equation}\label{def:R}
R=\left(2C_\kappa\right)^{1/4}.
\end{equation}
Then, for any $(t,x)\in [0,\tau)\times \mathbb T^d$,
\begin{equation}\label{est:phi1}
\int_{-R}^R\tilde n(t,x,y)\,dy=1-\int_{[-R,R]^c}\tilde n(t,x,y)\,dy\geq 1-\frac 1{R^4}\int_{[-R,R]^c}|y|^4\tilde n(t,x,y)\,dy\geq \frac 12.
\end{equation}
Let also
\begin{equation}\label{def:rho}
R'=R+\|y_{opt}\|_{L^\infty(\mathbb R_+\times \mathbb T^d)}.
\end{equation}
Then, for any $y\in[-R',R']^c$, we have $-\frac 12(y-y_{opt}(t,x))^2\leq \min_{\tilde y\in[-R,R]}\left(-\frac 12\left(\tilde y-y_{opt}(t,x)\right)^2\right)$. The maximum principle (see Remark~\ref{rem:comparaison}) applies to \eqref{eq:phi} (comparing the case $y\in[-R',R']^c$ to the case where $\tilde y\in[-R,R]$), and then, for $y\in[-R',R']^c$,
\begin{equation}\label{est:phi2}
\forall (s,z)\in[0,\tau)\times \mathbb T^d,\,\forall (t,x)\in [s,\tau)\times\mathbb T^d,\quad \phi_{s,z,y}(t,x)\leq \min_{|\tilde y|\leq R}\phi_{s,z,\tilde y}(t,x).
\end{equation}

\medskip

\noindent\textbf{Estimate 2}

For any $y\in\mathbb R$, 
\[-\frac 12\left(y+\textrm{sgn}(y)\|y_{opt}\|_{L^\infty(\mathbb R_+\times \mathbb T^d)}\right)^2\leq -\frac 12(y-y_{opt}(t,x))^2\leq -\frac 12\left(y-\textrm{sgn}(y)\|y_{opt}\|_{L^\infty(\mathbb R_+\times \mathbb T^d)}\right)^2\]
Then $\phi_{s,z,y}(t,x)e^{(t-s)\frac 12\left(y+\textrm{sgn}(y)\|y_{opt}\|_{L^\infty(\mathbb R_+\times \mathbb T^d)}\right)^2}$ is a super-solution of \eqref{eq:barphi}, and thanks to the comparison principle, $\bar \phi_{s,z}(t,x)\leq \phi_{s,z,y}(t,x)e^{(t-s)\frac 12\left(y+\textrm{sgn}(y)\|y_{opt}\|_{L^\infty(\mathbb R_+\times \mathbb T^d)}\right)^2}$. The reverse estimate can be obtained similarily, and together, those estimates imply for any $(s,z,y)\in [0,\tau)\times \mathbb T^d\times\mathbb R$ and $(t,x)\in (s,\min(s+1,\tau))\times \mathbb T^d$,
\begin{equation}\label{est:barphi1}
\phi_{s,z,y}(t,x)= \bar \phi_{s,z}(t,x)e^{-(t-s)\frac 12 (y+\mathcal O(1))^2},
\end{equation}
where $|\mathcal O(1)|\leq \|y_{opt}\|_{L^\infty(\mathbb R_+\times \mathbb T^d)}$.

\medskip

\noindent\textbf{Estimate 3}

$\bar\psi_{s,z}$ satisfies $\bar\psi_{s,z}(s,\cdot)=N(s,z)\delta_z$ and
\[\partial_t\bar\psi_{s,z}(t,x)-\Delta_x\bar\psi_{s,z}(t,x)\leq \left(1+\frac A2\right)\bar\psi_{s,z}(t,x).\]
Thanks to the comparison principle, $\bar\psi_{s,z}(t,x)\leq N(s,z)e^{\left(1+\frac A2\right)(t-s)}\Gamma_{t-s}(x-z)$, \rb{ where the notation $(t,x)\mapsto \Gamma_t(x)$ designates the fundamental solution of the heat equation on $\mathbb T^d$, and $x-y$ stands for the substraction of $x$ by $y$ on that torus. Since} $\bar \psi_{s,z}(t,x)=\bar \phi_{s,z}(t,x)N(t,x)$, we have\rb{, for $\theta\in(0,1)$,}
\[\int_{\mathbb R} \bar\phi_{s,z}(t,x)|z-x|^\theta\,dz\leq e^{\left(1+\frac A2\right)(t-s)}\int_{\mathbb R}\Gamma_{t-s}(x-z)\frac{N(s,z)}{N(t,x)}|z-x|^\theta\,dz.\]
We can use the estimate \eqref{est:lowerN2} to show that $\left|\frac{N(s,z)}{N(t,x)}\right|\leq C_\kappa$, as soon as $1<s\leq t\leq \min(s+1,\tau)$. If $0\leq s\leq t\leq 2$, we can use the lower bound \eqref{est:lowerN1} and the upper bound $\|N\|_{L^\infty(\mathbb R_+\times\mathbb T^d)}\leq \max\left(1,\|N(0,\cdot)\|_{L^\infty(\mathbb T^d)}\right)$ to obtain a similar estimate. Then,
\begin{equation}\label{est:barphi3}
\int_{\mathbb R} \bar\phi_{s,z}(t,x)|z-x|^\theta\,dz\leq  C_\kappa e^{\left(1+\frac A2\right)(t-s)}\int_{\mathbb R} \Gamma_{t-s}(x-z)|z-x|^\theta\,dz\leq C_\kappa  (t-s)^{\frac\theta 2},
\end{equation}
provided $0<s\leq t\leq \min(s+1,\tau)$.

\medskip

\noindent\textbf{Estimate 4}

For $(\bar t,\bar x)\in(0,+\infty)\times \mathbb T^d$, let $u_{\bar t,\bar x}$ the solution of the following problem (note that the \emph{time variable} is here reversed compared to usual problems)
\begin{equation}\label{eq:dual}
\left\{\begin{array}{l}
-\frac {\partial u_{\bar t,\bar x}}{\partial t}(t,x)-\Delta_{x} u_{\bar t,\bar x}(t,x)= -2\nabla_x\cdot\left(\frac{\nabla_x N(t,x)}{N(t,x)}u_{\bar t,\bar x}(t,x)\right),\quad (t,x)\in (-\infty,\bar t]\times \mathbb T^d\\
u_{\bar t,\bar x}(\bar t,x)=\delta_{\bar x}(x).
\end{array}\right.
\end{equation}
This problem is indeed the dual problem of \eqref{eq:barphi} in the sense that $\frac d{dt}\int_{\mathbb R} \bar \phi_{s,z}(t,x)u_{\bar t,\bar x}(t,x)\,dx=0$ for $t\in[s,\bar t]$. It follows that for any $s<\bar t$ and $z\in\mathbb T^d$, 
\[\int_{\mathbb R} \bar \phi_{s,z}(s,x)u_{\bar t,\bar x}(s,x)\,dx=\int_{\mathbb R} \bar \phi_{s,z}(t,x)u_{\bar t,\bar x}(\bar t,x)\,dx,\]
which, given the initial conditions specified in \eqref{eq:barphi} and \eqref{eq:dual} (note that the reversion of time in this dual problem implies that the initial condition holds for the largest time considered, ie $t=\bar t$), is equivalent to
\[u_{\bar t,\bar x}(s,z)=\bar \phi_{s,z}(\bar t,\bar x).\]
The divergence form of \eqref{eq:dual} implies that $\int_{\mathbb R} u_{\bar t,\bar x}(s,z)\,dz=\int_{\mathbb R} u_{\bar t,\bar x}(\bar t,z)\,dz=1$, and then for any $\bar t>s$ and $\bar x\in\mathbb T^d$,

\begin{equation}
\int_{\mathbb R} \bar \phi_{s,z}(\bar t,\bar x)\,dz=1.\label{est:barphi2}
\end{equation}

\subsection{Existence, uniqueness and stability of solutions of (KBM)}\label{Appendix:uniquenessSIM}

\gr{In this section, we first show that Proposition~\ref{prop:Linftyboundwithoutalpha} implies the existence of solutions to (KBM). Note that there are probably more direct proofs of this result. The stability of solutions of (KBM) is then proven Proposition~\ref{prop:uniqNZ}, which also implies the uniqueness of solutions. Note that thsi stability result is important to show \eqref{eq:cv}}

\medskip

\gr{We consider the assumptiuons made in Proposition~\ref{prop:Linftyboundwithoutalpha}, denote by $n_\gamma$ the solution of (SIM) and $N_\gamma$, $Z_\gamma$ defined by \eqref{NZ}.  Thanks to Proposition~\ref{prop:Linftyboundwithoutalpha}, $Z_\gamma$ is uniformly bounded for $(t,x)\in\mathbb R_+\times\mathbb T^d$ and $\gamma>\bar \gamma$.} We may then apply  Proposition~\ref{prop:regularityNZ} to show that $N_\gamma,Z_\gamma$ are uniformly $\theta-$Lipschitz continuous. Thanks to Ascoli, they converge to a limit $(\bar N,\bar Z)$ up to an extraction. We consider this subsequence from now on (the uniqueness of solutions proven in Proposition~\ref{prop:uniqueness} will show that the convergence holds without taking a subsequence). This limit is then  itself $\theta-$Lipschitz continuous. We can define $Y_\gamma:=N_\gamma Z_\gamma$, $\bar Y:=\bar N\bar Z$, and note that $\|N_\gamma\|_{L^\infty([0,\tau)\times\mathbb T^d)}+\|Y_\gamma\|_{L^\infty([0,\tau)\times\mathbb T^d)}\leq C$. Moreover, $(N_\gamma,Y_\gamma)$ satisfies a system of equations where only the $0^{th}$ order terms are non-linear (note that we already considered $Y_\gamma$ in Section~\ref{subsec:regNV}, see \eqref{eq:Y}):
\begin{equation}\label{eq:NA}
\left\{\begin{array}{l}
\partial_t N_\gamma(t,x)-\Delta_x N_\gamma(t,x)=\left(1-\frac 1{2}\left(Z_\gamma(t,x)-y_{opt}(t,x)\right)^2-N_\gamma(t,x)+\varphi_{N,\gamma}(t,x)\right)N_\gamma(t,x),\\
\partial_tY_\gamma(t,x)-\Delta_x Y_\gamma(t,x)=\left(1-\frac 1{2}\left(Z_\gamma(t,x)-y_{opt}(t,x)\right)^2-N_\gamma(t,x)+\varphi_{N,\gamma}(t,x)\right)Y_\gamma(t,x)\\
\qquad+\left(-A\left(Z_\gamma(t,x)-y_{opt}(t,x)\right)+\varphi_{Z,\gamma}(t,x)\right)N_\gamma(t,x),
\end{array}
\right.
\end{equation}
\gr{Applying Theorem 7.22 in \cite{Liebermann} to these equations show that $N_\gamma,Y_\gamma\in W_{d+2}^{2,1}(\mathbb R_+\times\mathbb T^d)$ (using the notations of \cite{Liebermann}) with a norm that is uniform in $\gamma>\bar \gamma$. Since $\bar N$ and $\bar Z$ are bounded, the factor of $N_\gamma$ on the right hand side of the first equation in \eqref{eq:NA} is uniformly bounded, and then 
\begin{equation}\label{est:loNbarN}
N_\gamma(t,x)\geq Ce^{-t/C},\quad \bar N(t,x)\geq Ce^{-t/C},
\end{equation}
for $(t,x)\in \mathbb R_+\times\mathbb T^d$. The second inequality being a consequence of the uniform bound on $N_\gamma$ and the convergence of $N_\gamma$ to $\bar N$ when $\gamma\to\infty$. This lower bound on $\bar N$ can be used in \eqref{eq:NA} to show that $\bar N$ and $\bar Y$ satisfy, in a weak sense (integrated against smooth test functions),
\begin{align}
\partial_t \bar N(t,x)-\Delta_x \bar N(t,x)&=\left(1-\frac 1{2}\left(\bar Z(t,x)-y_{opt}(t,x)\right)^2-\bar N(t,x)\right)\bar N(t,x),\nonumber\\
\partial_t\bar Y(t,x)-\Delta_x \bar Y(t,x)&=\left(1-\frac 1{2\bar N(t,x)^2}\left(\bar Y(t,x)-y_{opt}(t,x)\bar N(t,x)\right)^2-\bar N(t,x)\right)\bar Y(t,x)\nonumber\\
&\quad-A\left(\bar Y(t,x)-y_{opt}(t,x)\bar N(t,x)\right).\label{eq:barNY}
\end{align}
Since $\bar N,\bar Y\in W_{d+2}^{2,1}(\mathbb R_+\times\mathbb T^d)$ the right hand side of these equations belong to $L^{d+2}([0,T]\times\mathbb T^d)$ for any $T>0$. Applying Theorem 7.22 in \cite{Liebermann} and Morrey's inequality, we have that $\bar N$, $\bar Y$ are $C^1$ in $t$ and $C^2$ in $x$, and \eqref{eq:barNY} is satisfied as an equality between continuous functions. }Coming back to (KBM) with $\bar Z=\bar Y/\bar N$ shows that $(\bar N,\bar Z)$ is a solution of (KBM). This proves the existence of a solution to (KBM) under  Assumption~\ref{Assumption}.

\medskip

We are now interested in the uniqueness of solutions of (KBM), and use the notation $ W_{d+2}^{1,2}(\mathbb R_+\times\mathbb T^d)$ of \cite{Liebermann} to write the following result:

\gr{
\begin{prop}\label{prop:uniqNZ}
Let $y_{opt}\in C^1(\mathbb R_+\times \mathbb T^d)$, $N^0,Z^0\in W^{2,\infty}(\mathbb T^d)$, such that $N^0>0$. 

Let $\varphi_N,\varphi_Z\in L^\infty(\mathbb R_+\times\mathbb T^d)$ and $N,Z,\bar N, \bar Z\in W_{d+2}^{1,2}$ that are $C^1$ in $t$ and $C^2$ in $x$, with $N(t,x)>0$, $\bar N(t,x)>0$ for $(t,x)\in\mathbb R_+\times\mathbb T^d$. We assume that $(N,Z)$ satisfies 
\begin{equation}\label{eq:NAbis}
\left\{\begin{array}{l}
\partial_t N(t,x)-\Delta_x N(t,x)=\left(1-\frac 1{2}\left(Z(t,x)-y_{opt}(t,x)\right)^2-N(t,x)+\varphi_{N}(t,x)\right)N(t,x),\\
\partial_tZ(t,x)-\Delta_x Z(t,x)=2\frac{\nabla_x N(t,x)}{N(t,x)}\cdot \nabla_x Z(t,x)-A\left(Z(t,x)-y_{opt}(t,x)\right)+\varphi_{Z}(t,x),
\end{array}
\right.
\end{equation}
while $(\bar N,\bar Z)$ satisfies (KBM), as an equality between continuous functions in both cases, and with the same initial condition $(N(0,\cdot),Z(0,\cdot))=(\bar N(0,\cdot),\bar Y(0,\cdot))=(N^0,Z^0)$. If
\[|\varphi_N(t,x)|+\varphi_Z(t,x)|\leq \frac {\bar C}{\gamma^\theta}+\bar C 1_{[0,1/\gamma^\theta]}(t),\]
for some $\bar C>0$ and $\gamma>1$. For any $T>0$, there is $\hat C>0$ independent from $\gamma>0$ such that
\begin{equation}\label{est:NNYY}
\|(N-\bar N)(t,\cdot)\|_{L^\infty([0,T]\times \mathbb T^d)}+\|(Z-\bar Z)(t,\cdot)\|_{L^\infty([0,T]\times \mathbb T^d)}\leq \frac {\hat C}{\gamma^\theta}.
\end{equation}

\end{prop}
This proposition shows the uniqueness of solutions of (KBM) under the assumptions above. It also shows the convergence of solutions $(N,Z)$ of \eqref{eq:thm} to the solution $(\bar N,\bar Z)$ of (KBM) when $\gamma\to\infty$, see \eqref{eq:cv}.

\begin{proof}[Proof of Proposition~\ref{prop:uniqNZ}]

Since $N,Z,\bar N, \bar Z\in W_{d+2}^{1,2}$, it is possible to use the comparison principle to show explicit upper bounds on $|Z|$, $|\bar Z|$ and $N$ that are uniform in $\gamma>1$. We can use these to prove a lower bound on $N$ (see \eqref{est:loNbarN}) that is also uniform in $\gamma>1$. 

\medskip

We can define $Y=ZN$ and $\bar Y=\bar Z\bar N$ and write the equations satisfied by $Y$, $\bar Y$, see \eqref{eq:NA}, \eqref{eq:barNY}. Note that in these equation, the non-liear terms do not involve any derivative of the functions.  Since $N,Y,\bar N,\bar Y$ are bounded functions and $N,\bar N$ are bounded from below independently from $\gamma>1$ on $[0,T]\times\mathbb T^d$, we can then estimate 
\begin{equation}\label{est:NbarN}
\begin{array}{l}
\partial_t (N-\bar N)(t,x)-\Delta_x (N-\bar N)(t,x)= \mathcal O(1)(N-\bar N)(t,x)+\mathcal O(1)(Y-\bar Y)(t,x)+\mathcal O(1)\varphi_N(t,x),\\
\\
\partial_t (Y-\bar Y)(t,x)-\Delta_x (Y-\bar Y)(t,x)= \mathcal O(1)(N-\bar N)(t,x)+\mathcal O(1)(Y-\bar Y)(t,x)+\mathcal O(1)\varphi_N(t,x)\\
\phantom{\partial_t (Y-\bar Y)(t,x)-\Delta_x (Y-\bar Y)(t,x)= }+\mathcal O(1)\varphi_Z(t,x),
\end{array}
\end{equation}
where the notation $\mathcal O(1)$ denotes functions $[0,T]\times\mathbb R$ with an $L^\infty$ norm that is independent from $\gamma>0$. Let $C>0$ such that the coefficients on the right hand side of the system above are dominated by $C>0$, ie $|\mathcal O(1)|\leq C$. We consider the following differential equation:
\[y'(t)=2Cy+C\left(\frac {\bar C}{\gamma^\theta}+2\bar C 1_{[0,1/\gamma^\theta]}(t)\right),\]
together with $y(0)=\varepsilon>0$. The solution to this ODE is 
\begin{align*}
y(t)&=\varepsilon e^{2Ct}+Ce^{2Ct}\int_0^te^{-2Cs}\left(\frac {\bar C}{\gamma^\theta}+2\bar C 1_{[0,1/\gamma^\theta]}(t)\right)\,ds\leq\varepsilon e^{2Ct}+ \frac{\bar Ce^{2Ct}}{2\gamma^\theta}+2\frac {C\bar C}{\gamma^\theta}e^{2Ct}\\
&\leq \varepsilon e^{2Ct}+\frac{\hat C}{\gamma^\theta}e^{2Ct},
\end{align*}
for $\gamma>1$ and $\hat C:=\bar C/2+2C\bar C$. We can now define $\hat N(t,x):=y(t)$ and $\hat Z(t,x):= y(t)$, that satisfy
\begin{align}\label{est:hatN}
\partial_t \hat N(t,x)-\Delta_x \hat N(t,x)= C\hat N+C\hat Y+C\left(\frac {\bar C}{\gamma^\theta}+2\bar C 1_{[0,1/\gamma^\theta]}(t)\right)> C\hat N(t,x)+C\hat Y(t,x)+C\varphi_N(t,x),
\end{align}
as well as $\hat N(0,x)>(N-\bar N)(0,x)=0$ for $x\in\mathbb T^d$. Similarly, $-\hat N$ satisfies $-\hat N(0,x)<(N-\bar N)(0,x)=0$ and
\begin{align}\label{est:hatN2}
\partial_t (-\hat N)(t,x)-\Delta_x (-\hat N)(t,x)< C(-\hat N)(t,x)+C(-\hat Y)(t,x)-C\varphi_N(t,x).
\end{align}
The same estimates can be obtained on $\hat Z$, replacing $\varphi_N$ by $\varphi_Z$. We now use a contradiction argument. At $t=0$,we have $|(N-\bar N)(0,x)|=0<\hat N(0,x)$ and  $|(Z-\bar Z)(0,x)|=0<\hat Z(0,x)$ for $x\in\mathbb T^d$. Since $N,Y,\bar N,\bar Y$ are $C^1$ in $t$ and $C^2$ in $x$, the first time when $\max\left(|(N-\bar N)(t,x)|,|(Z-\bar Z)(t,x)|\right)=y(t)$ leads to a contradiction since the inequalities in \eqref{est:hatN} and \eqref{est:hatN2} (and the similar equations on $\hat Z$, $-\hat Z$) are strict while $(N-\bar N)$, $(Z-\bar Z)$ satisfy \eqref{est:NbarN}. 
Thanks to this contradiction argument, $\max\left(|(N-\bar N)(t,x)|,|(Z-\bar Z)(t,x)|\right)\leq y(t)$ for $t\geq 0$. since this holds for any $\varepsilon>0$, we have \eqref{est:NNYY} if we redefine $\hat C$ as $\hat Ce^{2CT}$.

\end{proof}}

\section*{Acknowledgement}
The author is very grateful to Eric A. Carlen and Maria C. Carvalho for instructive  discussions and advice on hydrodynamic limits methods. He acknowledges support from the ANR under grants MODEVOL: ANR-
13-JS01-0009, MECC: ANR-13-ADAP-0006, Kibord: ANR-13-BS01-0004, DEEV
ANR-20-CE40-0011-01; it was also funded by the European Union (ERC-Adg SINGER, 101054787).

\signgr


\begin{thebibliography}{100}

\bibitem{Agueh} M. Agueh, Local existence of weak solutions to kinetic models of granular media. \emph{Arch. Rational Mech. Anal.} {\bf 221}(2), 917--959 (2016).

\bibitem{Aguilee}  R. Aguilee, G. Raoul, F. Rousset, O. Ronce, Pollen dispersal slows geographical range shift and accelerates ecological niche shift under climate change. \emph{Proc. Natl. Acad. Sci. U.S.A.} {\bf 113}(39), 5741--5748 (2016).

\bibitem{Alfaro1} M. Alfaro, J. Coville and G. Raoul, Travelling waves in a nonlocal reaction-diffusion equation as a model for a population structured by a space variable and a phenotypical trait. \emph{Comm. Partial Differential Equations} {\bf 38}, 2126--2154 (2013).


\bibitem{Aronson} D.G. Aronson, J. Serrin, Local behavior of solutions of quasilinear parabolic equations. \emph{Arch. Rational Mech. Anal.} {\bf 25}(2), 81--122 (1967).



\bibitem{Barton} N. Barton, A. Etheridge, A. V\'eber, The infinitesimal model, to appear in \emph{Theor. Pop. Biol.} (2017).

\bibitem{Bassetti} F. Bassetti, L. Ladelli, G. Toscani, Kinetic models with randomly perturbed binary collisions. \emph{J. Stat. Phys.} {\bf 142}, 686--709 (2011).

\bibitem{Berestycki} H. Berestycki, T. Jin, L. Silvestre, Propagation in a non local reaction diffusion equation with spatial and genetic trait structure. \emph{Nonlinearity} {\bf 29}(4), 1434--1466 (2016).

\bibitem{BerestyckiN} N. Berestycki, C. Mouhot, and G. Raoul. Existence of self-accelerating fronts for a non-local reaction-diffusion equations. http://arxiv.org/abs/1512.00903.

\bibitem{Bolley} F. Bolley, J.A.J. Carrillo, Tanaka theorem for inelastic Maxwell models. \emph{Commun. Math. Phys.} {\bf 276}, 287--314 (2007).

\bibitem{Bouin1} E. Bouin, S. Mirrahimi, A Hamilton-Jacobi limit for a model of population stuctured by space and trait. \emph{Commun. Math. Sci.} {\bf 13}(6), 1431--1452 (2015).

\bibitem{Bouin2} E. Bouin, V. Calvez, Travelling waves for the cane toads equation with bounded traits. \emph{Nonlinearity} {\bf 27}(9), 2233--2253 (2014).

\bibitem{Bouin3} E. Bouin, C. Henderson, L. Ryzhik, Super-linear spreading in local and non-local cane toads equations. accepted in  \emph{J. Math. Pures Appl.}.

\bibitem{Bouin4} E. Bouin, V. Calvez, N. Meunier, S. Mirrahimi, B. Perthame, G. Raoul, R. Voituriez, Invasion fronts with variable motility: phenotype selection, spatial sorting and wave acceleration.
\emph{C. R. Math. Acad. Sci. Paris} {\bf 350}(15-16), 761--766 (2012).

\bibitem{Bardos} C. Bardos, F. Golse, and C.D. Levermore, Fluid dynamic limits of kinetic equations II convergence proofs for the Boltzmann equation. \emph{Comm. Pure Appl. Math.} {\bf 46}.5, 667--753 (1993).

\bibitem{Bulmer} M.G. Bulmer, The mathematical theory of quantitative genetics. Clarendon Press, 1980.

\bibitem{Bridle} J.R. Bridle, T.H. Vines, Limits to evolution at range margins: when and why does adaptation fail? \emph{Trends Ecol. Evol.} {\bf 22}(3), 140--147 (2007).

\bibitem{Caflisch}  R.E. Caflisch. The fluid dynamic limit of the nonlinear Boltzmann equation. \emph{Comm. Pure Appl. Math.}, {\bf 33}(5), 651--666 (1980).

\bibitem{Caflish} R.E. Caflisch, B. Nicolaenko, Shock profile solutions of the Boltzmann equation. \emph{Commun. Math. Phys.} {\bf 86}, 161--194 (1982).


\bibitem{Calvez1} V. Calvez, J. Garnier, F. Patout, Asymptotic analysis of a quantitative genetics model with nonlinear integral operator. \emph{ J. Ec. Polytech.} {\bf 6}, 537--579 (2019).

\bibitem{Calvez4} V. Calvez, J. Garnier, F. Patout,  A quantitative genetics model with sexual mode of reproduction in the regime of small variance. \emph{arXiv:1811.01779} (2019).

\bibitem{Calvez3} V. Calvez, T. Lepoutre, D. Poyato,  Ergodicity of the Fisher infinitesimal model with quadratic selection. Nonlinear Analysis {\bf 238}, p.113392 (2024).

\bibitem{Calvez2} V. Calvez, D. Poyato, F. Santambrogio,  Uniform contractivity of the Fisher infinitesimal model with strongly convex selection. \emph{arXiv:2302.12063} (2023).

\bibitem{Calvez5} V. Calvez, J. Crevat, L. Dekens, B. Fabr\`eges, F. Kuczma, F. Lavigne, G. Raoul, Influence of the mode of reproduction on dispersal evolution during species invasion. ESAIM: Proceedings and Surveys {\bf 67}, 120--134 (2020).

\bibitem{Carlen} E. Carlen and W. Gangbo, Solution of a model Boltzmann equation via steepest descent in the 2-Wasserstein metric. \emph{Arch. Rational Mech. Anal.} {\bf 172}, 21--64 (2004).


\bibitem{Champagnat} N. Champagnat, S. M\'el\'eard, Invasion and adaptive evolution for individual-based
spatially structured populations. \emph{J. Math. Biol.} {\bf 55}, 147--188 (2007).

\bibitem{Dekens1} L. Dekens, F. Lavigne, Front propagation of a sexual population with evolution of dispersion: a formal analysis. \emph{SIAM J. Appl. Math.} {\bf 81}(4), 1441--1460 (2021).

\bibitem{Dekens2} L. Dekens, S. Mirrahimi, Dynamics of dirac concentrations in the evolution of quantitative alleles with sexual reproduction. \emph{Nonlinearity} {\bf 35}(11), p.5781 (2022).


\bibitem{Dragomir} Dragomir, Sever S. Some Gronwall type inequalities and applications. Nova Science, 2003.

\bibitem{Evans} L.C. Evans, Partial Differential Equations. Graduate studies in mathematics, Graduate studies in mathematics, Vol 19, American Mathematical Society, Providence, RI, second edition, 2010.

\bibitem{Fisher} R. Fisher, The correlations between relatives on the supposition of Mendelian inheritance. \emph{Trans. R. Soc. Edin.} {\bf 52}, 399--433 (1919).

\bibitem{Fisher2} R.A. Fisher, The genetical theory of natural selection: a complete variorum edition. Oxford University Press, 1930.

\bibitem{FKPP1} R. A. Fisher, { The wave of advance of advantageous genes}, \emph{Ann. of Eugenics} {\bf 7} (1937), 355--369.

\bibitem{Griette} Q. Griette, G. Raoul, Existence and qualitative properties of travelling waves for an epidemiological model with mutations. \emph{J. Differential Equations} {\bf 260}(10), 7115--7151 (2016).

\bibitem{Guerand} J. Guerand, M. Hillairet, S. Mirrahimi, A moment-based approach for the analysis of the infinitesimal model in the regime of small variance. \emph{arXiv:2309.09567} (2023).

\bibitem{Guisan} A. Guisan, N. E. Zimmermann, Predictive habitat distribution models in ecology. \emph{Ecol. Model.} \textbf{135}, 147--186 (2000).

\bibitem{KPP} A.N. Kolmogorov, I.G. Petrovsky, N.S. Piskunov, Etude de l'\'equation de la diffusion avec
croissance de la quantit\'e de mati\`ere et son application \`a un probl\`eme biologique, \emph{Bull. Univ.
Etat Moscou (Bjul. Moskowskogo Gos. Univ.)}, Sr. Inter. A 1, 1--26 (1937).

\bibitem{Liebermann} G.M. Lieberman. Second order parabolic differential equations. World scientific (1996).

\bibitem{Moser} J. Moser. A Harnack inequality for parabolic differential equations. \emph{Comm. Pure Appl. Math.} {\bf 17}(1), 101--134 (1964).



\bibitem{Martinez} V. Martinez, L Bunger, W.G. Hill, Analysis of response to 20 generations of selection for body composition in mice: fit to infinitesimal model assumptions. \emph{Genet. Select. Evol.} {\bf 32}, 3--22 (2000).

\bibitem{Miller} J.R. Miller, H. Zeng, Range limits in spatially explicit models of quantitative traits. \emph{J. math. biol.} {\bf 68}1-2), 207--234 (2014).

\bibitem{Mirrahimi} S. Mirrahimi, G. Raoul, Population structured by a space variable and a phenotypical trait. \emph{Theor. Popul. Biol.} {\bf 84}, 87--103 (2013).

\bibitem{Mischler} S. Mischler, C. Mouhot, Stability, convergence to the steady state and elastic limit for the Boltzmann equation for diffusively excited granular media. \emph{Discrete and Contin. Dyn. S.} {\bf 24}, 159--185 (2009).

\bibitem{Patout} F. Patout. The cauchy problem for the infinitesimal model in the regime of small variance. \emph{Analysis \& PDE}, {\bf 16}(6), 1289--1350 (2023).

\bibitem{Pease} C.M. Pease, R. Lande, J. Bull, A model of population growth, dispersal and evolution in a changing environment. \emph{Ecology} {\bf 70}(6), 1657--1664 (1989).


\bibitem{Prevost} C. Prevost, Applications of partial differential equations and their numerical simulations of population dynamics. PhD Thesis, University of Orleans (2004).

\bibitem{Kirkpatrick} M. Kirkpatrick, N.H. Barton, Evolution of a species' range. \emph{Amer. Nat.} {\bf 150} (1), 1--23 (1997).

\bibitem{Magal} P. Magal, G. Raoul, Dynamics of a kinetic model for exchanges between cells. \emph{ArXiv:1511.02665}.

\bibitem{Miller2}J. R. Miller, Invasion waves and pinning in the Kirkpatrick–Barton model of evolutionary range dynamics. \emph{J. Math. Biol.}, {\bf 78}, 257--292 (2019).

 \bibitem{Norberg}J. Norberg, M. Urban, M. Vellend, C. Klausmeier, N. Loeuille, Eco-evolutionary responses of biodiversity to climate change. \emph{Nat. Clim. Change}, {\bf 2}, 747--751 (2012).

\bibitem{Liu} T.P. Liu, S.H. Yu, Boltzmann equation: micro-macro decompositions and positivity of shock profiles. \emph{Commun. Math. Phys.} {\bf 246}, 133--179 (2004).

\bibitem{Patout} F. Patout,  The cauchy problem for the infinitesimal model in the regime of small variance. \emph{Anal. PDE} {\bf 16}(6), 1289--1350 (2023).

\bibitem{Raoul} G. Raoul, Exponential convergence to a steady-state for a population genetics model with sexual reproduction and selection. \emph{arXiv:2104.06089} (2021).


\bibitem{Taing} C. Taing, A. Frouvelle,  On the Fisher infinitesimal model without variability. \emph{arXiv:2307.12735} (2023).

\bibitem{Tanaka} H. Tanaka, Probabilistic treatment of the Boltzmann equation of Maxwellian molecules. \emph{Probab. Theory Related Fields} {\bf 46}, 67--105 (1978).


\bibitem{Turanova} O. Turanova, On a model of a population with variable motility. \emph{Math. Models Methods Appl. Sci.} {\bf 25}(10), 1961--2014 (2015).

\bibitem{Turelli} M. Turelli, N.H. Barton, Genetic and statistical analyses of strong selection on polygenic traits: what, me normal?. \emph{Genetics} {\bf 138}, 913--941 (1994).


\bibitem{Verrier} E. Verrier, J. Colleau, J. Foulley, Methods for predicting response to selection in small populations under additive genetic models: a review. \emph{Livest. Prod. Sci.} {\bf 29}, 93--114 (1991).

\bibitem{Villani2} C. Villani, Mathematics of granular materials. \emph{J. Stat. Phys.} {\bf 124}, 781--822 (2006).

\bibitem{Villani} C. Villani, Optimal transport: old and new. Grundlehren Der Mathematischen Wissenschaften Vol. 338, Springer Science \& Business Media, 2008.

\bibitem{VillaniReview}C. Villani. A review of mathematical topics in collisional kinetic theory. Handbook of mathematical fluid dynamics 1, 3-8, 71--305 (2002).


\end{thebibliography}
\end{document}